\documentclass[11pt]{article}
\pdfoutput=1
\usepackage[english,ruled,vlined]{algorithm2e}
\SetKwInput{KwInput}{Inputs}
\SetKwInput{KwReturn}{Return}
\SetKwInput{KwRequire}{Require}
\usepackage[noend]{algpseudocode}

\usepackage{pgfplots}
\pgfplotsset{compat=newest}
\usetikzlibrary{shapes}
\usetikzlibrary{plotmarks}
\usepackage{footmisc}
\usepackage{multirow}
\usepackage{listings}
\usepackage{booktabs}
\usepackage{subcaption}

\usepackage{amssymb}
\usepackage{amsfonts}
\usepackage{amsmath}
\usepackage{textcomp}
\usepackage{amsthm}
\usepackage{mathbbol}
\usepackage{stmaryrd}
\usepackage{mathrsfs}
\usepackage{graphicx}
\usepackage{amsbsy}
\usepackage{xcolor}
\usepackage{hyperref}
\usepackage{tikz}
\usetikzlibrary{arrows,shapes}
\usetikzlibrary{shapes.geometric}
\usetikzlibrary{shapes.arrows}

\usepackage[normalem]{ulem}
\usepackage{float}
\usepackage{dsfont}

\usepackage{empheq}
\usepackage{cleveref}
\usepackage[toc,title,page]{appendix}

\newtheorem{theorem}{Theorem}[section]
\newtheorem{lemma}[theorem]{Lemma}

\newtheorem{remark}[theorem]{Remark}

\def\thesection{\arabic{section}}

\def\NN{\mathbb{N}}

\def\RR{\mathbb{R}}

\def\ZZ{\mathbb{Z}}

\definecolor{mygreen}{rgb}{0,0.7,0}

\definecolor{mygreen}{rgb}{0,0.4,0}

\begin{document}

\title{\sf  Reduced modelling and optimal control of epidemiological individual-based models with contact heterogeneity}
\author{C. Court\`es\footnote{IRMA, Universit\'e de Strasbourg, CNRS UMR 7501, Inria, 7 rue Ren\'e Descartes, 67084 Strasbourg, France ({\tt clementine.courtes@unistra.fr}).}
\and E. Franck\footnote{Inria, IRMA, Universit\'e de Strasbourg, CNRS UMR 7501, 7 rue Ren\'e Descartes, 67084 Strasbourg, France ({\tt emmanuel.franck@unistra.fr}).}
\and K. Lutz\footnote{Univ Lyon, Ecole Centrale de Lyon, CNRS UMR 5208, Institut Camille Jordan, F-69134 Ecully, France, ({\tt killian.lutz@ecl19.ec-lyon.fr}).}
\and L. Navoret\footnote{IRMA, Universit\'e de Strasbourg, CNRS UMR 7501, Inria, 7 rue Ren\'e Descartes, 67084 Strasbourg, France ({\tt laurent.navoret@unistra.fr}).}
\and Y. Privat\footnote{IRMA, Universit\'e de Strasbourg, CNRS UMR 7501, Inria, 7 rue Ren\'e Descartes, 67084 Strasbourg, France ({\tt yannick.privat@unistra.fr}).}~~\footnote{Institut Universitaire de France (IUF).}
}

\maketitle

\begin{abstract}
Modelling epidemics via classical population-based models suffers from shortcomings that so-called individual-based models are able to overcome, as they are able to take heterogeneity features into account, such as super-spreaders, and describe the dynamics involved in small clusters. In return, such models often involve large graphs which are expensive to simulate and difficult to optimize, both in theory and in practice.

By combining the reinforcement learning philosophy with reduced models, we propose a numerical approach to determine optimal health policies for a stochastic epidemiological graph-model taking into account super-spreaders. More precisely, we introduce a deterministic reduced population-based model involving a neural network, and use it to derive optimal health policies through an optimal control approach. It is meant to faithfully mimic the local dynamics of the original, more complex, graph-model. Roughly speaking, this is achieved by sequentially training the network until an optimal control strategy for the corresponding reduced model manages to equally well contain the epidemic when simulated on the graph-model.

After describing the practical implementation of this approach, we will discuss the range of applicability of the reduced model and to what extent the estimated control strategies could provide useful qualitative information to health authorities.

\end{abstract}






\section{Introduction}
\subsection{Population versus Individual-based models}

Modelling the spread of epidemics correctly is of paramount importance for defining health policy to control their development. Models allow to estimate the so-called basic reproduction ratio $\mathcal{R}_{0}$ that indicates whether the epidemic is growing or not and then propose optimal policies to limit the saturation of hospital services for instance. In such a program, one difficulty is to take into account the presence of the so-called {\it super-spreaders}. This refers to any individual who is likely to infect many more people than a generic person: indeed, the distribution of the number of contacts in the population is very heterogeneous. In the seminal work \cite{lloyd-smithSuperspreadingEffectIndividual2005}, the authors show that epidemics tend to be {\it rarer but more explosive} with super-spreaders. Indeed, super-spreaders tend to be infected in the early stages of an epidemic. 

To tackle this problem, several levels of descriptions can be considered \cite{Kiss_2017}. The two extreme types of models are population-based and individual-based models. Population-based or mean-field models, like the SIR one, are the simplest descriptions of epidemics: they describe the time evolution of the total number of susceptible (S), infected (I), retired (R) people or other categories. At the opposite level of description, \emph{individual-based models} are the most accurate: they describe the stochastic temporal evolution of the status (susceptible, infected, retired or other) of each individual by taking into account the contact graph between individuals. 

Individual-based models  are much better suited to describe the effect of {\it super-spreaders}. Indeed, they are based on an accurate description of the contact graph between individuals: each node corresponds to an individual and each edge to the contact between two individuals. It is therefore easy to include contact heterogeneities  by considering contact graphs with prescribed distributions of node degrees \cite{rafoSimpleEpidemicNetwork2020, kimAgentBasedModelingSuperSpreading2018}. The individual-based models are then constructed as a continuous-time Markov process, in which each individual can evolve between the different status (susceptible, infected, retired) at random times. Thus the infection of one susceptible individual depends on the number of connected infected neighbours as well as the individual transmission rate $\beta_{\rm ind}$ while the transition to the retired state depends only on the recovery rate $\gamma$. 

Individual-based models allow for more accurate modelling, but require more computational resources to simulate large or complex contact graphs \cite{anOptimizationControlAgentBased2017}. It is notable that models corresponding to intermediate descriptions such as branching or percolation processes, where the state process is simplified, have been proposed. For instance, in  \cite{lloyd-smithSuperspreadingEffectIndividual2005, garskeEffectSuperspreadingEpidemic2008}, the authors used this kind of models to describe the number of secondary infections. In some cases, population-based models can also be derived analytically from individual-based models, which is of interest for simulations. First, assuming some independence between the statuses of each individual, the individual-based stochastic models can be approximated by a deterministic model, where the status of each individual follows SIR differential systems coupled to those of other individuals \cite{Kiss_2017}. Then, making the additional assumption that the degree distribution has a small variance, this model can be further simplified and we recover the classical population based SIR models. Without this assumption of small variance, SIR models structured by contact numbers can be derived leading to larger models. 

Thus, to incorporate the effect of super-spreaders in population-based models, a possible approach is to extend the number of categories: for each health status, we can consider several compartments associated with sub-populations having different amount of contacts \cite{kissEffectContactHeterogeneity2006}. Another proposed strategy is to consider a single additional compartment with specific epidemiological properties at the population-level (e.g. recovery rates, transmission rates, etc.) \cite{mkhatshwaModelingSuperspreadingEvents2010}. Other studies have proposed to take super-spreaders into account through a spatial description of the epidemic \cite{fujieEffectsSuperspreadersSpread2007}. 

The infection probability then depends on the distance between two individuals, and the difference between normal individuals and super-spreaders is taken into account via this dependence. In another direction, recurrent neural networks have been proposed to improve population-based models \cite{bhouriCOVID19DynamicsUS2021}:  their long short-term memory is used to identify the transmission rate as a function of the observed mobility and social behaviour trends.

Despite the advantages of population-based models in terms of simplicity, only individual-based models are really able to handle heterogeneous contact distributions and describe epidemics with few individuals and stochastic effects.
\subsection{Optimal control issues}
\textit{In this work, we are interested in proposing a possible strategy to define an optimal control for an individual-based model. The optimal control problem we consider is to keep the number of infected individuals below a threshold by adjusting the average transmission rate and the parameters of the contact distribution over time.} 

Specifically, we consider an individual-based model where contacts follow a negative binomial distribution. Such a distribution is used to model populations with super-spreaders \cite{pmid:16292310} and is parametrized by its mean value $\alpha>0$ and the so-called \emph{dispersion coefficient} $\kappa>0$. Misleadingly, since the variance in the contacts distribution is $\alpha+\alpha^2/\kappa$, a \emph{low value} of the dispersion coefficient $\kappa$ is associated with a \emph{high variance} in the distribution of contacts and thus with the presence of super-spreaders. The dispersion control acts mainly on the super-spreaders while the control of the average transmission rate $\beta = \alpha \beta_{\rm ind} $ is a uniform control on the population. We are therefore interested in defining an optimal pair $(\beta(\cdot),\kappa(\cdot)) = (b(\cdot)\beta_{0},k(\cdot)\kappa_{0})$, with the smallest deviation from the initial values $(\beta_{0},\kappa_{0})$, in order to keep the number of infected people below a certain threshold.

To define a control of the stochastic individual-based model, the classical methods use dynamic programming algorithms.  Indeed, the problem can be formulated as a Markov Decision Process (MDP) \cite{sutton2018reinforcement}, with given probability transitions between the $3^N$ states of the SIR model, where $N$ denotes the number of individuals. For large $N$, it is no more possible to easily deal with such a model completely and a reinforcement learning approach should be used. In  \cite{10.1613/jair.1.12632}, a Deep-Q reinforcement learning algorithm is used for large graphs using a global control on agents (partially observable MDP). For such global control problem, cooperative multi-agent approaches can also be used \cite{DBLP:journals/corr/abs-2004-12959}. Another approach is to consider reinforcement learning based on a reduced model. As proposed in \cite{anOptimizationControlAgentBased2017}, it may be worthwhile to build on simpler models, like population-based ones, for which the standard control theory framework applies. The optimal control strategy for the reduced model is then used as the starting point for designing controls for the individual-based model.  

Concerning the control of population-based model, references are plentiful: they are generally based on the use of optimality conditions such as the Pontryagin Maximum Principle (PMP). If few of them propose an analytical design of the controls (see for instance \cite{MR4198237,MR4255688}), many introduce adapted optimization algorithms based either on a discretization of the complete problem (discretize then optimize, see e.g. \cite{MR4363007}) or an algorithm on the continuous problem applied on a discretized version of the model (optimize then discretize, see e.g. \cite{MR4198237,MR4255688}). 

\subsection{Organization of the article}

In this work, the proposed strategy for designing a control of the individual-based model is decomposed into the three following steps:
\begin{enumerate}
\item First, learn a reduced population-based SIR model via data coming from numerical simulations of the individual-based model using neural networks. 
\item Then, define a control of the parameters of the population-based models.
\item Finally, use a reinforcement algorithm to improve the population-based model around the controlled solution and thus the control itself.
\end{enumerate} 
The data-driven population-based SIR model is intended to capture the effect of the contact distribution heterogeneity and stochastic effects due to relatively small population size. The model thus depends on the \emph{dispersion parameter} $\kappa$, the \emph{transmission rate} $\beta$ and the \emph{relative size of the population} $n$, a coefficient depending on $N$ describing the closeness to the large population regime.

The neural network is trained to compute the time variation of the number susceptible people following an ordinary differential equation of the form: $S' = -F_\theta(S,I;n,\beta,\kappa)$, where $\theta$ denotes a vector of parameters. Note that throughout this study, the recovery rate $\gamma$ is fixed, equal to $1/6$ day$^{-1}$. Then, the optimal control of the data-driven population-based SIR model is defined by means of an optimal control algorithm. This will define time-varying parameters $(\beta(t),\kappa(t))$. However, since the learned population-based model is not a priori trained with such time varying parameters, the latter control is not well adapted to the underlying individual-based model. This is why the reinforcement strategy is essential to obtain a meaningful control with respect to the individual-based model. 

The outline of the article is the following. In the first section, we present how the data-driven SIR model is constructed from the data using a classical multi-perceptron neural network. We show that this model is interesting on its own to evaluate epidemiological quantities, especially in parameter regimes not covered by the classical SIR model. The next section is then devoted to the optimal control method. A theoretical analysis is performed to show that the control of the data-driven SIR model is well defined. Then the reinforcement strategy is detailed. Finally, several numerical tests are performed to assess the validity of the whole methodology. Some appendices conclude this paper by detailing technical points on the individual-based model (Appendix \ref{apdx:mean_traj_calculation}), the reduced data-driven population-based model (Appendix \ref{apdx:rigorous_R0_derivation}), some rigorous proofs of the control section (Appendix \ref{apdx:well_posed_controlled_system}) and numerical algorithms (Appendix \ref{apdx:numerical_implementation}).

\section{From an individual-based to a data-driven population-based SIR model}

In this section, we first introduce the individual-based SIR model, hereafter denoted (IBM), which enables to model epidemic dynamics with contact heterogeneity. Then we present our data-driven approach to learn the reduced population-based dynamics of the total number of safe, infected and recovered people.

\subsection{Individual-based SIR model with contact heterogeneity (IBM)}
\label{sec:IBM}

The individual-based SIR model consists in a graph with $N$ vertices. Each vertex represents one individual, whose epidemic state over time is denoted: $X_{j}(t) \in \{ s, i, r\}$ for $t \geqslant 0$, for susceptible ($s$), infected ($i$) and retired ($r$). The edges of the graph represent the contacts between individuals: the number of contacts of the $j$-th individual is denoted $\nu_{j} \in \mathbb{N}$. Then the dynamics is described by a continuous-time Markov process, whose two main parameters are the individual transmission rate $\beta_{\rm ind} > 0$ and the recovery rate $\gamma > 0$. 

One individual evolves in time from the susceptible  to the infected state ($s\to i$) and then from the infected state to the recovered one  ($i\to r$). The change of the individual states in the graph occurs one by one at random times $(T^{m})_{m \in \mathbb{N}}$. Given the graph state at time $T^{m}$, the next time $T^{m+1}$ and the associated transition is defined as follows. We assign random clocks $C_{j}$ to any states and these clocks follow exponential distributions. If the $j$-th individual is infected then the exponential distribution has rate $\gamma$. If the $j$-th individual is susceptible, then the exponential distribution has rate $\beta d_{j}$ where  $d_{j}$ is the number of its infectious contacts and $\beta=\alpha\beta_{\rm ind}$ is the mean transmission rate. Then the next transition occurs for the $j^{\ast}$-th individual  at time $T^{m+1} = T^{m}+ C_{j^{\ast}}$, where $j^{\ast}$ corresponds to the smaller clocking time: $C_{j^{\ast}} = \min_{j} C_{j}$. With such dynamics, the more contacts one individual has with infected neighbors, the more likely he is to be infected in turn.

To investigate the role of super-spreaders in the dynamics, we consider a heterogeneous distribution of contacts. Thus, the edges of the graph are distributed such that the number of contacts $\nu_{j}$ of the $j$-th individual follow a (generalized) negative binomial also named P\'olya distribution\footnote{The probability distribution of the P\'olya distribution writes: $P(\nu_{j} = k) =  \frac{\Gamma(\alpha+k)}{k! \Gamma(\alpha)} (1-p)^{k} p^{\alpha} $ for all $k \in \mathbb{N}$, with $p = \frac{\alpha}{\alpha+\kappa}$.}:
\begin{equation*}
\nu_{j} \sim \mathcal{BN}\left(\kappa, \frac{\kappa}{\alpha+\kappa}\right),
\end{equation*}
where $\alpha > 0$ denotes the average number of contacts and $\kappa > 0$ is the dispersion parameter. The mean of the distribution is $\alpha$ and the variance equals $\alpha + \alpha^{2}/\kappa = \alpha (1+ \alpha/\kappa)$. Thus, a \textbf{small dispersion coefficient $\kappa$} corresponds to a large variance and so to the \textbf{existence of super-spreaders} (see Figure~\ref{fig:Loi_Poisson_Neg_bino}). For large $\kappa$, the P\'olya distribution converges (in law) towards the Poisson one used classically to model an homogeneous population (see Figure~\ref{fig:Loi_Poisson_Neg_bino}). With this distribution, the heterogeneity of the contacts relatively to the average number of contacts is parametrized by $\alpha/\kappa$. In practice, in order to fit this distribution, the edges are constructed thanks to the Molloy-Reed algorithm \cite{MolloyReed1998}.

The possible controls of this model can either be on the individual transmission rate $\beta_{\rm ind}$ (by imposing masks, for instance) or on the contact distribution parameters $(\alpha, \beta)$ (with confinements or restaurant closures). However, the control of $\beta_{\rm ind}$ and $\alpha$ are quite similar as they both modify the mean transmission rate $\beta = \alpha \beta_{\rm ind}$, which is the key parameter for epidemics developments. Thus the two main parameters that we are aiming to control are:
\begin{enumerate}
\item the mean transmission rate $\beta = \alpha \beta_{\rm ind}$,
\item the dispersion coefficient $\kappa$,
\end{enumerate}
while the recovery rate $\gamma$ and the population size $N$ are two given quantities.

The model is one of the simplest model for epidemic dynamics on graphs. It is simulated by using a Gillepsie algorithm \cite{GILLESPIE1976403}. We refer to \cite{Kiss_2017} for more details. We performed numerical simulations for this model by using the Python packages EpidemicsOnNetworks \cite{millerEoNEpidemicsNetworks2019} and NetworkX \cite{hagbergaricandswartpieterandschultdanielExploringNetworkStructure2008}. At time $t=0$, a given proportion of states are initialized randomly as infected, the other ones being considered susceptible.

\subsection{Average dynamics and stochastic delay} 

For modelling and control purposes, we are interested in deriving a population-based model which approximates the dynamics generated by the individual-based one. We are thus looking at the dynamics of the average number of susceptible, infected and recovered individuals in the graph:
\begin{equation*}
S(t) = \frac{1}{N}\sum_{j=1}^{N} P(X_{j}(t) = s),\quad I(t) = \frac{1}{N}\sum_{j=1}^{N} P(X_{j}(t) = i),\quad R(t) = \frac{1}{N}\sum_{j=1}^{N} P(X_{j}(t) = r).
\end{equation*}
To this end, we average several time series associated with the same set of parameters $(\beta, \gamma, \kappa, N)$.

In settings in which population size is particularly small or contact heterogeneity drives the epidemic, two main difficulties arise: (i) some simulations lead to \emph{immediate extinctions} whereas the others lead to outbreaks, (ii) \emph{randomness} in the time of the epidemic \emph{onset}. Consequently, as illustrated in Figure \ref{fig:drawbacks_naive_mean}, computing the average trajectory via a naive average frequently leads to severe under estimations of the total number of infected individuals $I(t)$, as already noted in \cite[Appendix~A.2]{Kiss_2017}.  One standard way to overcome this issue is to compute the standard pointwise average only after having time-translated the time series such that the outbreaks occur all at the same time. Details on the calculation of the average trajectories can be found in Appendix \ref{apdx:mean_traj_calculation}. Though it seems that this method \emph{most} of the time solves the above mentioned issues and gives robust results, as a safeguard individual stochastic trajectories will also be shown when plotting results. 

If we do not take into account this time shift, it is shown in \cite{Kiss_2017} that the averaged quantities solve the following differential equations: 
\begin{align*}
&S' = - \beta_{\rm ind} [SI],\\
&I' = \beta_{\rm ind} [SI] - \gamma I,\\
&R' =  \gamma I,
\end{align*}
where the quantity $[SI]$ denotes the average number of edges connecting an infected and a susceptible individual. Obtaining a closed system requires a relation between $[SI]$ and the variables $S, I$. For a homogeneous contact graph, the relation $[SI]= \alpha S I / N$ holds in the large population limit. However, as we are considering contact heterogeneity, this relation is no more valid. It is therefore necessary to establish a valid closure for time-shifted dynamics with contact heterogeneity.

\begin{figure}[ht!]
    \centering
    \includegraphics[scale=.5]{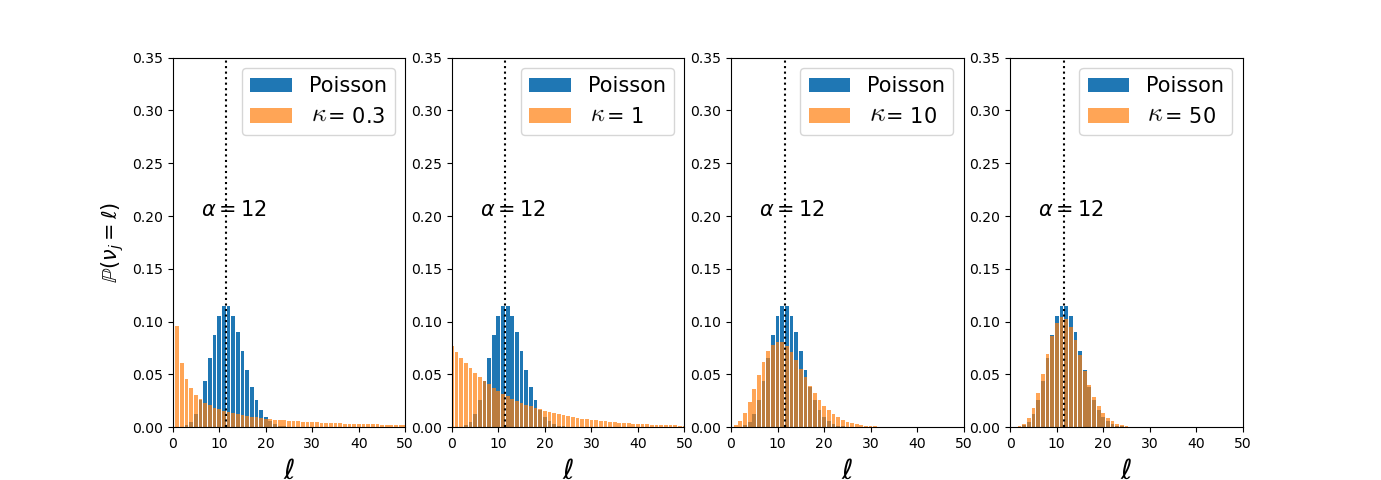}
    \caption[flushleft]{Convergence of the P\'olya distribution to the Poisson one with respect to parameter~$\kappa$.}
    \label{fig:Loi_Poisson_Neg_bino}
\end{figure}
\begin{figure}[h!]
    \centering
    \includegraphics[scale=.3]{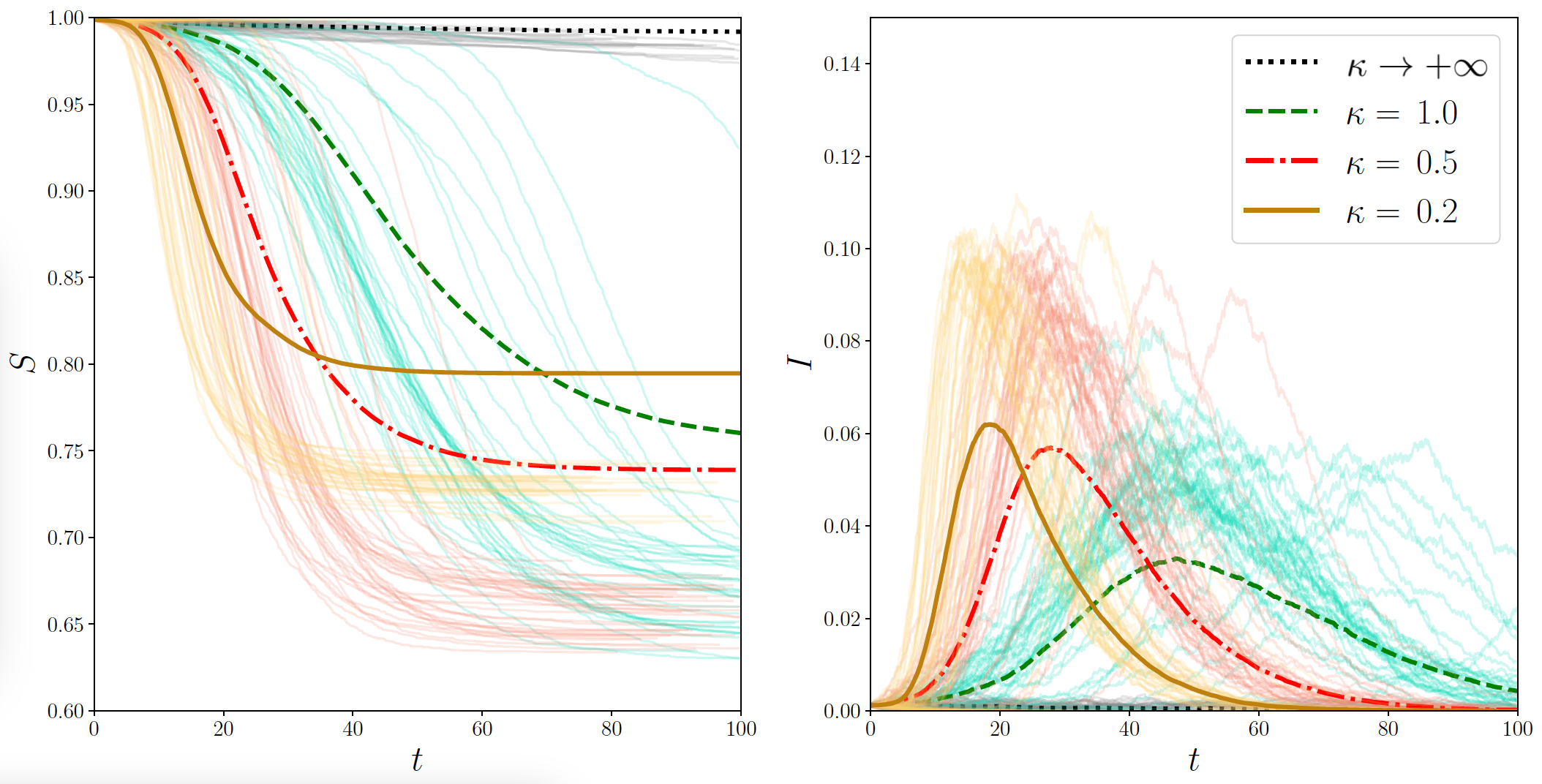}
    \caption[flushleft]{Example of simulations illustrating the insufficiency of a naive average: the delays in the start of the epidemic due to stochasticity lead to an underestimation of the punctual number of infected people. Also it illustrates the importance of modelling the dispersion since here, in the homogeneous case ($\kappa = +\infty$, corresponding to a Poisson distribution), there is immediate extinction. $(n,\beta, \gamma, i_0)=(.25,.15,.15,.001)$.  }
    \label{fig:drawbacks_naive_mean}
\end{figure}

\subsection{Data-driven population-based model \eqref{eq:learned_model}}
\label{sec:learning_global_model}

This subsection deals with the construction of  a \emph{deterministic} population-based model, which faithfully approximates the dynamics of the stochastic individual-based model (IBM) discussed in Section \ref{sec:IBM}. Using supervised machine learning techniques, the main goal is to capture, through an SIR-like system of autonomous differential equations, the effects of contact heterogeneity and population size on the dynamics of an epidemic.  More precisely, we assume the following relationships hold between the state variables $S,I,R$ and their time derivatives:
\begin{align}\tag{RM$_\theta$}
\begin{array}{c c c}
				S' &=&-F_\theta(S,I;n,\beta,\kappa),\\
				I'  &= & F_\theta(S,I;n,\beta,\kappa)-\gamma I,\\
				R' &=& \gamma I,
		\end{array}
		\label{eq:learned_model}
\end{align}
where $\gamma$ is still the individual recovery rate and the function $F_\theta:\RR ^5\to\RR$ is a parametrized \emph{incidence} function \cite{martchevaIntroductionMathematicalEpidemiology2015} whose purpose is to capture the infection process. Its inputs are assumed to be both the instantaneous proportion of susceptible $S$ and infected $I$ individuals, as well as three positive parameters: the population size ratio $n =\min \{1, N/N_{\max}\}$, the mean transmission rate $\beta$ and the dispersion coefficient $\kappa$. The introduction of the parameter $N_{\max}$, taken equal to $20,000$, allows us to extend our approach to large population sizes. From a practical point of view, we choose $N_{\max}$ empirically in such a way that the population dynamic does not vary anymore when increasing $N_{\max}$. In other words, $n$ is a ratio describing the closeness to the large population regime.

In order to ensure that the state variables remain in the interval $[0,1]$, it is further assumed that the incidence function writes:
\begin{equation}\label{Ftheta_part}
		F_\theta(S,I;n,\beta,\kappa) = f_\theta(S,I;n,\beta,\kappa)\, SI,
\end{equation}
where $f_\theta:\RR ^5\to\RR$ is another function, with the same inputs, called the \emph{transmission rate function}. 
For the sake of clarity, we postpone to Appendix~\ref{apdx:well_posed_controlled_system} details on 
regularity assumptions on both functions $f_\theta$ and $F_\theta$, as well as well-posedness issues (existence and uniqueness of an absolutely continuous global solution) of System~\eqref{eq:learned_model}.

From now on, we will assume that such regularity properties on $F_\theta$ are satisfied so that System~\eqref{eq:learned_model} has a unique solution which is moreover Lipschitz, with non-negative components. Note that because $S + I + R =1$ (remember that these quantities are proportions), all state variables of \eqref{eq:learned_model} are bounded from above by 1 and one of the three equations in \eqref{eq:learned_model} is redundant. Hence, the equation corresponding to recovered individuals is hereafter omitted. 

\paragraph{Neural network structure.}  The function $f_{\theta}$ is built by means of a fully-connected neural network (or multilayer perceptron), which shows a great ability to learn non-linear functions \cite{zhangDiveDeepLearning2021,goodfellow2016deep}. The function is then defined by composition of layers: each layer performs an affine transformation on its inputs and then applies a non-linear function (a so-called activation function) which is determined in advance.  All the coefficients involved in the affine transformations are thus parameters of the function  $f_{\theta}$ and correspond to the vector-valued parameter $\theta$.  The neural network structure is then determined by so-called hyperparameters, for instance, the number of layers, the input and output sizes of each layer and the activation functions used. The hyperparameters that we have selected are specified in Table~\ref{tab:keras_hyperparam} (left column).

\begin{table}[h]
    \centering
    \begin{tabular}{ll|ll}
        \toprule
        Hyperparameters        & Values                      & Learning parameters           & Values \\
        \midrule
        Neurons per layer      & 64 / 128 / 64 / 16        & Initial learning rate     & $10^{-3}$ \\ 
        No. inputs / outputs   & 5 / 1                       & Validation split          & 15\% \\
        Inputs normalization   & Centered and reduced        & Cost function             & Mean squared error \\
        Initialization         & Orthogonal                  & Optimizer                 & Adam \\
        Activations [Last]     & ReLu [Linear]               & Batch size                & 512 \\                       
        Learning rate schedule & Exponential                 & Epoch                     & 15 \\
        \bottomrule
    \end{tabular}
    \caption[flushleft]{Practical details regarding the learning of the transmission rate function $f_\theta$ using the open-source Python library Keras \cite{franccoischolletandothersKeras2015}.}
    \label{tab:keras_hyperparam}
\end{table}

\paragraph{Learning the transmission rate function from data.} The parameter $\theta$ is then set in such a way that the transmission rate function $f_{\theta}$ best approximates the rate observed on the individual-based model simulations.  More precisely, the parameters $\theta$ are found by regression, i.e. by minimizing the mean squared error:
\begin{equation*}
L(\theta) = \sum_{(\tilde S, S,I;n,\beta,\kappa) \in \mathcal{D}}\left\| f_{\theta}(S,I;n,\beta,\kappa) - \frac{\tilde S-S}{\Delta t\, S I}\right\|^{2},
\end{equation*}
where $\mathcal{D}$ denotes the data set composed of samples $(\tilde S, S,I;n,\beta,\kappa)$, where $S, I$ are the average number of the susceptible and infected populations at a given time $t$, $\tilde S$ the value at time $t+\Delta t$, obtained after averaging individual-based simulations with parameters $(n,\beta, \kappa)$. The hope is that, by considering a diverse enough data set, the function corresponding to an optimal parameter will manage to capture the  underlying trend which relates the inputs to the output, especially for input values lacking in the data set. Details about the practical implementation of the learning algorithm are displayed in Table~\ref{tab:keras_hyperparam} (right column).

To generate the data set, parameters $(n, \beta, \kappa,I(0))$ are chosen randomly according to the distributions given in Table \ref{tab:param_range}. The susceptible state at $t=0$ are then $S(0)=1-I(0)$. Then, we 
run the individual-based model over a given time interval. We then average the time-series of the corresponding susceptible and infected populations over 50 simulations at discrete times $t^{m} = m\Delta t$. For the simulations, we choose $\Delta t\simeq 0.28$.
Repeating this process a significant number of times and storing the results makes up the \emph{training data set} $\mathcal{D}$ which, in our case, contains about 7.4 millions of samples.

\begin{table}[h]
    \centering
    \begin{tabular}{ccccl}
        \toprule
        Parameters  & Lower bound  &  Upper bound & Units      & Interpretation            \\
        \midrule
        $n$         & 0.1           &  1           & -           & Population size ratio           \\
        $\beta$     & 0.075         & 0.9          & days$^{-1}$ & Transmission rate \\
        $\kappa$    & 0.1           &  10          & -           & Dispersion coefficient    \\
        $I(0)$ & $10^{-4}$ & $10^{-3}$ & - & Initial proportion of infected people \\
        \bottomrule
    \end{tabular}
    \caption[flushleft]{Samples used to learn the function $f_\theta$ have their input parameters randomly drawn (uniformly for $n,\beta$ and log-uniformly for $\kappa$ and $I(0)$) in a subset of their possible values.}
    \label{tab:param_range}
\end{table}

\subsection{Model validation} 
\label{sec:model_validation}
Once the parameters defining the transmission rate function $f_\theta$ have been determined, the validation step is carried out to prevent over-fitting and evaluate the model accuracy. 
To do so, we select unseen values of $(n,\beta,\kappa)$ in the ranges of interest (see Table \ref{tab:param_range}) for which the population-based model \eqref{eq:learned_model} is numerically solved over a given time horizon. Then, the corresponding trajectory of $(S,I)$ is compared to the one simulated via the (IBM) initialized with the same parameters. In each case, the time-series of the number of susceptible and infected individuals are respectively compared in light of both qualitative behaviour and quantitative error criterion. More precisely, to perform qualitative comparison, we were mostly interested in the ability of the \eqref{eq:learned_model} learned model to correctly predict \emph{immediate disease extinctions}, even in parameter regimes in which stochasticity plays a significant role (low dispersion coefficient $\kappa$ or small population size ratio $n$). 

 On Figure \ref{fig:learned_mdl_validation_cstant_parameters}, we illustrate our approach on several examples: we compare for different parameters ($n$, $\beta$ and $\kappa$) the average of the IBM and our data-driven population-based model. The outcome is pretty convincing and the reduced model shows decent accuracy for a wide range of parameter values (low population size ratio and dispersion coefficient, etc.). These results suggest that the model captures well the dynamics involved in a heterogeneous epidemic and might be used to analyze those. The results also show that in many cases, the bifurcation outbreak/extinction is correctly captured by the model. We further discuss this issue at the end of the next section.

\begin{figure}
    \centering
    \includegraphics[scale=.3]{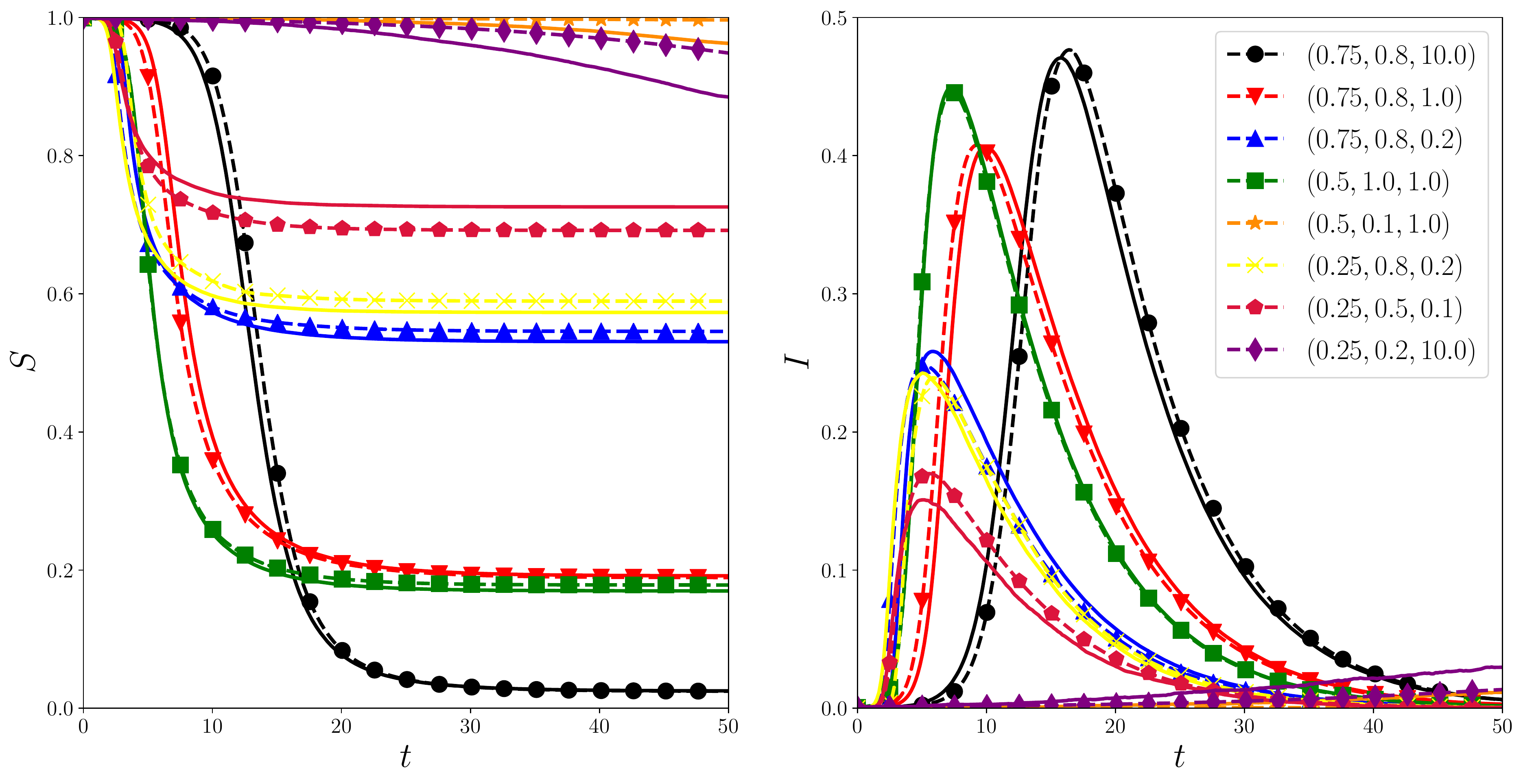}
    \caption[flushleft]{Response of the learned model \eqref{eq:learned_model} to parameters that are \emph{constant over time}. In the legend, these are specified in the format $(n,\beta,\kappa)$. Learning is based on the data set $\mathcal{D}$. In dotted lines with markers: predictions of the \eqref{eq:learned_model} learned model. In continuous line: average of the IBM (individual trajectories are not shown for clarity). \label{fig:learned_mdl_validation_cstant_parameters}}
\end{figure}

\subsection{Byproduct of our approach: estimating key epidemiological quantities}
\label{sec:key-epidemio_quantities}

So far, we trained the population-based model \eqref{eq:learned_model} which is expected to faithfully capture not only the global dynamics arising from individual variation, but also the impact of super-spreaders on an epidemic. In this subsection and based on the latter reduced model, we are working towards defining a \emph{threshold number} via the so-called next-generation matrix theory and estimating an \emph{epidemic size}. The goal is not so much about getting precise quantitative results regarding those epidemiological indicators, but rather to make qualitative statements and gain insight into how the parameters at stake, namely $n, \beta$ and $\kappa$, interact and influence them. Indeed some quantities like the \emph{epidemic threshold} are very useful for the epidemiologists but, contrary to the classical homogeneous models, not easy to calculate when super-spreaders are taken into account \cite{pmid:16292310}.

\paragraph{Estimating a threshold number.} In the case of population-based models and under suitable hypothesis, the \emph{next-generation matrix} theory introduced by Diekmann \textit{et al.} \cite{diekmannDefinitionComputationBasic1990} offers a \emph{systematic} framework to define a threshold number whose properties are identical to the well known $\mathcal{R}_0$ in the traditional SIR model. By analogy, this threshold number will hereafter be called $\mathcal{R}_0$. Informally, it provides information about the stability of the disease-free equilibrium (DFE) $(S,I)=(1,0)$ in the population-based model \eqref{eq:learned_model}. That is, if $\mathcal{R}_0<1$, then the DFE is locally asymptotically stable and the epidemic dies out; if not, then it is unstable and an outbreak occurs \cite{martchevaIntroductionMathematicalEpidemiology2015,vandendriesscheFurtherNotesBasic2008}.

Having numerically checked that the learned incidence function $F_\theta$ satisfies all the hypotheses of the next-generation matrix theory  (mainly dealing with positivity) \cite{diekmannDefinitionComputationBasic1990}, the calculation is straightforward and unambiguous. The threshold number $\mathcal{R}_0$ can be seen as a scalar function of the three parameters $(n,\beta,\kappa)$ given by:
\begin{equation}
  \mathcal{R}_0(n,\beta,\kappa) =  \left. \frac{\partial_{I} F_\theta }{\gamma}\right|_{(S=1,I=0;n,\beta,\kappa)}.
    \label{eq:R0_next_gen_matrix}
\end{equation}

We refer to  Appendix \ref{apdx:rigorous_R0_derivation} for a detailed derivation of this formula. Equation \eqref{eq:R0_next_gen_matrix} suggests that, in the early stage of an epidemic, an outbreak is all the more likely to occur as the rate of secondary infections is sensitive to increases in infected individuals.  Since the parametric function $F_{\theta}$ is a neural network, its partial derivative can be easily computed using automatic differentiation implemented in libraries such as Keras \cite{franccoischolletandothersKeras2015}.

To visualize the dependence of the threshold number on its parameters, we plot in Figure \ref{fig:seuil_next_gen_surf_col}, for many values of population size ratios and dispersion coefficient $(n,\kappa)$, the ratio $\beta_{c}(n,\kappa)/\gamma$ where $\beta_{c}(n,\kappa)$ denotes the critical transmission rate value above which the DFE becomes unstable.
More precisely, $\beta_c=\beta_c(n,\kappa)$ is the smallest value of the transmission rate such that the following inequality holds: $\mathcal{R}_0 \left(n,\beta_c (n,\kappa),\kappa\right) \geqslant 1.$

\begin{figure}
    \centering
    \includegraphics[scale=.5]{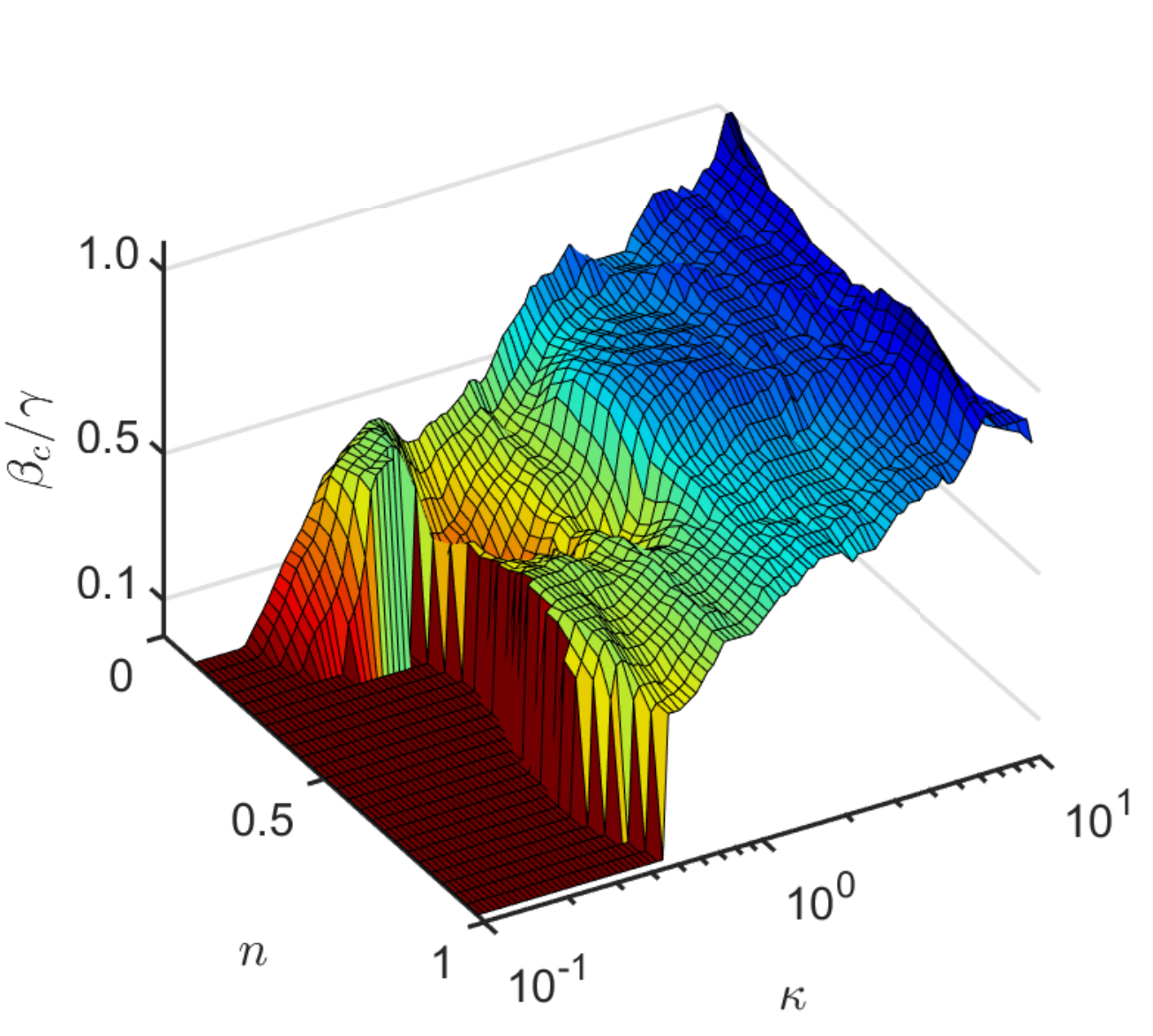} \quad    \includegraphics[scale=.5]{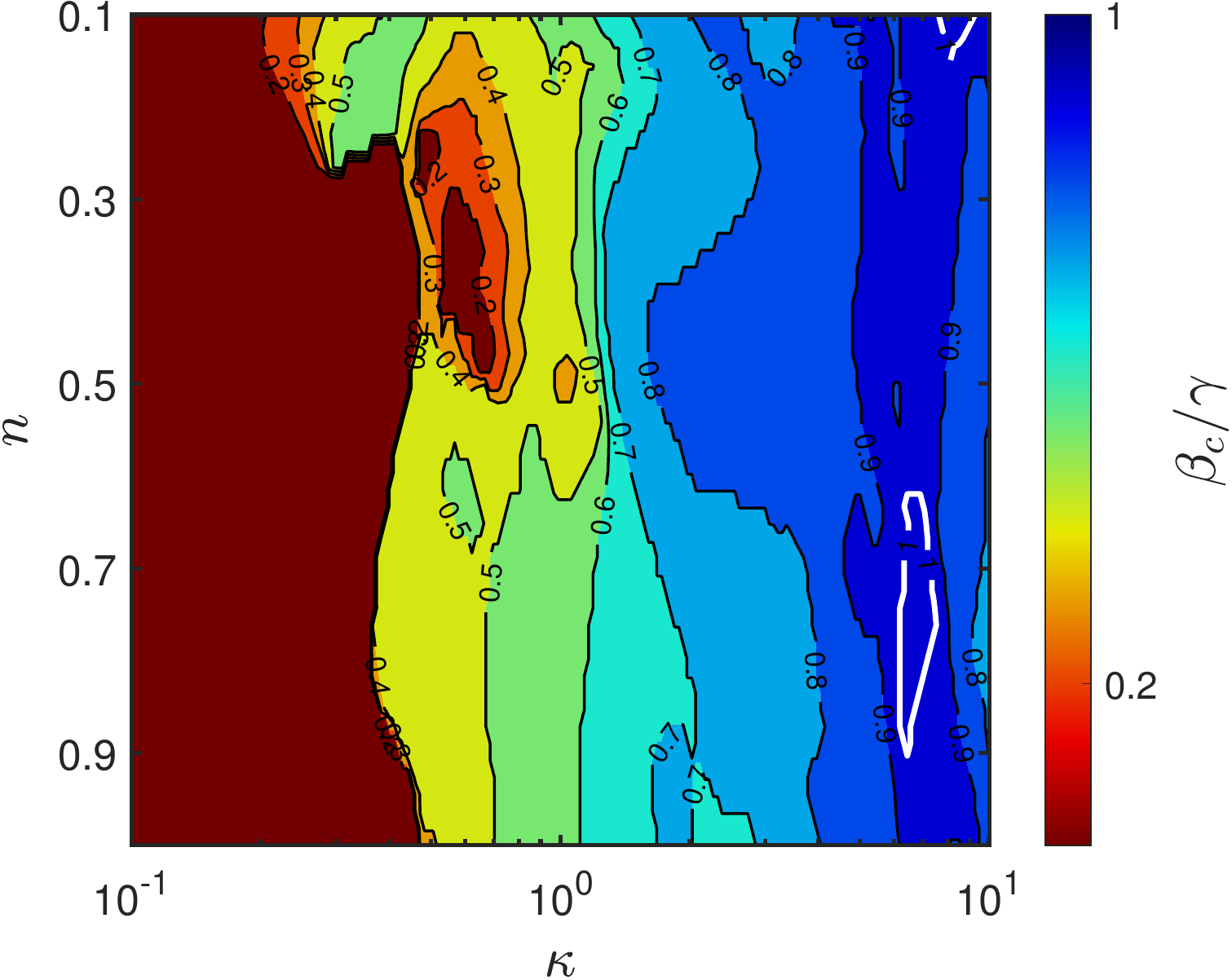}
    \caption[flushleft]{Left: We represent the critical value $\beta_c$ beyond which the DFE is unstable. It is normalized by $\gamma$ to obtain a dimensionless number (similar to $\mathcal{R}_0$ in the classical SIR model). The grid used contains 50 points for $n$, 200 for $\kappa$ and 200 for $\beta$. 
    Right: Same plot in 2D. The white contour corresponds to the case where the threshold is exactly equal to one.}
    \label{fig:seuil_next_gen_surf_col}
\end{figure}

In Figure \ref{fig:seuil_next_gen_surf_col}, we observe that the critical transmission rate $\beta_{c}$ is an increasing function of the dispersion coefficient $\kappa$. Consequently, with small dispersion coefficients (and thus super-spreaders), low transmission rates may be more likely to lead to the development of epidemics. In other words, low-dispersion diseases have a high risk of developing into epidemics. This is the kind of tendency we would expect, as suggested in the work \cite{lloyd-smithSuperspreadingEffectIndividual2005}.
For dispersion coefficients lower than $0.4$, the critical transmission rate is smaller than $10^{-3}$ and the graph looks flat. The estimation of the threshold number $\mathcal{R}_0$ is probably not very accurate or relevant in this region. Indeed, for these parameters, the individual-based dynamics are very stochastic and the population-based model may have some difficulty to learn the relevant threshold. However, the validation results (see previous subsection) show that the model captures epidemic outbreak well even for low $\kappa$. The difficulty is thus likely to be localized around the critical transmission rate. To validate this explanation, we plot in Figure \ref{fig:NGM_vs_empirical} the critical transmission surface projected in 2D and compare it with the individual-based simulations. Green dots refers to IBM simulations without outbreak, while red ones refer to simulations with outbreak. If the model were perfect, the green dots would all be under the surface and the red ones above the surface. We observe this behaviour for $\kappa>0.4$ and see that the trend described by the surface corresponds to the one found empirically.

\begin{figure}
    \centering
    \includegraphics[scale=.6]{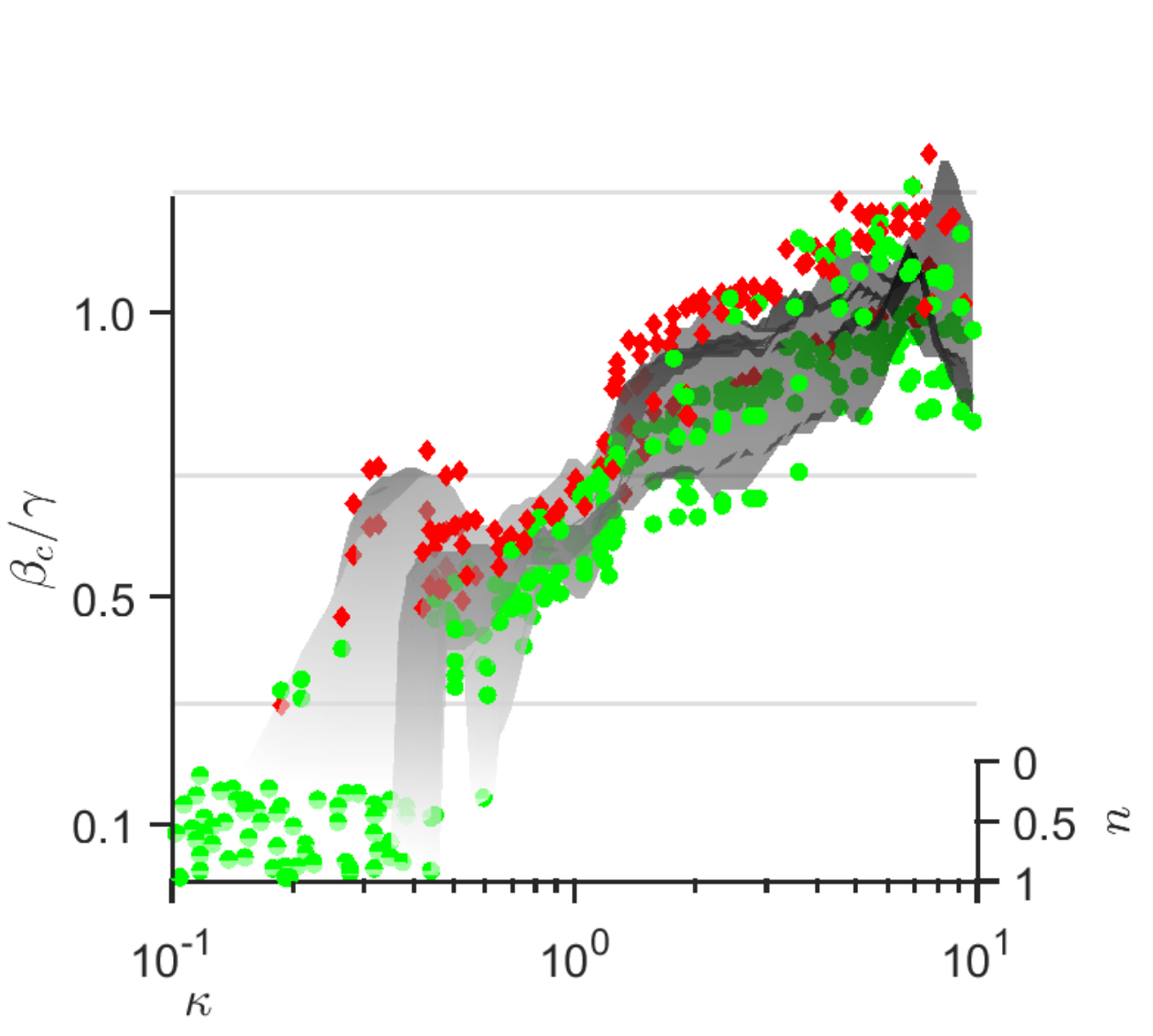}
    \caption[flushleft]{Green dots: stable IBM simulations. Red dots: IBM simulations with outbreak epidemic. In grey the projection of the critical transmission rate surface of Figure \ref{fig:seuil_next_gen_surf_col}.  }
    \label{fig:NGM_vs_empirical}
\end{figure}

\paragraph{Estimating the epidemic size.} The epidemic size, hereafter denoted $\mathcal{R}_\infty$, is defined as the total number of people who caught the disease. The epidemic size thus corresponds to the number of recovered individuals in large time:
\begin{equation}
    \mathcal{R}_\infty = \lim_{t \to +\infty} R(t).
    \label{eq:R_infty_definition}
\end{equation}
This quantity can be seen as a function of the parameters $(n,\beta, \kappa)$ as the dynamics of $R$ is depending on them. In practice, the epidemic size is found by running the population-based model \eqref{eq:learned_model} on a sufficiently long time interval, namely $T$ equal to 200, which is large enough so that the system dynamics approach a stationary state. Figure \ref{fig:rinfty_map_col} shows contour plots of $ \mathcal{R}_\infty$ based on the outcome of numerical simulations with, for each  population size ratio $n = 0.1, 0.2$ and $0.8$, many parameters $(\beta, \kappa)$. 
This suggests that:
\begin{enumerate}
\item $\mathcal{R}_\infty$ is less dependent on the infection rate $\beta$ whenever $\kappa \leq 1$,
\item dependency of $\mathcal R_\infty$ on $n$ is more sensitive in the large dispersion coefficient case ($\kappa >1$) and for large values of the transmission rate $\beta$,
\item $\mathcal{R}_\infty$ seems to decrease with $\kappa$; a possible interpretation of this observation is that if epidemics are more intense but shorter, the total number of infected people may be less than if the epidemic is less intense but extends over longer periods of time.
\end{enumerate}

\begin{figure}
    \centering
    \includegraphics[width=\linewidth]{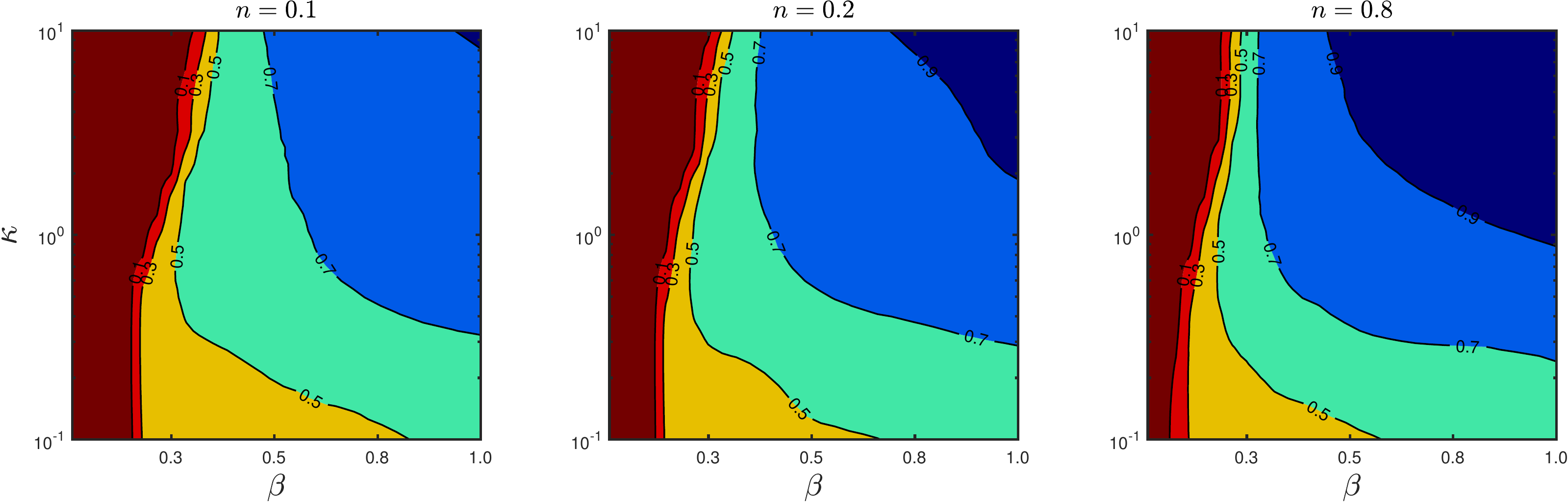}
    \caption[flushleft]{$\mathcal{R}_\infty$ as a function of $(\beta,\kappa)$ for different population size ratios $n$. The grid used contains 50 points for $\kappa$, 30 points for $\beta$.
    \label{fig:rinfty_map_col}}
\end{figure}

\subsection{Model limitations}
\label{sec:model_limitations}

The results of the previous section show that the data-driven population-based model \eqref{eq:learned_model} is able to capture complex dynamics with super-spreading in a wide range of population size ratios. Moreover, the neural network structure of the incidence function makes the analysis of the reduced system theoretically and numerically easier. 

However, all previous simulations were concerned with parameters $\beta$, $\kappa$ that were \emph{constant over time}. When considering time varying parameters, as would be required to control the dynamics, the model no longer works well as observed on Figure~\ref{fig:learned_mdl_validation_PWC_parameters_ex1}.

A possible explanation is that learning the incidence function for parameters close to the epidemic threshold $\beta_c$ is delicate. It is likely that constant parameters choices generate smaller sets for $(S,I,R)$ as the ones required to capture the epidemic threshold $\beta_c$ in more general cases, and that the transmission rate function $f_\theta$ is non regular. Indeed, when considering non-constant parameters, many additional configurations $(S,I,R)$ generating bifurcations may arise. Roughly speaking, the incidence function is probably more intricate to approximate. 

The next section is dedicated to improving the data-driven population-based model in order to manage such time-varying parameters. It is an essential ingredient for the control strategy of the individual-based model.

\begin{figure}
\centering
\begin{subfigure}{.33\linewidth}
  \centering
  \includegraphics[width=\linewidth]{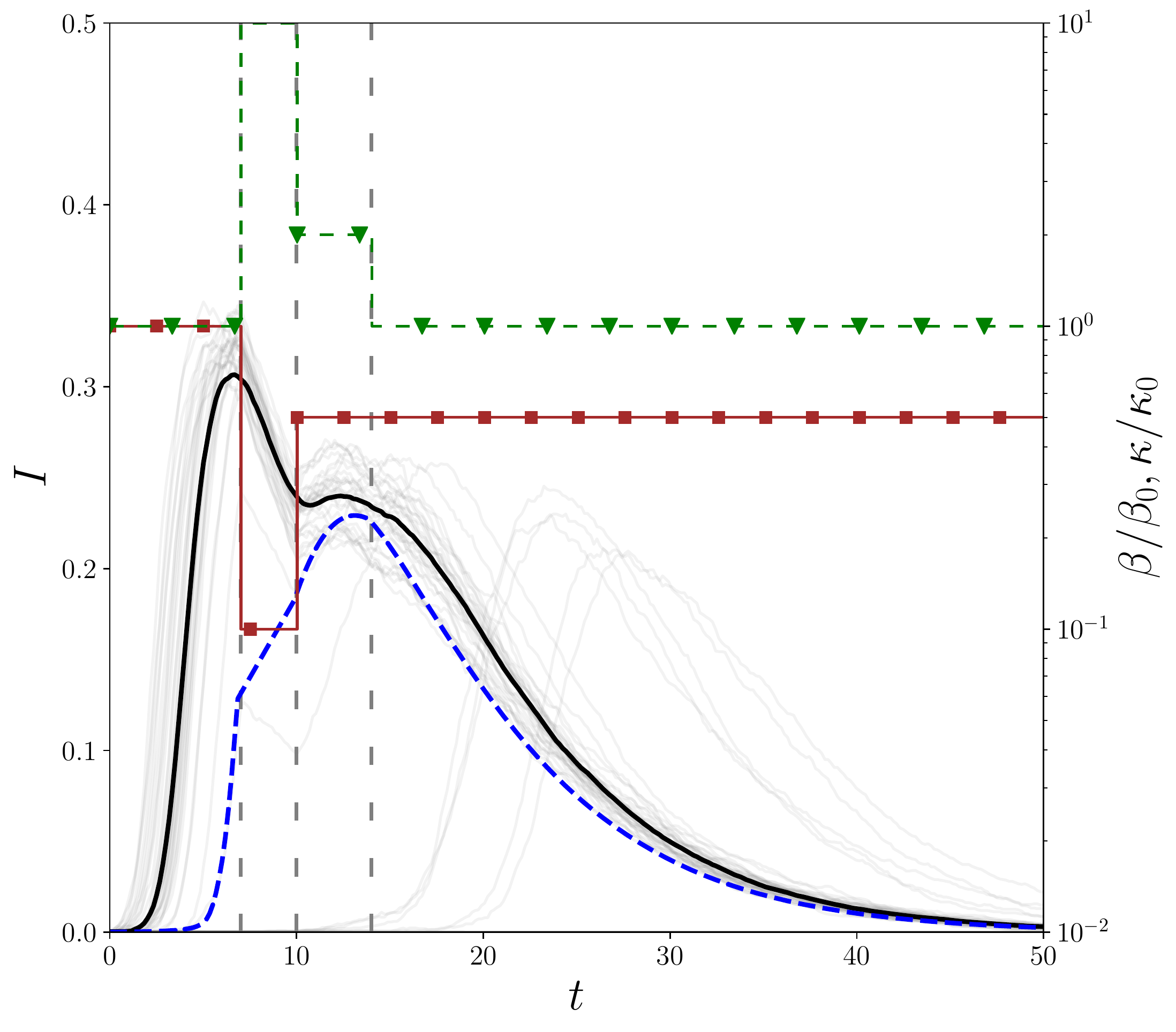}
  \caption{$n = .15$, $\beta=(.8,.08,.4,.4)$,\\ $\kappa=(.5,5,1,.5)$}
  \label{fig:pwcSubFig1}
\end{subfigure}%
\begin{subfigure}{.33\linewidth}
  \centering
  \includegraphics[width=\linewidth]{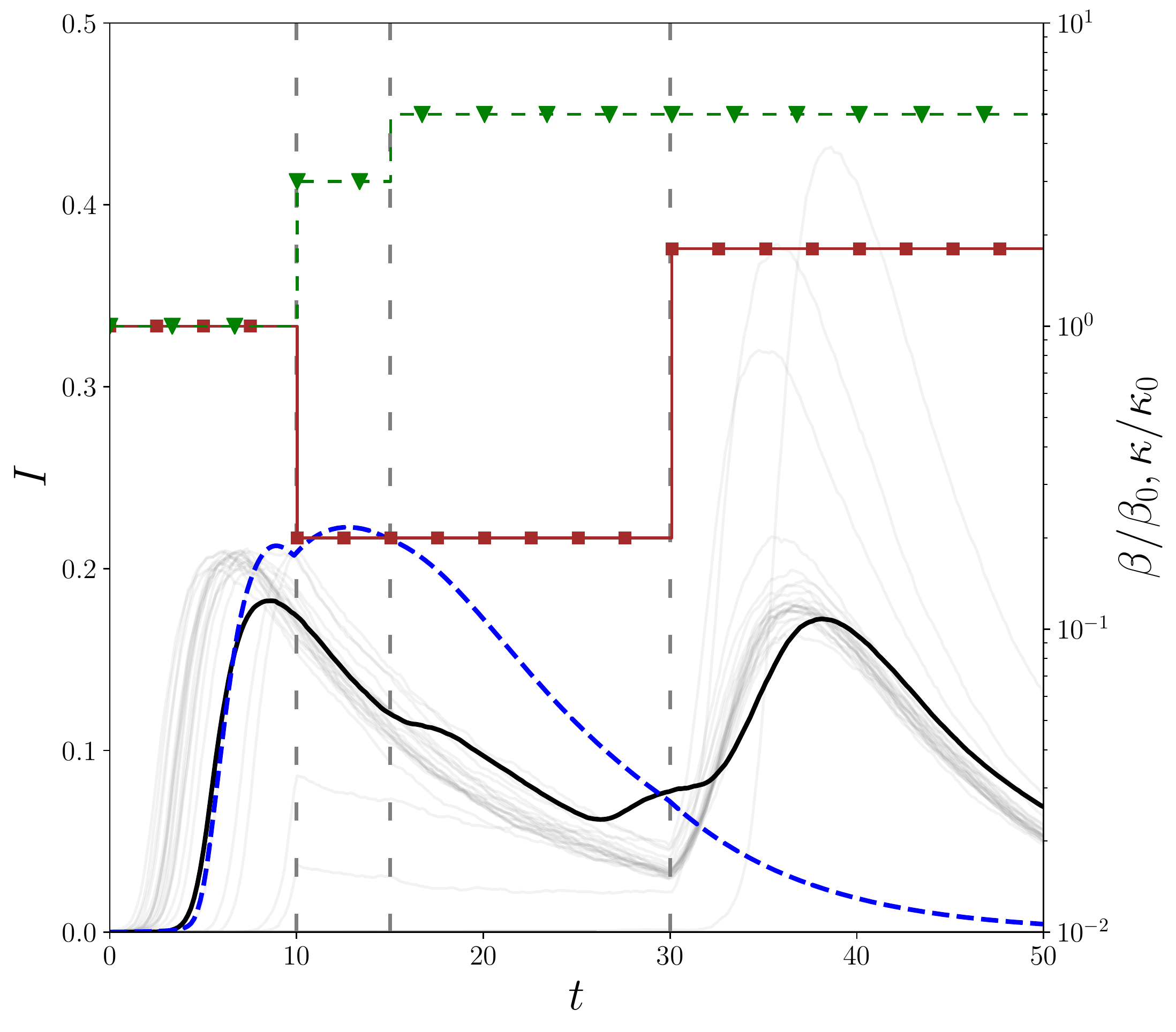}
  \caption{$n = .5$,$\beta=(.5,.1,.1,.9)$,\\ $\kappa=(.2,.6,1,1)$}
  \label{fig:pwcSubFig3}
\end{subfigure}
\begin{subfigure}{.33\linewidth}
  \centering
  \includegraphics[width=\linewidth]{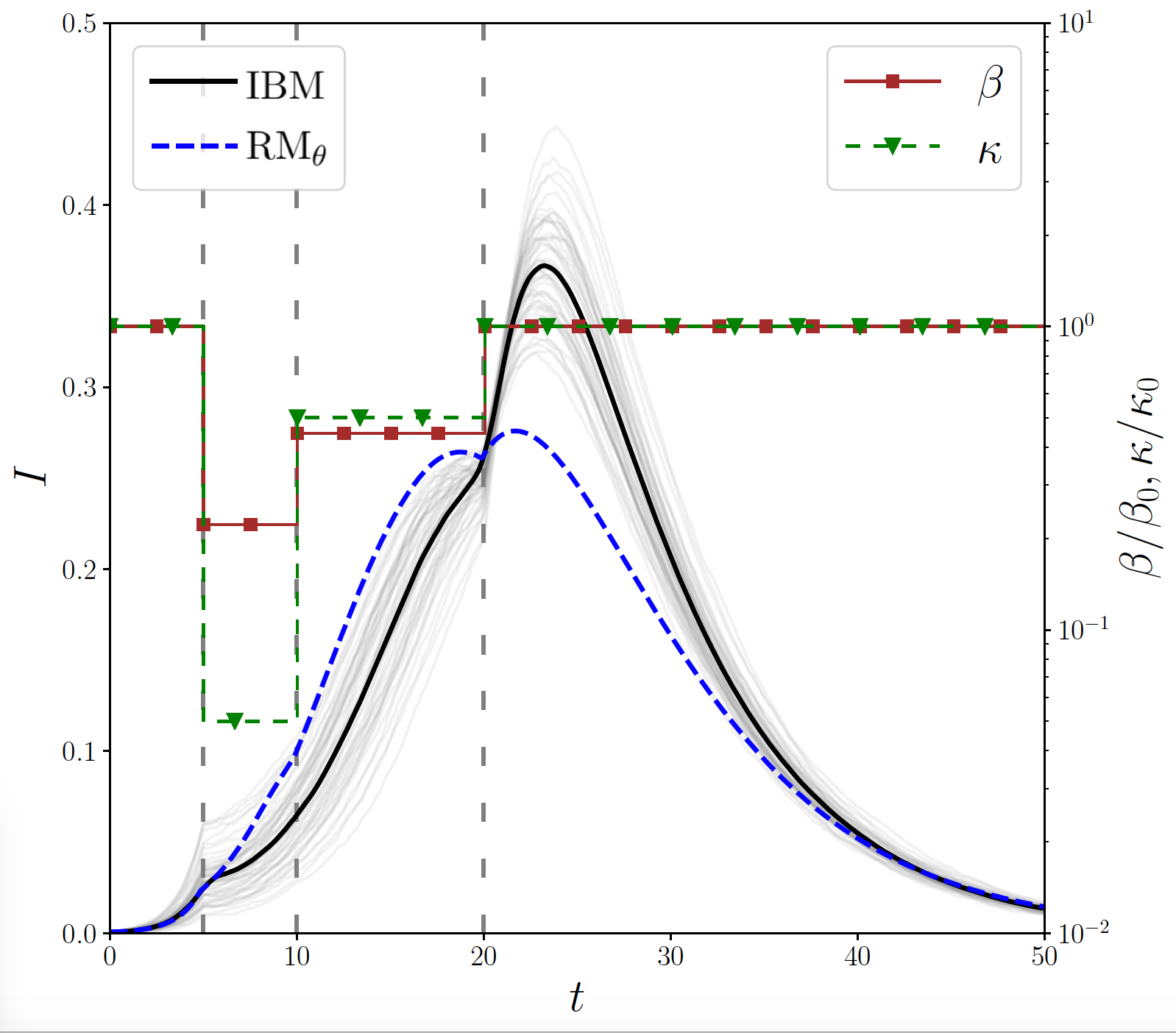}
  \caption{$n = .75$, $\beta=(.9,.2,.4,.9)$,\\ $\kappa=(10,.5,5,10)$}
  \label{fig:pwcSubFig2}
\end{subfigure}
\caption{Comparison of the trajectory $I$ between the \eqref{eq:learned_model} model (dashed line blue) and IBM (in black on the figure) in three cases where the parameters $\beta$ (red squares markers) and $\kappa$ (green triangle markers) \emph{vary over time}.
\label{fig:learned_mdl_validation_PWC_parameters_ex1}}
\end{figure}

\section{Optimal control strategies based on the data-driven population based SIR model}

In this section, we seek to define a control problem for the individual-based model with heterogeneous contacts. We will consider a problem in which we aim at minimizing the maximum number of infected individuals by including an important constraint reflecting the limited capacity of hospitals. The control parameters, allowing to act on the population-based SIR system \eqref{eq:learned_model}, are the coefficients $\beta$ and $\kappa$. The action expressed by $\beta$ can be seen as a policy concerning the whole population (such as lockdowns, indoor masks) contrary to the action expressed by $\kappa$ which allows acting on super-spreaders (cancelling of large events for instance). Since the control of the (IBM) is complex, we choose to use the reduced model to determine an optimal policy. However, as mentioned before, the model is not efficient anymore whenever $\beta$ and $\kappa$ are non-constant functions of the time.
As a consequence, if we use it, we will have a reduced model and {\it a fortiori} a control that will not be relevant in some areas. To avoid this, we will use the principle of \textbf{model predictive control with local model} (MPC), where we learn a local model around a trajectory, compute the control and learn the controlled trajectory. By iterating this method, we expect to obtain a control relevant to the original individual-based model. 

Before introducing this strategy, we will first define an optimal control on the data-driven population based SIR model constructed in Section \ref{sec:learning_global_model}. As expected, without additional care being taken, this control deteriorates the fit between the reduced population based-model and the individual-based model. Subsequently, an MPC-based reinforcement learning strategy is introduced to compute an accurate control of the individual-based model.

\subsection{Optimal control of the data-driven reduced population based model}
\label{sec:OC_learned_system}
This section is dedicated to applying the standard theory of optimal control (OC) to the learned model \eqref{eq:learned_model} involving the incidence function $F_\theta$ whose expression has been obtained thanks to a neural network. Recall that the weights of the latter neural network have been optimized so that the output of $F_\theta$ accurately estimates the rate of change of the proportion of susceptible individuals (see Section \ref{sec:learning_global_model}). Therefore, it is not clear that the partial derivatives of $F_\theta=F_\theta(S,I;n,\beta,\kappa)$ are also good approximations of the corresponding quantities. Thus, when solving the equations numerically, difficulties may arise due to the fact that the tools of OC theory rely heavily on the \emph{differentiation} of the system dynamics with respect to the state variables and control.

Since the learned model captures in particular the dynamics related to contact heterogeneity, the health policies we consider are not only modeled by variations in the transmission rate $\beta$ but also in the dispersion coefficient $\kappa$. A decrease in the former is considered to be the mathematical translation of mandatory measures such as lockdown or indoor masks. In contrast, an increase in the latter should be seen as a consequence of decisions to either close down places where super-spreaders are likely to be found, such as nightclubs or concert halls, or to require people to hold a valid COVID certificate. 
In what follows, we will use the coefficients $\beta$ and $\kappa$ as optimization variables to contain the epidemic.
The idea behind such a choice is that decreasing the coefficient of dispersion flattens the tail of the negative binomial probability density.

\paragraph{Introduction of the reduced controlled system.} Suppose that at the initial time, only a small proportion of a population of size $1=S(t)+I(t)+R(t)$ has contracted a disease whose transmission rate is estimated to be $\beta_0$ and that the coefficient of dispersion is approximately known to be $\kappa_0$. Furthermore, let $T>0$ be the time horizon up to which we wish to study the effect of given health policies and $T_c < T$ be the non-negative time required for a sufficient number of secondary infections to occur and the health authorities to intervene. At this time, the state of the population is denoted by the non-negative numbers $S_c$ and $I_c$.

The equations describing the dynamics of the system under the health policies $b(\cdot):=\beta(\cdot)/\beta_0$ and $k(\cdot):=\kappa(\cdot)/\kappa_0$ are given over the time interval $[T_c, T]$ as
\begin{align}
	\begin{pmatrix} S'\\ I' \end{pmatrix} = g\left(S,I,bv (k), k \right ), \qquad \qquad \qquad 
    \begin{pmatrix}S\\I\end{pmatrix} (T_c) = \begin{pmatrix}S_{c}\\I_{c}\end{pmatrix},
    \label{eq:controlled_system}
\end{align}
where we have defined the right-hand side by
$$ 
g : \RR ^4\ni (a,b,c,d) \mapsto \left ( -F_\theta(a,b;n,c\beta_0,d\kappa_0),
                                     F_\theta(a,b;n,c\beta_0,d\kappa_0) - \gamma b \right )\in \RR ^2. 
$$
Since the total population size is constant in time, it is enough to consider only the $(S,I)$ equation in the control problem. In addition, the health policy $(b,k)$ is assumed to belong to the set of admissible control
\begin{equation}
    \mathcal{U} = \left\{(b,k) \in L^\infty(T_c,T;\RR ^2) : b_{\mathrm{min}}\leq b(\cdot) \leq 1, \: 1\leq k(\cdot) \leq k_{\mathrm{max}}\; \mathrm{a.e}.\right\},
\end{equation}
where $b_{\mathrm{min}} \in (0,1]$ and $ k_{\mathrm{max}} > 1$ are given and reflect the fact that a perfect application of sanitary measures is unrealistic. It is recalled in Appendix~\ref{apdx:well_posed_controlled_system} that in this setting, the ODE system~\eqref{eq:controlled_system} is well-posed. Let us also comment on the choice of transmission rate $bv(k)$ appearing in model~\eqref{eq:controlled_system}. The mapping $v$ is defined as a non-increasing function of $k$ by  
$$
v : [1, k_{\max}]\ni k \mapsto 1/\left(1+\log_{10}(k)\right)\in \RR .
$$ 
It is a relatively simple way to account for the fact that controlling super-spreaders (e.g. closing down certain public places or introducing mandatory COVID certificate) invariably has an influence on the whole population. Mathematically, we take this into account by expressing that the action on $\kappa$ (the action on the super-spreaders) also has a small influence on $\beta$ (on the whole population). Introducing an effective rate $bv(k)$ enables us to couple both the controls. 

\paragraph{Towards an optimal control problem for the reduced model.} There are different aspects that we want to include in the definition of the optimal control problem:
\begin{itemize}
\item on the one hand, the application of sanitary measures can be detrimental in the long run, both to the mental health of the citizens but also to the economy. We therefore choose to integrate weights in the definition of the criterion, which is similar to a choice of the political decision-maker. This makes it possible to penalize or not certain types of measures (confinement or closure of certain public places) in the cost of control. 
\item on the other hand, we wish that the epidemic dies out as soon as possible without putting too much pressure on the health infrastructures (hospitals, intensive care units), i.e. having proportions of infected individuals above a given threshold $I_{\mathrm{hosp}}$. 
\end{itemize}
These considerations lead us to balance the costs using a convex combination of three non-negative weights $\omega_\beta$, $\omega_\kappa$ and $\omega_{\mathrm{hosp}}$. The last term of the cost function aims to penalize strongly, say by $1/\varepsilon$ for a small positive $\varepsilon$, any control leading to proportions of infected individuals above a certain threshold $I_{\mathrm{max}} \in ( I_{\mathrm{hosp}}, 1]$. This constraint can be understood as a strong constraint such as the one on intensive care beds for example. 

Thus, by denoting by $I_{b,k}$ the second component of the $(S,I)$ solution to \eqref{eq:controlled_system} associated with a $(b,k)$ health policy, all the previous elements are taken into account in the fixed-time optimal control problem
\begin{equation}\tag{OCP}
\boxed{ \inf_{(b,k)\in \: \mathcal{U}}J[b,k]},
    \label{eq:OCP}
\end{equation}
relying on the cost functional $J$ defined on the set of admissible controls $\mathcal{U}$ by
\begin{equation*}
    J[b, k] = \frac12 \int_{T_c}^T \omega_\beta \left(1-b(t)\right)^2 
                                      + \omega_\kappa \left(k(t)-1\right)^2 
                                      + \omega_{\mathrm{hosp}} \left(\frac{I_{b,k}(t)}{I_{\mathrm{hosp}}}-1\right)_+^2 
                                      + \frac{1}{\varepsilon} \left(\frac{I_{b,k}(t)}{I_{\mathrm{max}}}-1\right)_+^2 \mathrm{d}t.
    \label{eq:cost_fction}
\end{equation*}
Notice that, in the definition above, the purpose of the positive part function is to avoid penalizing efficient health policies that limit the quantity of sick individuals to proportions lower than $I_{\mathrm{hosp}}$. 

\paragraph{On the existence of an optimal control.}

In the problem~\eqref{eq:OCP}, the control intervenes in the dynamics in a strongly nonlinear way. It is known that, for such control problems, it is not guaranteed that a solution exists and phenomena such as relaxation or homogenization of the minimizing sequences, leading to numerical pathologies, may occur. 
For this reason, we will in fact slightly modify the previous optimal control problem by adding a regularization term to the cost function $J$. We have decided to consider here a BV regularization\footnote{Recall that if $\Omega$ denotes an open set of $\RR^n$, $f$ belongs to $\operatorname{BV}(\Omega)$ whenever $f$ belongs to $L^1(\Omega)$ and 
$$
\operatorname{TV}(f)<+\infty\quad \text{where}\quad \operatorname{TV}(f)=\sup_{\substack{\psi \in C^1_c(\Omega;\RR^n)\\ \Vert \psi\Vert_{L^\infty(\Omega)=1}}}\int_\Omega f\operatorname{div}\psi .
$$
The Banach space $\operatorname{BV}$ is endowed with the norm $\Vert \cdot\Vert_{\operatorname{BV}(\Omega)}$ defined by $\Vert f\Vert_{\operatorname{BV}(\Omega)}:=\Vert f\Vert_{L^1(\Omega)}+\operatorname{TV}(f)$.
} of \eqref{eq:OCP}, by introducing the following problem:
\begin{equation}\tag{OCP$_\delta$}
  \boxed{  \inf_{(b,k)\in \: \mathcal{U}}J_\delta[b,k]},
    \label{eq:OCPreg}
\end{equation}
where $\delta>0$ is a parameter standing for the strength of the regularization and
$$
J_\delta[b,k]=J[b,k]+\delta (\operatorname{TV}[b]+\operatorname{TV}[k]).
$$
Such regularization is interesting from several points of view. For example, if the control is of the bang-bang type\footnote {in other words, if the control takes only two distinct values}, the BV regularization imposes a maximum number of switches, which may reflect an economic cost. On the other hand, this term imposes the membership of the control to the BV Banach space, which leaves the freedom to choose the control functions among a large variety of functions, not necessarily continuous. From now on, we will have to make the following regularity assumptions on the function $F_\theta$:
\begin{equation}\tag{$\mathcal{H}_{F_\theta}$}\label{hypFtheta}
\begin{minipage}{15cm}
$F_\theta:\Omega\to \RR$ is $W^{1,\infty}$ and of the form \eqref{Ftheta_part},
\end{minipage}
\end{equation}
where $\Omega=[0,1]^3 \times [\beta_0 b_{\min},\beta_0]\times [\kappa_0,\kappa_0 k_{\max}]$.
Under this assumption, System~\eqref{eq:controlled_system} has a unique global solution that belongs to $W^{1,\infty}(T_c,T;\RR^3)$, according to Appendix~\ref{apdx:well_posed_controlled_system}.

We claim that Problem~\eqref{eq:OCPreg} has a solution $(b,k)$. Since the arguments are rather standard, we refer to Appendix~\ref{sec:analysisOCP} for additional explanations.  Let us now introduce the first-order optimality conditions for this problem, which are at the heart of the numerical solution algorithm that we then implement. The optimality conditions for this problem involve the notion of subdifferential $\partial\operatorname{TV}$ of the total variation operator. For the sake of readability, the proof of the following result is postponed to Appendix~\ref{apdx:well_posed_controlled_system}.

\begin{theorem}\label{prop:diffJ}
Let $M_\theta(S,I;n,\beta_0 bv(k),\kappa_0 k)$ denote the matrix 
\begin{equation}\label{def:mtheta}
M_\theta=\begin{pmatrix}
-\partial_S F_\theta (S,I;n,\beta_0 bv(k),\kappa_0 k) & -\partial_I F_\theta (S,I;n,\beta_0 bv(k),\kappa_0 k)\\
\partial_S F_\theta (S,I;n,\beta_0 bv(k),\kappa_0 k) & \partial_I F_\theta (S,I;n,\beta_0 bv(k),\kappa_0 k)-\gamma
\end{pmatrix}
\end{equation}
and let $[p_1,q_1,p_2,q_2]$ denote the solution of the (linear) adjoint system
\begin{equation}\label{eq:adjoint}
-\frac{d}{dt}\begin{pmatrix}
p_1\\ q_1\\p_2\\q_2
\end{pmatrix}=\begin{pmatrix}
M_\theta^\top & 0_{\mathcal{M}_2(\RR)}\\
0_{\mathcal{M}_2(\RR)} & M_\theta^\top 
\end{pmatrix}
\begin{pmatrix}
p_1\\ q_1\\p_2\\q_2
\end{pmatrix}+\left(\frac{\omega_{\mathrm{hosp}}}{{I_{\mathrm{hosp}}}}\left(\frac{I}{I_{\mathrm{hosp}}}-1\right)_++\frac{1}{ I_{\mathrm{max}} \varepsilon}\left(\frac{I}{I_{\mathrm{max}}}-1\right)_+\right)\begin{pmatrix}
0\\ 1\\ 0 \\ 1 
\end{pmatrix}
\end{equation}
completed with the terminal conditions
$$
p_1(T)=q_1(T)=p_2(T)=q_2(T)=0.
$$
The functionals $\mathcal U\ni [b,k]\mapsto (S,I)\in [W^{1,\infty}(T_c,T)]^2$ and $J$ are differentiable, where the pair $(S,I)$ denotes the solution of \eqref{eq:controlled_system} associated with the control choice $(b,k)$. Furthermore, the differential of $J$ is given by
\begin{eqnarray}
\langle dJ[b,k],[h_1,h_2]\rangle &=& \int_{T_c}^T \left(h_1\partial_b J(b,k)  + h_2\partial_k J(b,k) \right)\, {\rm d}t
\end{eqnarray}
for every $[b,k] \in \mathcal U$ and every admissible perturbation\footnote{More precisely, we call ``admissible perturbation'' any element of the tangent cone $\mathcal{T}_{[b,k],\mathcal U}$ at $[b,k]$ to the set $\mathcal U$. The cone $\mathcal{T}_{[b,k],\mathcal U}$ is the set of functions $[h_1,h_2] \in L^\infty (T_c,T;\RR^2)$ such that, for any sequence of positive real numbers $(\varepsilon_n)_{n\in \NN}$ decreasing to 0, there exists two sequences of functions $h_{i,n} \in L^\infty(T_c, T )$ converging to $h_i$, $i=1,2$, as $n \to +\infty$, and $[b,k] + \varepsilon_n [h_{1,n},h_{2,n}] \in \mathcal U$ for every $n \in \NN$.} $[h_1,h_2]$, where 
\begin{eqnarray*}
\partial_b J(b,k) & = &  \omega_\beta   \left(b-1\right)+\begin{pmatrix}
p_1\\ q_1
\end{pmatrix}\cdot \begin{pmatrix}
- \beta_0 v(k) \partial_{\beta} F_\theta \\
 \beta_0 v(k) \partial_{\beta} F_\theta 
\end{pmatrix}\\
\partial_k J(b,k) & = & \omega_\kappa  \left(k-1\right)+\begin{pmatrix}
p_2\\ q_2
\end{pmatrix}\cdot \begin{pmatrix}
- \beta_0bv'(k) \partial_{\beta} F_\theta -\kappa_0 \partial_{\kappa} F_\theta \\
 \beta_0bv'(k) \partial_{\beta} F_\theta +\kappa_0 \partial_{\kappa} F_\theta
\end{pmatrix}.
\end{eqnarray*} 

Now, 
let us assume that \eqref{hypFtheta} is true and let $(b,k)$ denote a solution to Problem~\eqref{eq:OCPreg}. There exist $T_b\in \partial \operatorname{TV}(b)$ and $T_k\in \partial \operatorname{TV}(k)$ such that
$$
\forall (B,K)\in L^\infty(T_c,T;[ b_{\mathrm{min}},1])\times L^\infty(T_c,T;[ 1,k_{\mathrm{max}}]), \quad 
\left\{\begin{array}{l}
\langle \partial_b J-T_b,B-b\rangle_{L^2(T_c,T)} \geq 0\\
\langle \partial_k J-T_k,K-k\rangle_{L^2(T_c,T)} \geq 0.
\end{array}
\right.
$$
\end{theorem}
\begin{remark}[subdifferential of the total variation]
Let us recall that, according to \cite[Proposition~I.5.1]{MR1727362},  the subdifferential of the total variation is given by
$$
\partial \operatorname{TV}(b)=\{\eta \in C^0([T_c,T])\mid \Vert \eta \Vert_\infty\leq 1\text{ and }\int \eta {\rm d}b=\operatorname{TV}(b)\}.
$$
\end{remark}
From a practical point of view, we will not directly use these optimality conditions which remain rather abstract written as they are. Instead, we will regularize the TV term and introduce a descent method using the differential calculation established in Theorem~\ref{prop:diffJ}, using the adjoint state $(p_1,q_1,p_2,q_2)$. The implemented algorithm is introduced in Appendix~\ref{apdx:numerical_implementation}. 

As explained earlier, we will consider an optimal control computed from the reduced non-linear model that we will apply to the individual-based model. To numerically compute an estimate of the control solving the \eqref{eq:OCPreg} problem, we use a direct approach consisting in discretizing the differential systems involved via a regular $\mathcal{S}$ subdivision of the $[T_c, T]$ interval with step-size $\Delta t$. This also allows us to transform the optimal control problem into a nonlinear program whose decision variables are the control values $(b,k)$ evaluated at each point of $\mathcal{S}$. The optimization of the latter values is performed using a relatively simple adaptive step projected gradient algorithm, using a linear search of the step size taken in the direction of greatest descent\footnote{In other words, such that the next control leads, after projection onto the set of constraints, to the greatest possible decrease in the value of the cost functional}. In order to limit the computational cost, the latter online search is performed using a gradient-free method called \emph{golden-section search} \cite{pressNumericalRecipesArt1992}. Details are provided in Appendix~\ref{apdx:numerical_implementation}.

\begin{figure}[ht!]
\centering
\begin{subfigure}{.3\linewidth}
  \centering
  \includegraphics[width=\linewidth]{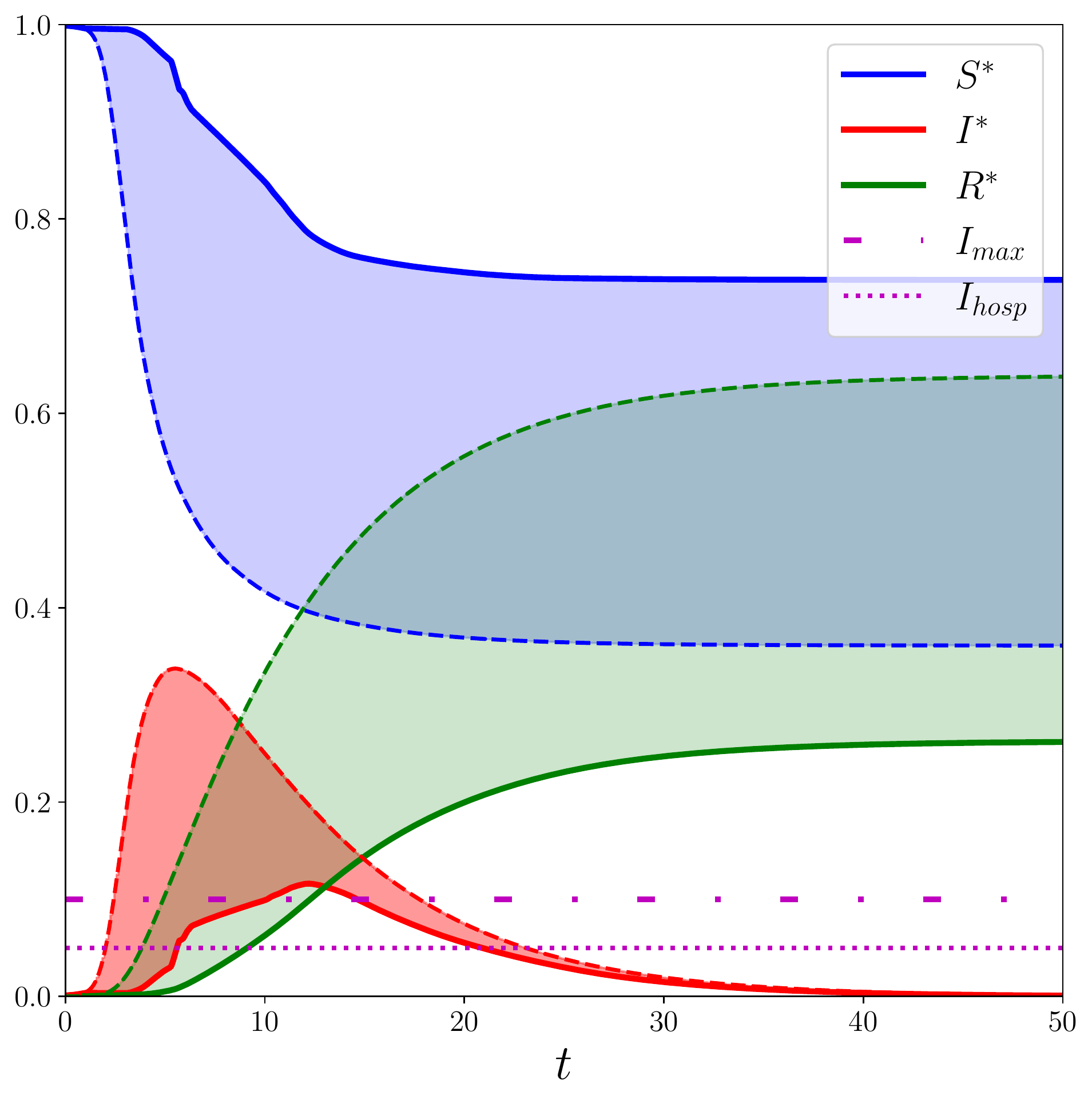}
  \caption{Trajectories with (solid line) / without (dashed line) control.\\}
  \label{fig:learnedMdlCtrlSubFig1}
\end{subfigure}%
\begin{subfigure}{.34\linewidth}
  \centering
  \includegraphics[width=\linewidth]{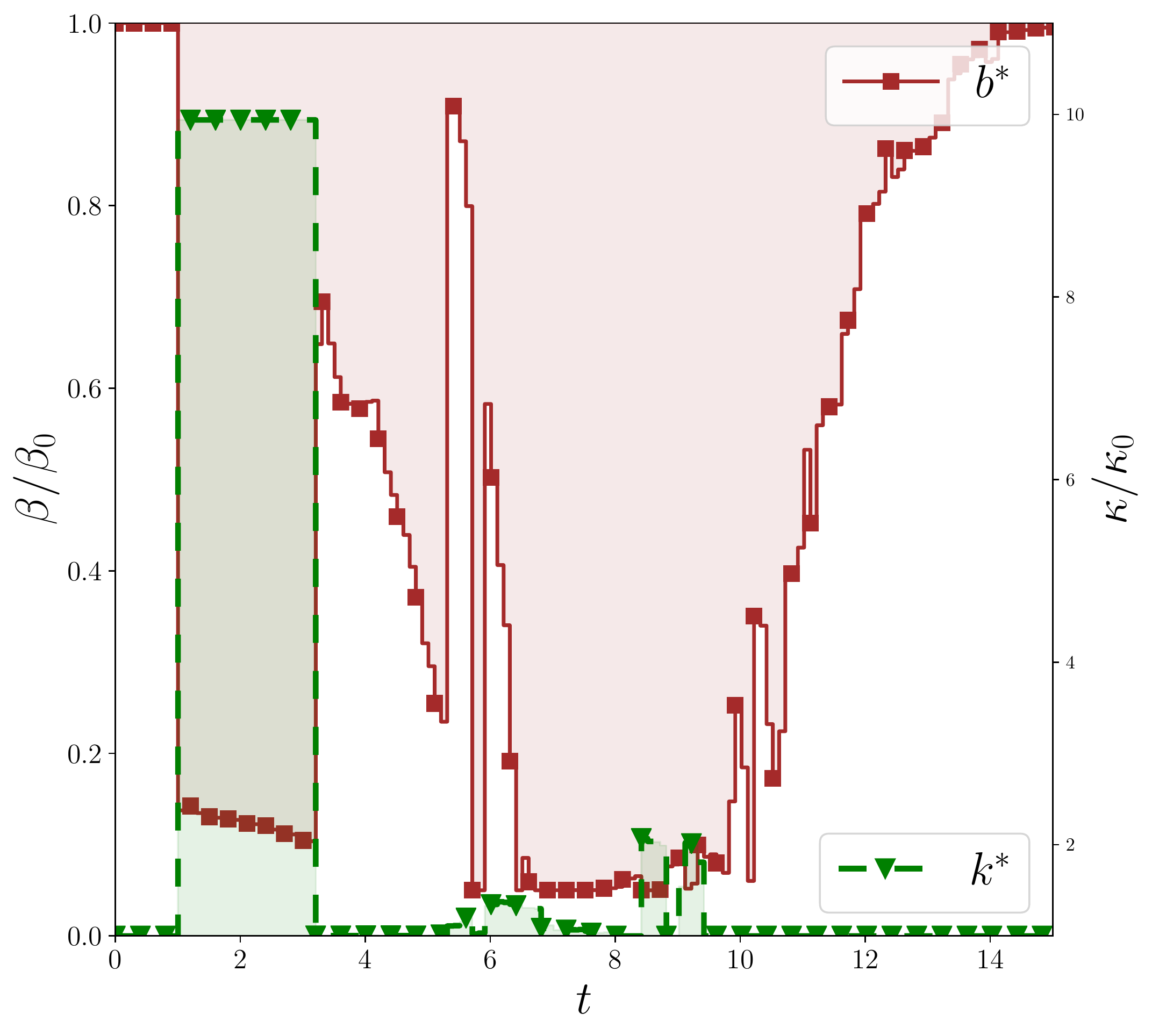}
  \caption{Evolution of the controls over the first 15 days. Afterwards, the controls remain constant.}
  \label{fig:learnedMdlCtrlSubFig2}
\end{subfigure}
\begin{subfigure}{.35\linewidth}
  \centering
  \includegraphics[width=\linewidth]{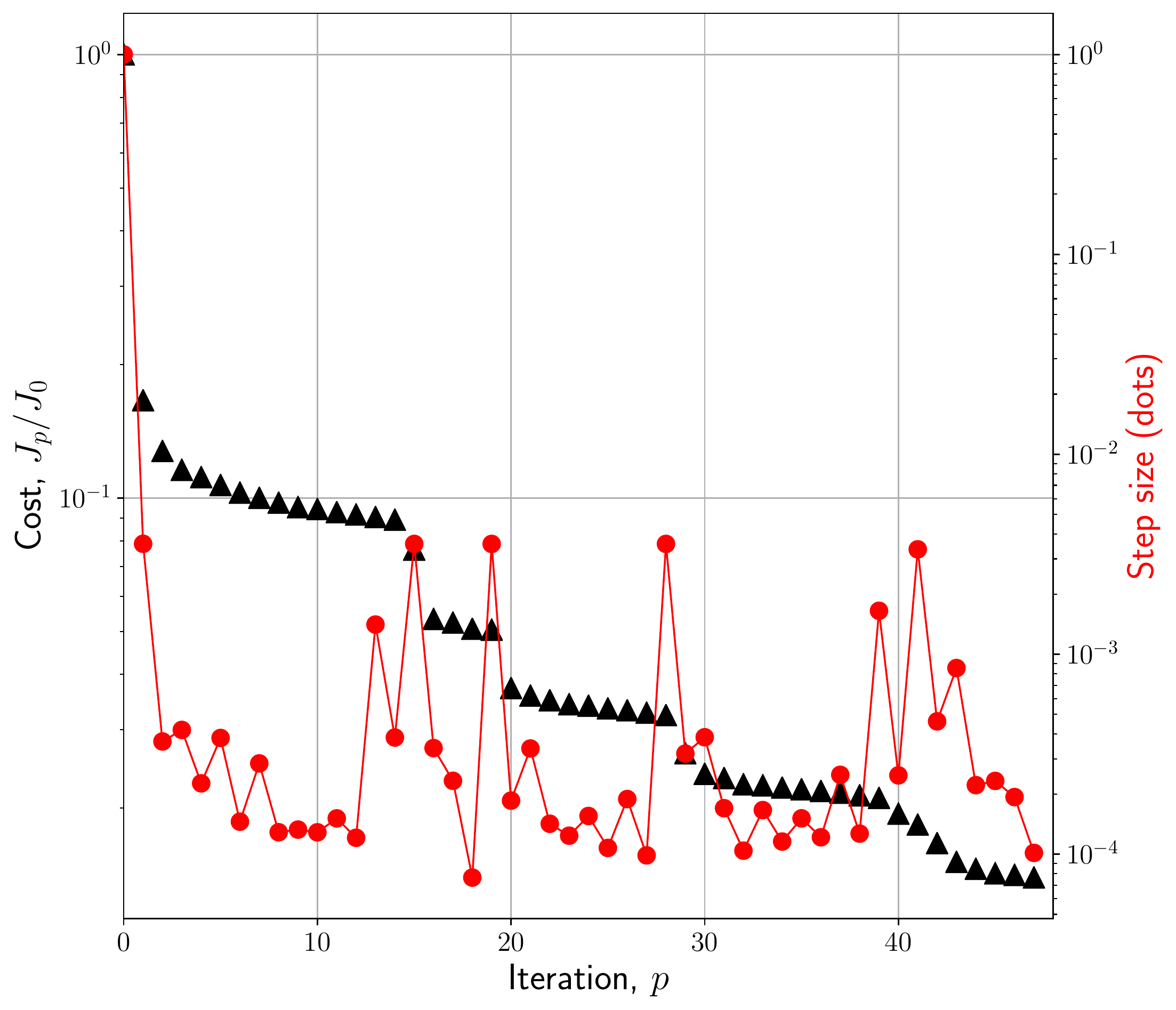}
  \caption{Cost evolution (black triangle markers) and descent step evolution.\\}
  \label{fig:learnedMdlCtrlSubFig3}
\end{subfigure}
\caption{Example of controlled trajectories. Parameters: $(n,\beta_0,\kappa_0)=(0.6,0.8,0.4)$, $T_c = 1$, $T=50$, $\Delta t = 0.1$, $\delta=10^{-7}$, $\varepsilon=10^{-2}$, $(I_\mathrm{hosp},I_\mathrm{max})=(5\%, 10\%)$, $(\omega_\beta,\omega_\kappa,\omega_{\mathrm{hosp}})=(0.2,0.2,0.6)$. Here 50 iterations of gradient are required to converge. On the right figure, $J_p$ denotes the value of $J_\delta$ at the $p$-th iteration.}
\label{fig:learned_mdl_control_example}
\end{figure}

In Figure \ref{fig:learned_mdl_control_example}, we give an example of control computation on the reduced model \emph{independently of} the individual-based model. On the left, we see an example of uncontrolled (dashed lines) and controlled (solid lines) trajectories. We obtain that the maximum number of infected is exceeded during a very short time compared to the uncontrolled trajectories. Since the beginning of the control is delayed (by $T_c$) and since it is not realistic to set $\beta$ (resp. $\kappa$) too low (resp. too high), it is sometimes not possible to avoid exceeding $I_{\mathrm{max}}$. Let us also note that the implemented gradient algorithm allows a priori to determine \emph{only local} minima: there is no guarantee that a global minimizer has been obtained. In the middle, we plot the associated controls of the coefficients $\beta$ and $\kappa$. Finally, it is shown on the right the evolution of the cost function and size of the descent step. As expected, we obtain a decreasing cost function. This example provides an overview of the accuracy of the algorithm and we will now focus on the final algorithm for the individual-based model.

\subsection{Control of the individual-based model (IBM) based on reinforcement learning}\label{sec:OC_of_IBM}
Recall that in Section \ref{sec:learning_global_model}, the complexity of the individual-based model was simplified to obtain a reduced model \eqref{eq:learned_model} consisting of only two deterministic ODEs, at the cost of richness and accuracy.
The reduced model approaches the dynamics of the individual-based model over a wide range of constant parameters (i.e., large ranges of values of $n$, $\beta$, and $\kappa$). In Section \ref{sec:OC_learned_system}, we looked for an optimal health policy specifically for the reduced model. A question then naturally arises: to what extent can a control minimizing the cost function of the optimal control problem for the reduced system be used to obtain a "good" control for the original individual-based model?

In what follows, we seek to improve the above mentioned approach. To this end, we follow an approach borrowed from the theory of model-based reinforcement learning in order to build a model whose sole purpose is to approximate only locally, but very accurately, the dynamics of the IBM. We propose to use the \textbf{model predictive method (MPC)} which consists in alternating between a learning step on the model and an optimal control step. At each iteration, the control trajectory is recomputed based on the IBM and added to the data set used to train the reduced model. In the MPC literature for reinforcement learning problems, there are two families of methods: the global model-based method or the local model-based method (see e.g. \cite{koziel2013surrogate,conn2009introduction}). In Section \ref{sec:model_limitations}, we have seen that it is difficult to build a versatile model capable of handling time-varying parameters $(\beta,\kappa)$. Therefore, we propose to use a more \emph{local} approach. It is common practice to use a valid \emph{linear} model just around a state $(S(t^m),I(t^m))$. However, since we are here able to efficiently control a non-linear system, we propose to compute a valid \emph{non-linear} system around a \emph{whole} trajectory. This choice appears to be a compromise between a local and a global model. We will now describe the algorithm used to control the IBM.

\paragraph{Local model predictive control approach.}
Suppose the setting is such that at the beginning of an epidemic, the authorities have recorded a percentage of infected people $I_{\mathrm{0}}$ with an estimated coefficient of dispersion $\kappa_{\mathrm{0}}$. Moreover, we assume that the transmission rate $\beta_{\mathrm{0}}$ of the disease is known, and that an optimal control problem was defined for the local model \eqref{eq:controlled_system} we aim to construct.

To begin with, the function $F_{\theta(0)}$ is trained on $\mathcal{D}_0$, a very small fraction (e.g. 5-10\%) of the shuffled training dataset $\mathcal{D}$, introduced in Section \ref{sec:learning_global_model}. In other words, the neural network defining $F_\theta$ receives information about the IBM dynamics corresponding to a reduced but representative region of the parameter space $(n,\beta,\kappa)$. The main objective of this exploration is \textit{stability}: it guarantees that the solutions of the ODE \eqref{eq:controlled_system} do not blow up during the numerical computation. Recall that the definition of the population size ratio $n$ has been introduced and commented in Section~\ref{sec:learning_global_model}.

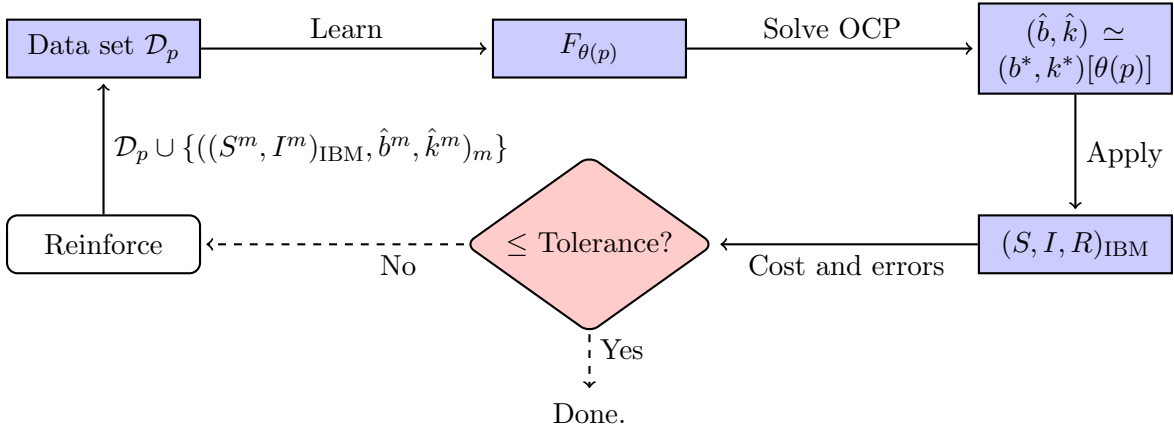
\begin{figure}[!h]
    \centering
    \begin{tikzpicture} [
    auto,
    decision/.style = { diamond, draw=black, thick,
                        text width=7em, fill=red!20, text badly centered,
                        inner sep=1pt, rounded corners },
    block/.style    = { rectangle, draw=black, thick, 
                        fill=blue!20, text width=6em, text centered, minimum height=2em },
    reinfor/.style    = { rectangle, draw=black, thick, rounded corners,
                        text width=6em, text centered, minimum height=2em },
    line/.style     = { draw, thick, ->, shorten >=2pt },
  ]
  \matrix [column sep=35mm, row sep=8mm] {
                     \node [block] (baseDS) {Data set $\mathcal{D}_p$};   & \node [block] (F_theta) {$F_{\theta(p)}$};    & \node [block] (OC) {$(\hat b, \hat k) \simeq (b^*, k^*)[\theta(p)]$}; \\
                    \node [reinfor] (add_data) {Reinforce};               & \node [decision,aspect=1.4] (tol){$\leq $ Tolerance?}; & \node [block] (IBM) {$(S,I,R)_{\mathrm{IBM}}$}; \\
                                                                          & \node [text centered] (out) {Done.};           &                                                 \\
  };
  \begin{scope} [every path/.style=line]
    \path (baseDS)      --    node {Learn} (F_theta);
    \path (F_theta)      --    node {Solve OCP} (OC);
    \path (OC)       --    node {Apply} (IBM);
    \path (IBM)      --    node {Cost and errors} (tol);
    \draw [dashed] (tol)   -- node [near start] {No} (add_data);
    \path (add_data)   -- node [right] { $\mathcal{D}_p\cup\{((S^m,I^m)_{\mathrm{IBM}},\hat{b}^m, \hat{k}^m)_m\}$} (baseDS);
    \draw [dashed] (tol) -- node [near start] {Yes} (out);
  \end{scope}
\end{tikzpicture}
    \caption{Generic iteration $p$ of the reinforcement algorithm used to estimate an efficient health policy for the IBM. Potential policies are selected among optimal $(b^*, k^*)$ controls for a reduced local model involving a parametrized function $F_\theta$. Throughout the algorithm, the weights $\theta$ are sequentially updated until a given control satisfies a criterion related to the cost function of the optimal control problem \eqref{eq:OCP}.}
    \label{fig:flowchart_local_RL}
\end{figure}

Let us now describe a current iteration of the control algorithm using the MPC approach. Assume that the $p$-th iteration of the algorithm begins and that the neural network has a weight configuration $\theta(p)$. Figure \ref{fig:flowchart_local_RL} shows a flowchart of the proposed approach based on the MPC method. Following the steps described in Section \ref{sec:OC_learned_system}, we can thus estimate a control $(b^*, k^*)$ optimally driving the reduced model \eqref{eq:controlled_system} based on the knowledge induced by the weight configuration $\theta(p)$. Depending on the fineness of the partition of the time interval of interest, this optimal health policy may correspond to measures evolving freely on an unrealistic time scale (a few days or a few hours). For this reason, the control $(b^*, k^*)$ is approximated via a regression tree (computed using the SK-learn library) by two piecewise constant functions, denoted $(\hat b, \hat k)$, taking at most 8 different values over a time horizon of 200 days. Starting from the initial configuration corresponding to the operating point $(S,I,n,\beta,\kappa)_{\mathrm{0}}$, the IBM is then simulated under this last policy and the corresponding scenario denoted by $(S,I,R)_{\mathrm{IBM}}$.

We decide to stop the algorithm when the obtained health policy is sufficiently efficient. The stopping criterion is based on the comparison of a cost inspired by the cost function for the reduced optimal control problem \eqref{eq:OCP}. More precisely, we say that the current control $(\hat b, \hat k)$ is \textit{acceptable} with tolerance $\tau_\mathrm{RL}>0$ if 
$$ c_p \leq \tau_\mathrm{RL} c_0,$$ 
where the cost $c_p$ associated with the $p$-th scenario is given by
$$c_p := \int_{T_c}^T \omega_{\mathrm{hosp}} \left(\frac{I_{\mathrm{IBM}}(t)}{I_{\mathrm{hosp}}}-1\right)_+^2 
                      + \frac{1}{\varepsilon} \left(\frac{I_{\mathrm{IBM}}(t)}{I_{\mathrm{max}}}-1\right)_+^2 \mathrm{d}t,$$
and $c_0$ is the cost associated with the control-free IBM solution. Recall that $T_c$ is the beginning time of the intervention of the health authorities (detailed in Section~\ref{sec:OC_learned_system}). In addition to requiring that the control to be acceptable, the $p$-th scenario must, under the same $(\hat b, \hat k)$ control, be associated with a lower cost than the cost of reduced model under the $(\hat b, \hat k)$ control. In other words, by denoting by $c_p^{\rm reduced}$, the cost corresponding to the reduced model with $(\hat b, \hat k)$, the algorithm does not stop until $c_p \leq c_p^{\rm reduced}$. This stopping criterion allows us to \emph{relate the performance} of the control on the reduced model and on the IBM.

Since the success of the algorithm depends upon the ability of the reduced model to \emph{accurately predict} the output of the IBM, the stopping criterion also involves the following three error metrics: we retrieve global information by computing the discrete $L^2$-norm of the difference between the reduced model and the IBM for the state variables $S$ and $I$, and estimating the mismatch between the final proportion $\mathcal{R_\infty}$ of removed people, defined by~\eqref{eq:R_infty_definition}. Accuracy is also assessed by measuring the delay between the time at which the infection peak (IP) occurs, respectively for the IBM and reduced model. The numerical values of the associated tolerances $\tau_{L^2}, \tau_{\mathcal{R_\infty}}$ and $\tau_{\mathrm{IP}}$ are shown in Table \ref{tab:results_paramOCP}.

If the stopping criterion is not satisfied, then the reduced model is strengthened by training the weights $\theta(p+1)$ on a \emph{larger} training set $\mathcal{D}_{p+1}$ containing not only $\mathcal{D}_p$, but also the local information corresponding to the $p$-th scenario $(S,I,R)_{\mathrm{IBM}}$ as well as the parameters defining the candidate control $(\hat{b}, \hat k)$. The above sequence of steps repeats until these criteria are satisfied, at which point the output of this algorithm is the $(\hat b, \hat k)$ health policy corresponding to the last iteration.

Note that, in an attempt to reduce the computational cost and to escape as much as possible from local minima wells, at each $p$ reinforcement step, the optimal control algorithm is initialized with the control obtained at the end of the previous reinforcement step. Moreover, the more we advance in the reinforcement algorithm, the more precise the optimal control algorithm must be (more iterations, smaller step sizes). The reasoning behind this last point is that in the first few iterations of reinforcement, high accuracy is not so important because the behaviour of the reduced model under control is likely to be an unfaithful approximation of that of the IBM.
 
\section{Numerical results}\label{sec:results_local_RL}
In this section, we provide numerical simulations for reinforcement learning based on the model introduced in Section~\ref{sec:OC_of_IBM}. Since it is not easy to make statistics on these results, we will illustrate our approach with numerous examples. In each case, we provide  the quantities $n$, $\beta_{\mathrm{0}}$, $\kappa_{\mathrm{0}}$ and the number of iterations of the reinforcement learning algorithm. We plot on each of them the trajectories relative to susceptible individuals on the left,
the trajectories for the infected in the middle, and the control for the IBM model (red for the $\beta$ control, green for the $\kappa$ control) on the right. Scales for $\beta/\beta_0$ and $\kappa/\kappa_0$ are shown on the left and right on the right figure. In Table~\ref{tab:results_paramOCP}, we specify the parameters common to all the test cases. If one of these parameters changes, it will be written in the legend of the figure.

\begin{table}[h]
    \centering
    \begin{tabular}{cc|cc|cc}
        \toprule
        Parameters        & Values                      & Param.           & Val.  & Param. & Val.\\
        \midrule
        $S_{\mathrm{0}}$      & $99.95\%$     &  $(I_{\mathrm{hosp}}, I_\mathrm{max})$ & $(0.025, 0.1)$ & $T_c$   &  1 \\ 
        $I_{\mathrm{0}}$    & $0.05\%$         & $(\omega_\beta,\omega_\kappa,\omega_{\mathrm{hosp}})$ & $(0.2, 0.2, 0.6)$ & $T$        & 200 \\
        $\gamma$            &  1/6              & $(b_\mathrm{min},k_\mathrm{max})$ & (0.1,10) & $\Delta t$      & 2/7 \\
        $\tau_\mathrm{RL}$ & $10^{-3}$       & $\varepsilon$ & $10^{-2}$ & $\tau_{\mathcal{R}_\infty}$ & $10^{-3}$\\
        $\tau_{L^2}$ & 1             & $\delta$ & $10^{-7}$  & $\tau_\mathrm{IP}$&6 \\
        \bottomrule
    \end{tabular}
    \caption{Numerical values of the main parameters taking part in the reinforcement algorithm (Section \ref{sec:OC_of_IBM}) and \eqref{eq:OCPreg}. These are common to all results shown in Section \ref{sec:results_local_RL}.}
    \label{tab:results_paramOCP}
\end{table}

Since the legend of the figures is the same for all the tests, let us explain below the notations we use: 
\begin{itemize}
    \item $\operatorname{IBM}$ denotes an average trajectory (based on 50 simulations) for the IBM \emph{without} control,
    \item $\operatorname{IBM}^{C}$ denotes an average trajectory (based on 50 simulations) for the IBM \emph{with the final control} obtained by the reinforcement learning algorithm (individual trajectories are in grey),
    \item $\operatorname{RM}^{C}$ denotes a trajectory produced by the reduced model trained \emph{only} on the {initial} data set $\mathcal D_0$ with the {final} control of the algorithm,
    \item $\operatorname{RM}^{RLC}$ denotes a trajectory produced by the reduced model \emph{after} model-based reinforcement with the {final} control of the algorithm.
    \item $\hat b$ and $\hat k$ denote the final controls, piecewise constant, provided by the algorithm. The vertical dashed segments indicate times when changes in control values occur.
\end{itemize} 
    
The presentation of the results is divided into three parts. In the first one, we shed light on the outcome of simulations corresponding to several parameter configurations. Then, we focus on how the number of iterations of reinforcement affects the reduced model accuracy and the effectiveness of the associated control. Lastly, we turn our attention to the limitations and drawbacks of the proposed reinforcement learning approach.

\subsection{Examples of control dynamics}
\label{sec:results_local_RL_1}

We here present several results in different configurations. We recall that, in the optimal control problem associated with the IBM system, we only consider the cost on the infected peoples, and not the terms related to the control. 

To begin with, we consider two cases where the population size ratio and $\kappa$ are large (large population and homogeneous contact regime), meaning that they are in the validity regime of the classical SIR dynamics. We observe the results on this type of configuration on Figures \ref{fig:examp_RL_loc_1}-\ref{fig:examp_RL_loc_2}.
On the first one, we observe that the strong constraint on $I_{\mathrm{max}}$ is preserved and the number of infected stays close to $I_{\mathrm{hosp}}$. In this case the reduced model learned only with $\mathcal{D}_0$ (blue curve) 
 is pretty good. For the second case, we obtain similar results but see that the reinforcement step allows to increase the accuracy of the reduced model which subsequently improves the control efficiency.

\begin{figure}[ht!]
    \centering
    \includegraphics[width=\linewidth]{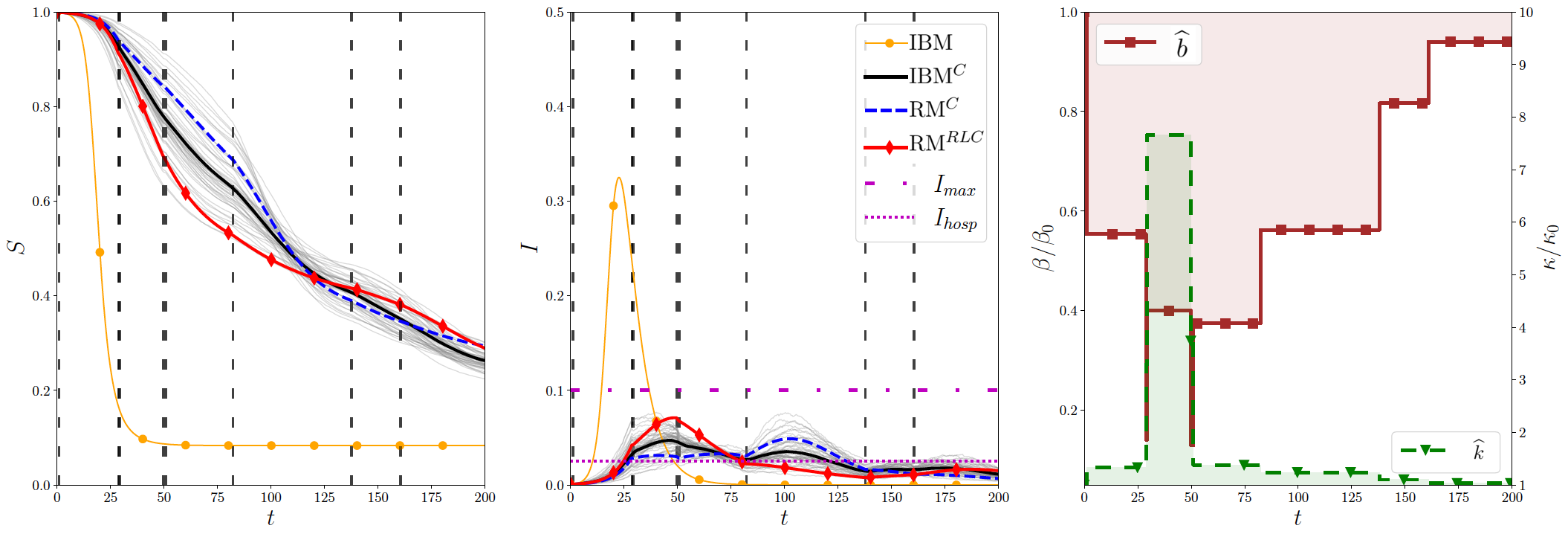}
    \caption[flushleft]{$n=0.95$, $\beta_0=0.5$, $\kappa_0=9$, $10$ iterations\label{fig:examp_RL_loc_1}}
\end{figure}

\begin{figure}[ht!]
    \centering
    \includegraphics[width=\linewidth]{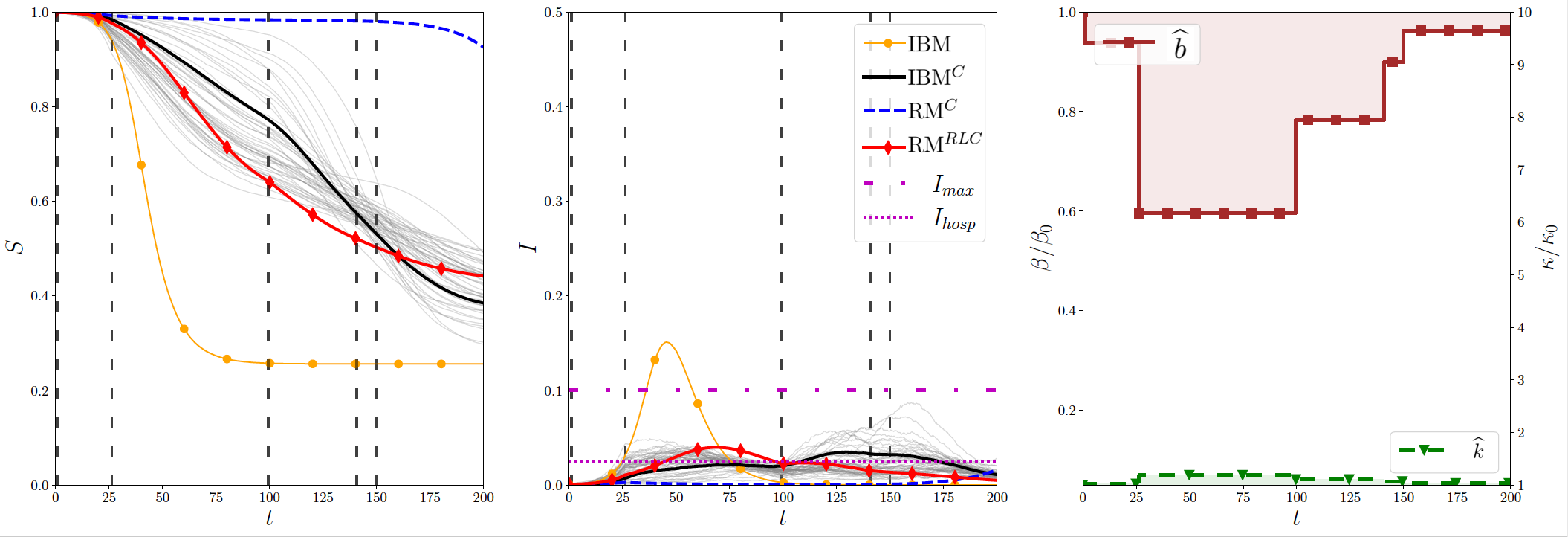}
    \caption[flushleft]{$n=0.85$, $\beta_0=0.8$, $\kappa_0=10$, $6$ iterations\label{fig:examp_RL_loc_2}} 
\end{figure}

In Figures \ref{fig:examp_RL_loc_3}-\ref{fig:examp_RL_loc_4}, we stray away from classical population-level regimes by considering intermediate population sizes and dispersion coefficients. In this slightly more complicated regime, stochastic behaviours are commonly observed in the IBM simulations. Nevertheless, in the first case (Figure \ref{fig:examp_RL_loc_3}), the reinforced reduced model faithfully approximates the IBM average trajectory and the control is effective enough to ensure that $I$ does not exceed the threshold $I_\mathrm{max}$. Similarly the control policy remains satisfying in the second case (Figure \ref{fig:examp_RL_loc_4}), although it is more difficult for the reduced model to capture the averaged random behaviour of the IBM due to the very low value of parameter $n$.

\begin{figure}[h!]
    \centering
    \includegraphics[width=\linewidth]{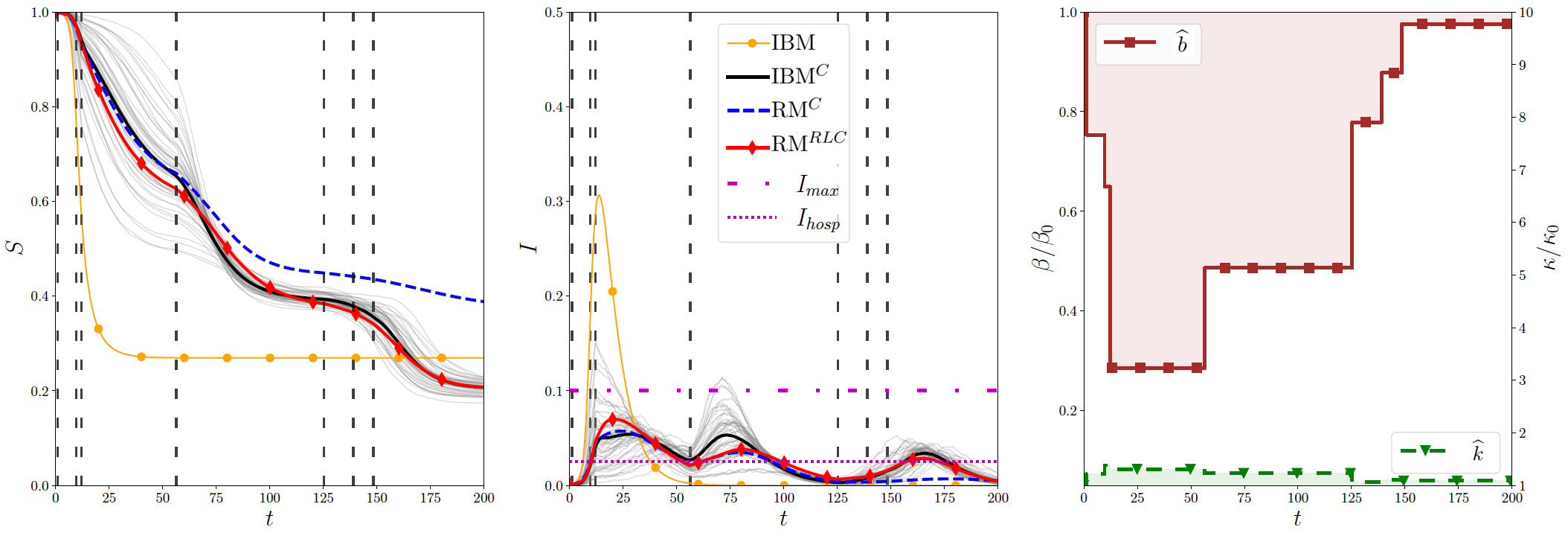}
    \caption[flushleft]{$n=0.6$, $\beta_0=0.5$, $\kappa_0=1$, $15$ iterations}
    \label{fig:examp_RL_loc_3}
\end{figure}

\begin{figure}[ht!]
    \centering
    \includegraphics[width=\linewidth]{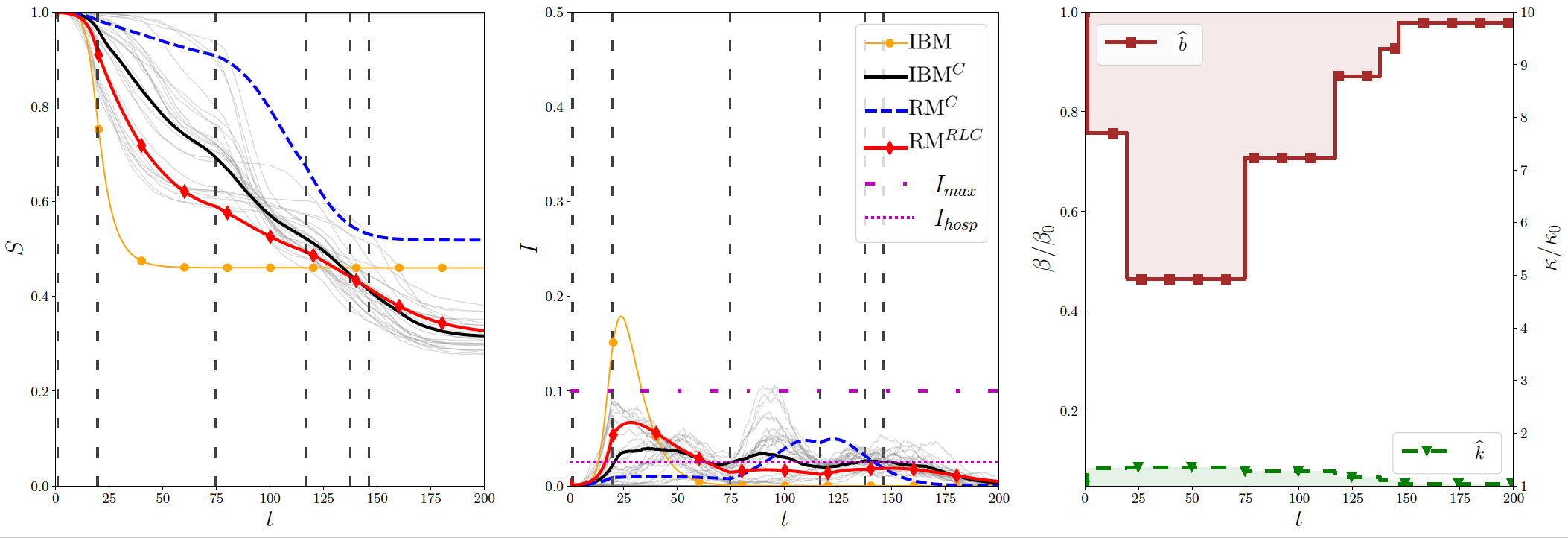}
    \caption[flushleft]{$n=0.2$, $\beta_0=0.3$, $\kappa_0=0.8$, $7$ iterations}
    \label{fig:examp_RL_loc_4}
\end{figure}

In Figure \ref{fig:examp_RL_loc_5}, we consider a case similar to the one investigated in Figure \ref{fig:examp_RL_loc_3} where weights in the cost function are modified. We strengthen the penalization on $\beta$, and weaken the ones on $\kappa$ and the infected population. We observe as expected that the control on $\kappa$ plays a more important role than previously, but it does not seem to significantly change the average behaviour of the IBM.

\begin{figure}[ht!]
    \centering
    \includegraphics[width=\linewidth]{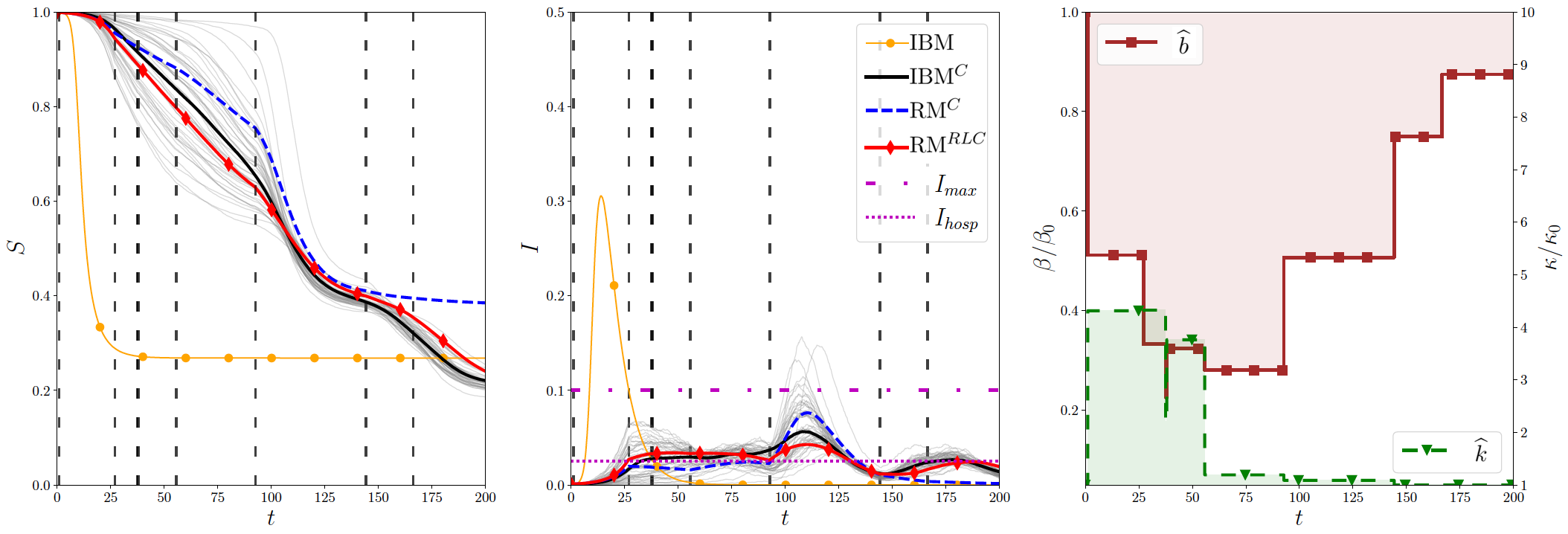}
    \caption[flushleft]{$n=0.6$, $\beta_0=0.5$, $\kappa_0=1$, $38$ iterations with $(\omega_\beta,\omega_\kappa,\omega_{\mathrm{hosp}})=(.5,.1,.4)$.}
    \label{fig:examp_RL_loc_5}
\end{figure}

In the following results, we increase one step further the difficulty by considering even lower population size ratios and dispersion coefficients. Figure \ref{fig:examp_RL_loc_6} deals with a very large dispersion effect (low $\kappa$). The resulting control is accurate and we remain far away from the strong constraint. We also observe that the reinforcement learning allows to improve a lot the reduced model which seems to be more and more faithful to the IBM trajectory, even if we observe a discrepancy between the final values of $S_{\infty}$. Indeed, since $S_\infty=1-\int_0^\infty\gamma I$, the accumulation of non-compensating errors on $I$ seems to lead to a poor estimate of $S_\infty$. Note that the cost functional in the optimal control problem does not involve $S_\infty$. In practice, this does not cause any concern because the main objective of the method is to compute an optimal control for IBM, in order to reduce the infection peaks. 

Figure \ref{fig:examp_RL_loc_7} highlights results in a very low population size regime where the graph and stochastic effects are important, yet the results are also convincing. It is notable that the reinforcement learning procedure drives to an improvement of the reduced model which in the end captures correctly the infection peaks.  
\begin{figure}[ht!]
    \centering
    \includegraphics[width=\linewidth]{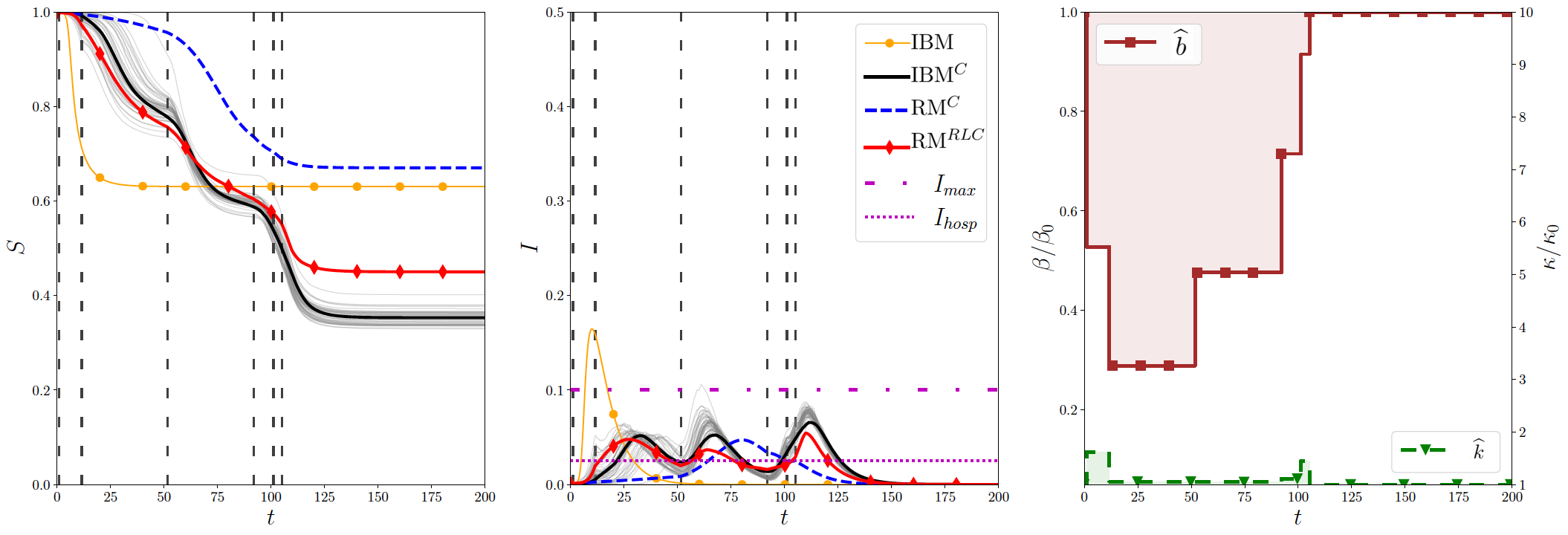}
    \caption[flushleft]{$n=0.8$, $\beta_0=0.3$, $\kappa_0=0.2$,  $44$ iterations with $(\omega_\beta,\omega_\kappa,\omega_{\mathrm{hosp}})=(.5,.1,.4)$.}
    \label{fig:examp_RL_loc_6}
\end{figure}
\begin{figure}[ht!]
    \centering
    \includegraphics[width=\linewidth]{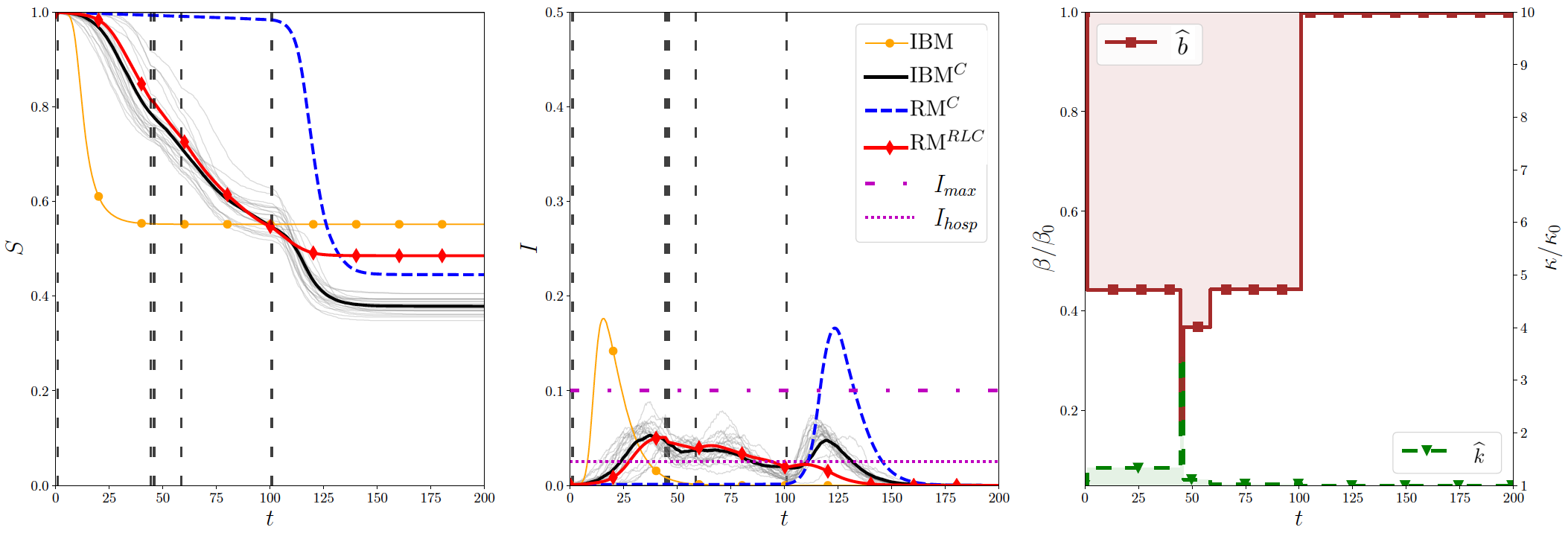}
    \caption[flushleft]{$n=0.2$, $\beta_0=0.3$, $\kappa_0=0.4$, $23$ iterations}
    \label{fig:examp_RL_loc_7}
\end{figure}

\subsection{Overall improvement with the number of iterations}
\label{sec:results_local_RL_2}

In this second subsection, we investigate how the number of reinforcement iterations affects the results of the algorithm. First, we consider a case with large dispersion in Figure \ref{fig:examp_RL_loc_8}. We observe that making some additional iterations increases a little bit the accuracy of the reinforcement learning reduced model and allows to compute a better control, since the infection peak is less close to the $I_{\mathrm{max}}$ constraint.

\begin{figure}[ht!]
    \centering
    \includegraphics[width=\linewidth]{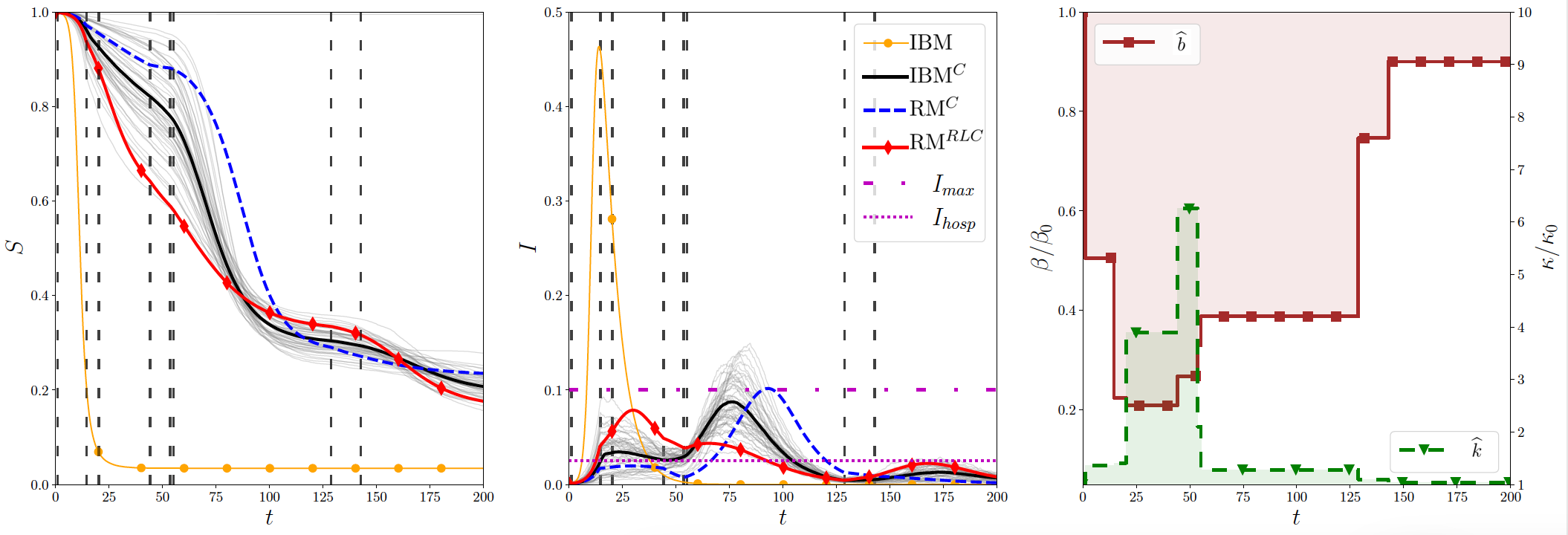}\\
    \includegraphics[width=\linewidth]{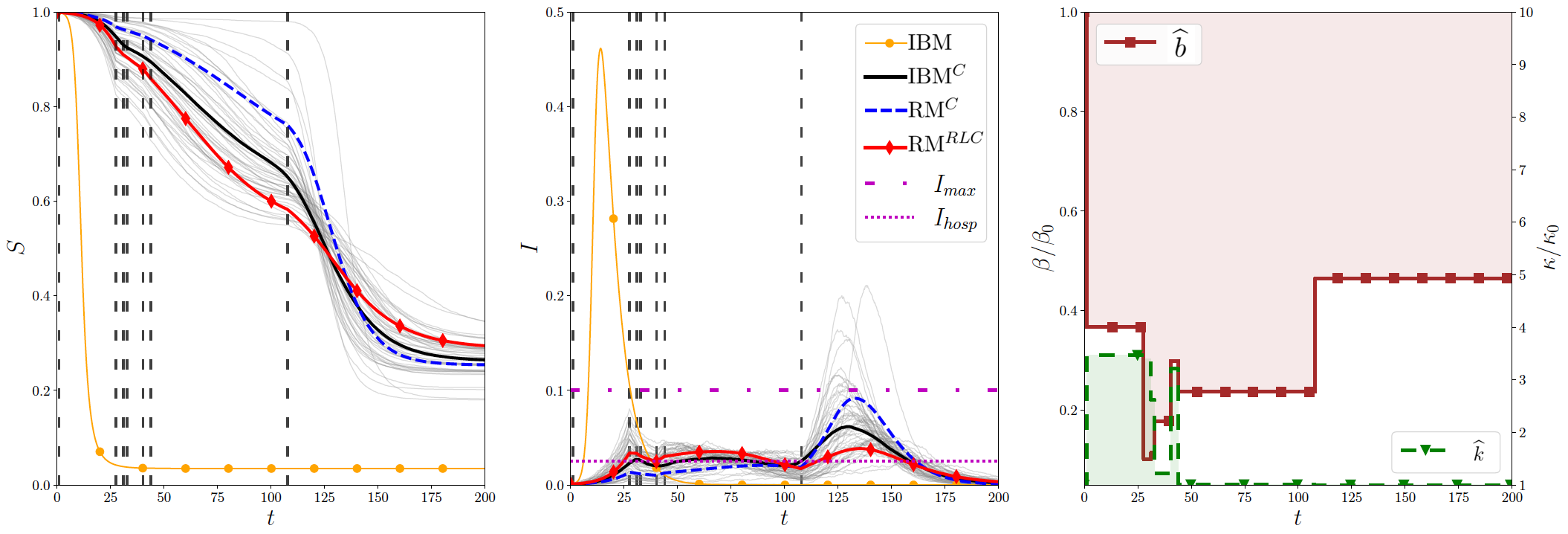}
    \caption[flushleft]{$n=0.5$, $\beta_0=0.8$, $\kappa_0=9$. $14$ iterations for the top and $30$ for the bottom.}
    \label{fig:examp_RL_loc_8}
\end{figure}

We now consider a highly dispersive test case (low $\kappa_0$) with a low $\beta_0$ but a large population size ratio, in Figure~\ref{fig:examp_RL_loc_10}. Since the epidemic is small, capturing the threshold is generally more complicated for the reduced model. Here, we compare one training after 18 and 34 iterations respectively, as well as a new training with 45 iterations and a smaller initial data set. As before we observe that increasing the number of iterations improves the accuracy of the reduced model and control. Indeed on the top of Figure~\ref{fig:examp_RL_loc_10}, the run after 18 iterations does not preserve the strong constraints, contrary to the second on the middle of Figure~\ref{fig:examp_RL_loc_10} which generates a trajectory satisfying the constraints. However, the reduced model is not flawless and previous tests show that this impacts the accuracy of the control. 

To improve control accuracy, we propose in that case to \emph{reduce the size} of the initial data set $\mathcal{D}_0$. This modification allows us to obtain a better reduced model and comparable control. A possible explanation is that the initial training may lead the neural network to learning a trajectory that deviates too far from the test case, making it difficult to explore the space of admissible trajectories. In other words, its ability to adapt to new samples may be impaired. This shows that the size of the data set $\mathcal{D}_0$ and its diversity may impact the efficiency of the algorithm. 
\begin{figure}[ht!]
    \centering
    \includegraphics[width=\linewidth]{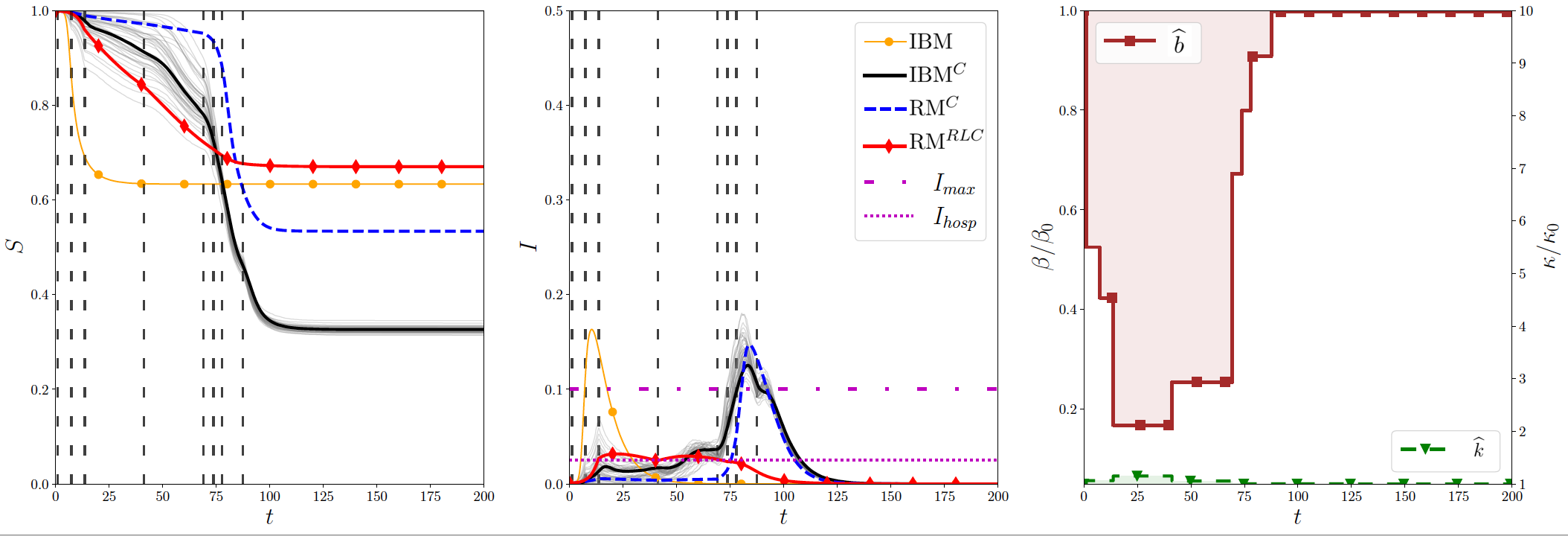}
    \includegraphics[width=\linewidth]{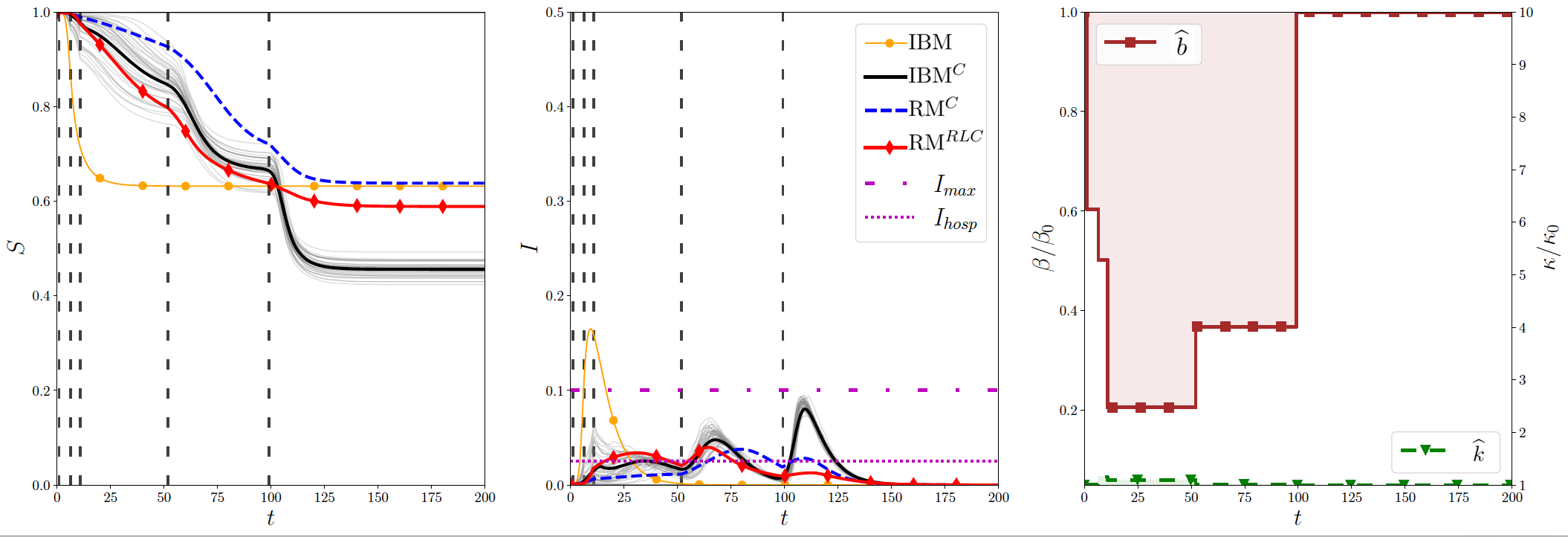}
    \includegraphics[width=\linewidth]{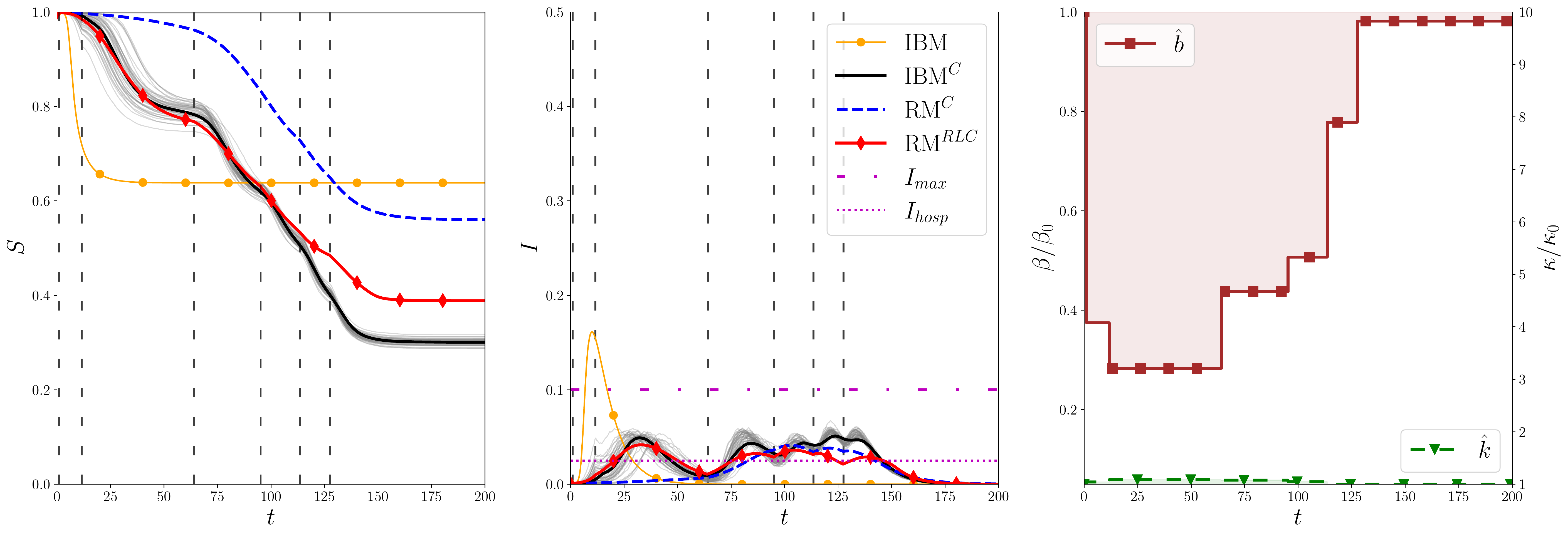}
    \caption[flushleft]{$n=0.8$, $\beta_0=0.3$, $\kappa_0=0.2$. $18$ iterations for the top, $34$ iterations for the middle and $17$ iterations with a (40\%) smaller data set $\mathcal{D}_0$ for the bottom.}
    \label{fig:examp_RL_loc_10}
\end{figure}

Figure~\ref{fig:examp_RL_loc_9} deals with a test case involving moderate dispersion and population size ratio. This example illustrates that usually, during the algorithm, the control improves while the reduced model may \emph{momentarily} worsen. Indeed, after only 8 iterations, the control fails to contain the epidemic (most stochastic trajectories of the IBM violate the strong constraint $I_\mathrm{max}$) even if the reduced model is qualitatively and quantitatively accurate. At the expense of model accuracy, making 8 additional iterations (middle plot) improves the control which now mitigates the peak of the average IBM trajectory, but not of all individual ones (grey trajectories). However, making twice as many iterations (bottom plot) leads to a control that is acceptable with tolerance $5\cdot10^{-3}$ (i.e. $c_p/c_0\leq 5\cdot10^{-3}$) and (again) to a faithful reduced model.

\begin{figure}[ht!]
    \centering
    \includegraphics[width=\linewidth]{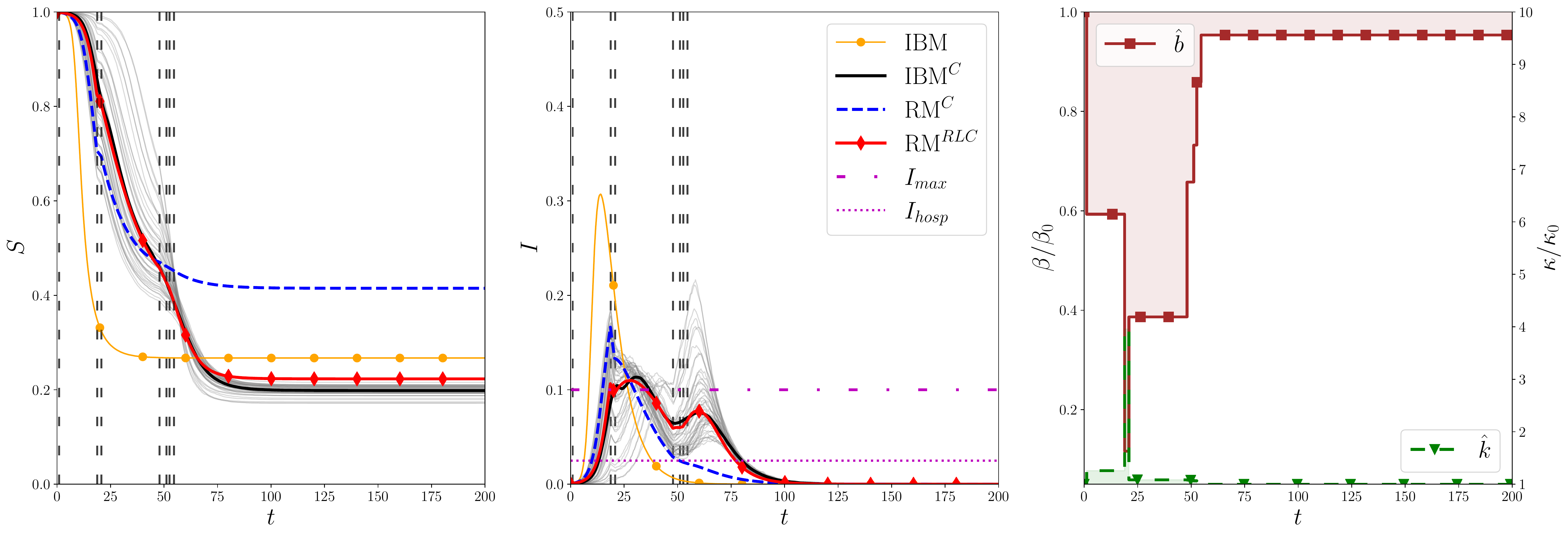}
    \includegraphics[width=\linewidth]{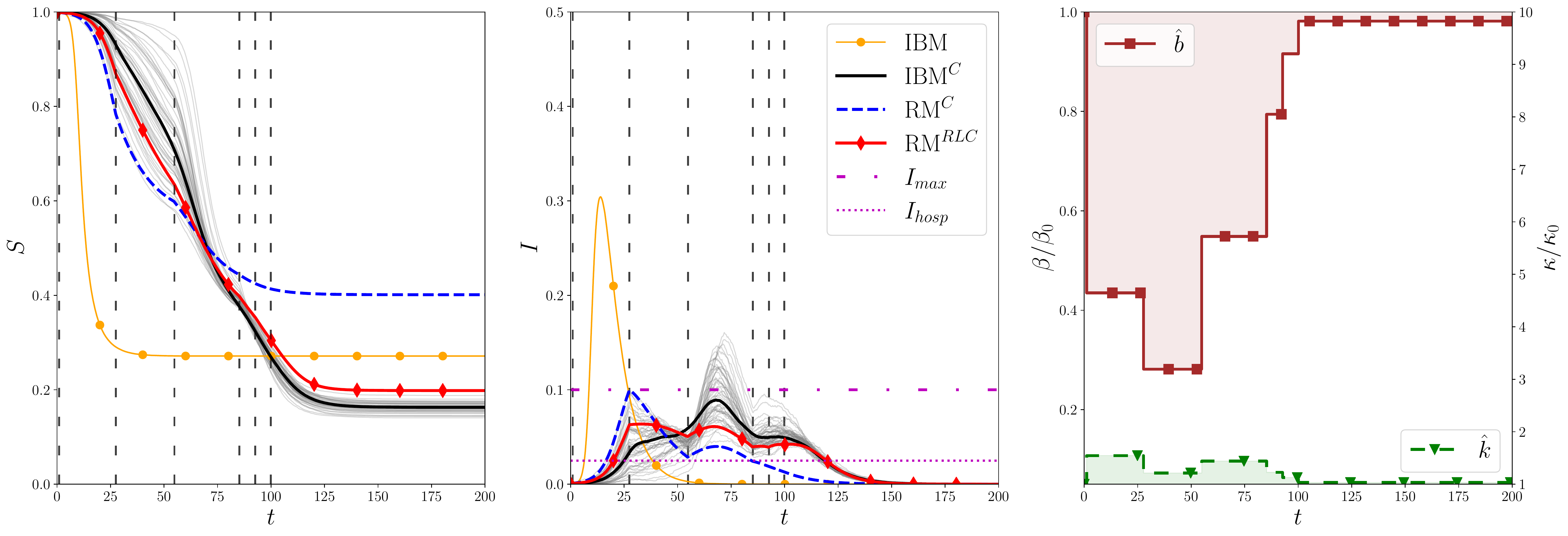}
    \includegraphics[width=\linewidth]{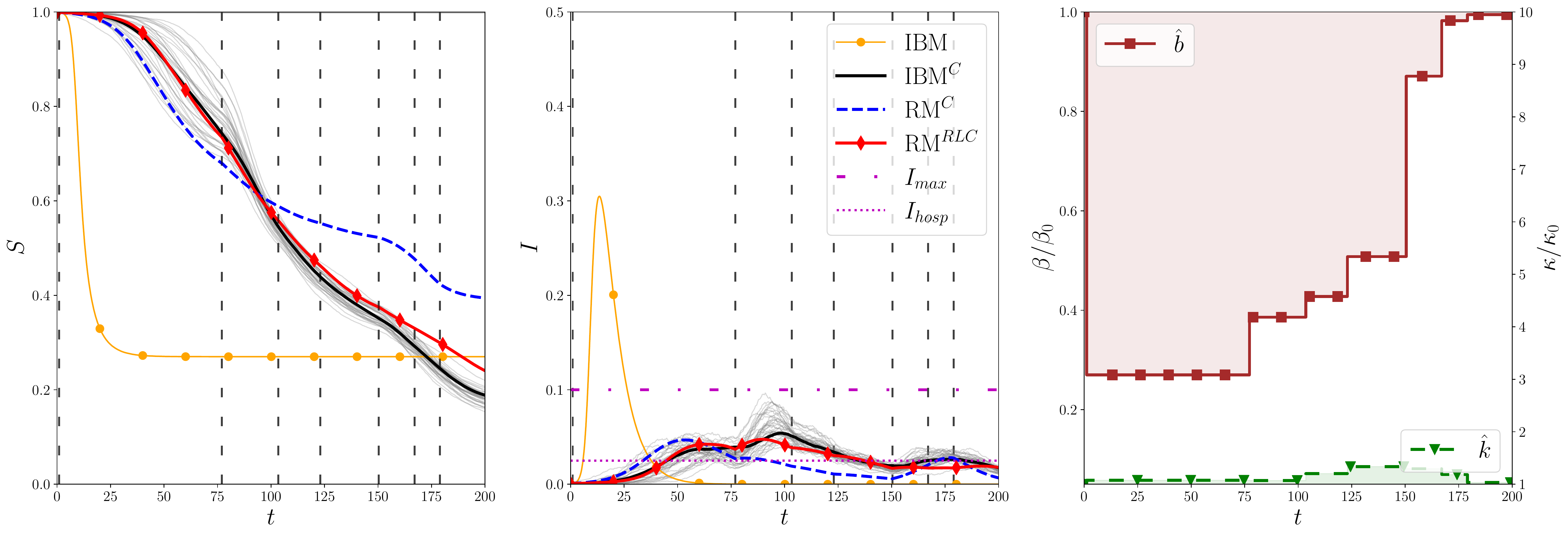}
    \caption[flushleft]{$n=0.6$, $\beta_0=0.5$, $\kappa_0=1$. Respectively $8$, $16$ and $35$ iterations for the top, middle and bottom.}
    \label{fig:examp_RL_loc_9}
\end{figure}

These results show that, whenever we make more iterations, the reduced model will generally overall become better. However, the improvement depends on the degree of randomness involved in the parameter regime at stake and it is not given that an increase in the accuracy of the model will lead to a better control. Indeed, we sometimes observe efficient controls associated with unfaithful reduced models. But, generally speaking, when the model becomes good, so does the control. That is, the "convergence" of the reduced model towards the IBM trajectory seems to guarantee that the corresponding control is accurate, although sometimes not the best.

\subsection{Fails and over-fitting}
\label{sec:results_local_RL_3}

 Our algorithm may fail to both generate an accurate reduced model and an acceptable control. We identified two possible mechanisms at the root of these seemingly rare failures. The first one relates to an \emph{over-fitting-like effect}. More precisely, throughout the reinforcement algorithm iterations, the neural network is generally fed with more and more similar samples. Thus, at a certain point, the network predictions are likely to deteriorate in an attempt to capture the dynamics associated with the mean IBM. For instance, it is seen in Figure \ref{fig:examp_RL_loc_11} that, after 18 iterations, the reduced model is very accurate even if the control is unsatisfactory. Hence, to improve the effectiveness of the control on the IBM, we try to make 7 additional reinforcement iterations. This strategy ends up paying off, but the improvement comes at the expense of a significant deterioration in the accuracy of the reduced model. Moreover, if we were to continue, the predictions of the reduced model might not improve. Considering early-stopping or a posteriori \emph{model selection among saved intermediate models} (by tracking performance criteria) may help overcoming these difficulties.

\begin{figure}[ht!]
    \centering
    \includegraphics[width=\linewidth]{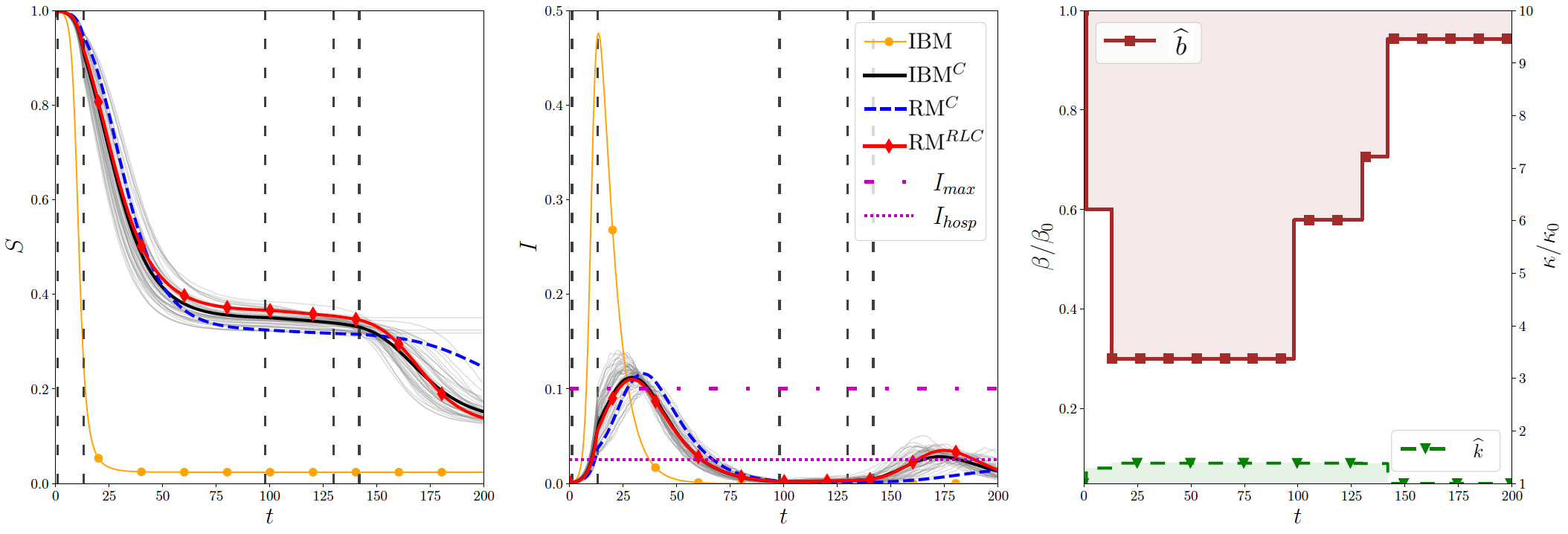}
     \includegraphics[width=\linewidth]{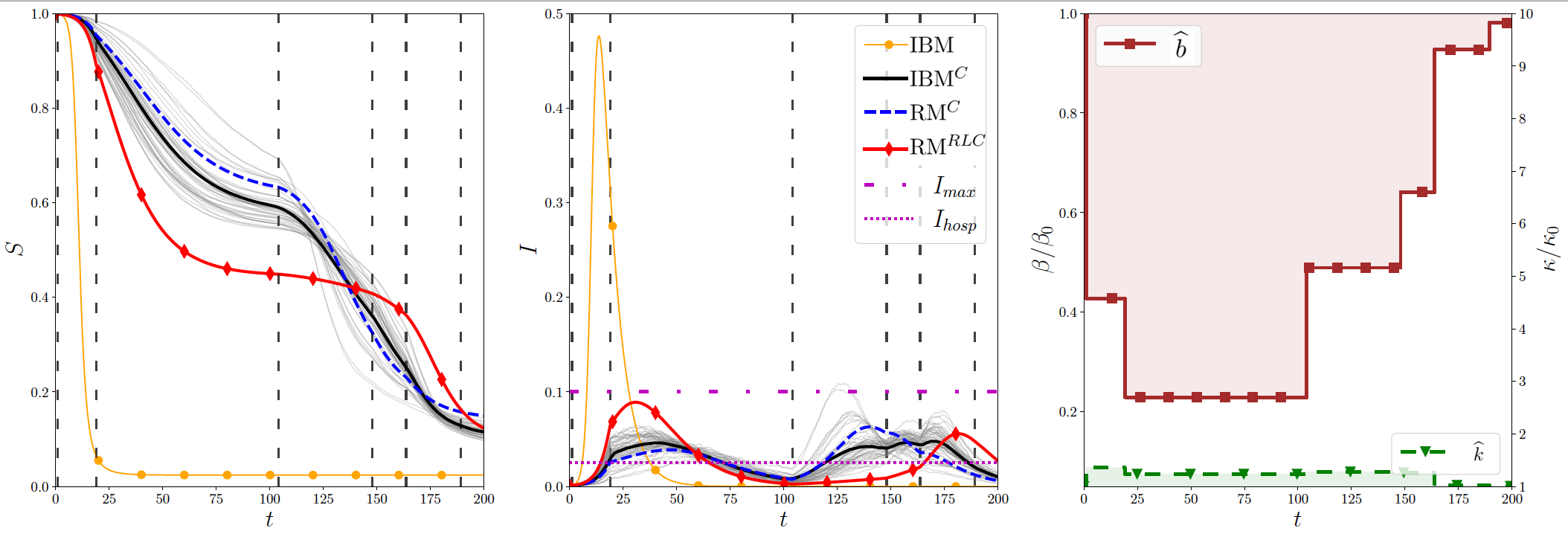}
    \caption[flushleft]{$n=0.95$, $\beta_0=0.8$, $\kappa_0=9$. Top: $18$ iterations. Bottom: $25$ iterations  }
    \label{fig:examp_RL_loc_11}
\end{figure}
The second mechanism we identified relates to the limitations of approximating a fundamentally stochastic system (the IBM) by means of a deterministic reduced model. In other words, when the parameter regime is such that randomness is dominating the behaviour of the system (e.g. all three parameters $n$, $\beta$ and $\kappa$ are low), approximating the average trajectory of the IBM is very challenging, as can be seen in Figure \ref{fig:examp_RL_loc_12}. Indeed, in this setting, computing the incidence function $F_\theta(S(t),I(t);n,\beta,\kappa)$ is particularly sensitive to errors. In addition, this range of parameters is scarcely represented in the initial data set $\mathcal{D}_0$ (see Table \ref{tab:param_range}). The latter could partially explain the difficulties observed in Figure \ref{fig:examp_RL_loc_12}. However, enriching $\mathcal{D}_0$ may not be recommended because the previous examples showed that, in a second step, the reduced model would probably have difficulty specializing around the controlled trajectory.

\begin{figure}[ht!]
    \centering
    \includegraphics[width=\linewidth]{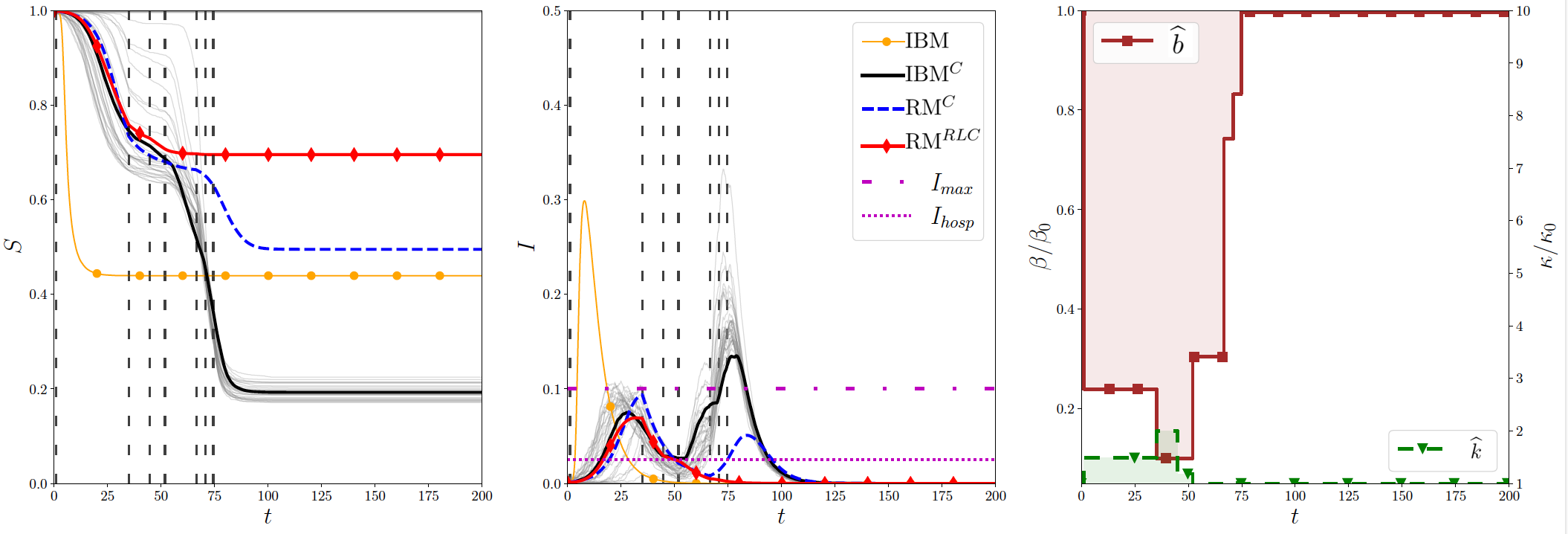}
    \caption[flushleft]{$n=0.15$, $\beta_0=0.7$, $\kappa_0=0.4$, $33$ iterations }
    \label{fig:examp_RL_loc_12}
\end{figure}

\section{Conclusion and perspectives}

In this work we have proposed a control method for an individual-based, stochastic epidemic model taking into account super-spreaders. For this purpose, we proposed a model-based reinforcement approach where we alternate between the learning phase of a reduced model and the control phase. Our approach can be interpreted as a Model Predictive Control (MPC) type method. In the literature on model-based methods, it is common to use global models for all states, or linear and local models around the current state. Here we propose a compromise by building a non-linear model valid for a certain sub-region of the possible values of the controls. Then the reduced model is solved by optimal control approaches for ODEs. The iterative algorithm allows us to build a control for the original IBM based on the one computed for the learned reduced model. The results show the ability of the algorithm to compute efficient controls in classical regimes (low dispersion, large population) as well as in more complicated regimes, as they are generally more stochastic, when the population size ratio is small and contact heterogeneity large. This algorithm involves building a reduced SIR model, relying on a neural network, which takes into account the effects of small population and dispersion effects associated with super-spreaders. Constructing the latter model also provides tools to study the effects of contact heterogeneity (dispersion) in epidemics. Indeed, since we learn a SIR-type model where the incidence function $F_{\theta}(S,I;n,\kappa,\beta)$ is differentiable, we could, for instance, derive by analytical means a formula for the basic reproduction ratio $\mathcal{R}_0$. Hence, it is possible to investigate the dependence of the latter on the dispersion and the population size, which is difficult to estimate for individual-based models. In a general way, building a reduced SIR-type model with a neural network from heavier simulations may be an interesting way to study phenomena that are not easily understood in large models such as the epidemic threshold, the group epidemic threshold, etc. One of the limitations of our approach is that we only aim at controlling the mean trajectory of the individual-based model, but it could be relevant and interesting to take into account the variance.

\section*{Acknowledgements}
The last author were partially supported by the ANR Project “TRECOS”.


\bibliographystyle{abbrv}
\bibliography{biblio.bib}

\begin{thebibliography}{10}

\bibitem{anOptimizationControlAgentBased2017}
G.~An, B.~G. Fitzpatrick, S.~Christley, P.~Federico, A.~Kanarek, R.~M. Neilan,
  M.~Oremland, R.~Salinas, R.~Laubenbacher, and S.~Lenhart.
\newblock Optimization and {{Control}} of {{Agent}}-{{Based Models}} in
  {{Biology}}: {{A Perspective}}.
\newblock {\em Bulletin of Mathematical Biology}, 79(1):63--87, 2017.

\bibitem{bhouriCOVID19DynamicsUS2021}
M.~A. Bhouri, F.~S. Costabal, H.~Wang, K.~Linka, M.~Peirlinck, E.~Kuhl, and
  P.~Perdikaris.
\newblock {{COVID}}-19 dynamics across the {{US}}: {{A}} deep learning study of
  human mobility and social behavior.
\newblock {\em Computer Methods in Applied Mechanics and Engineering},
  382:113891, 2021.

\bibitem{MR4198237}
P.-A. Bliman and M.~Duprez.
\newblock How best can finite-time social distancing reduce epidemic final
  size?
\newblock {\em J. Theoret. Biol.}, 511:Paper No. 110557, 12, 2021.

\bibitem{MR4255688}
P.-A. Bliman, M.~Duprez, Y.~Privat, and N.~Vauchelet.
\newblock Optimal immunity control and final size minimization by social
  distancing for the {SIR} epidemic model.
\newblock {\em J. Optim. Theory Appl.}, 189(2):408--436, 2021.

\bibitem{10.1613/jair.1.12632}
R.~Capobianco, V.~Kompella, J.~Ault, G.~Sharon, S.~Jong, S.~Fox, L.~Meyers,
  P.~R. Wurman, and P.~Stone.
\newblock {Agent-Based Markov Modeling for Improved COVID-19 Mitigation
  Policies}.
\newblock {\em J. Artif. Int. Res.}, 71:953–992, sep 2021.

\bibitem{MR0069338}
E.~A. Coddington and N.~Levinson.
\newblock {\em Theory of ordinary differential equations}.
\newblock McGraw-Hill Book Co., Inc., New York-Toronto-London, 1955.

\bibitem{conn2009introduction}
A.~R. Conn, K.~Scheinberg, and L.~N. Vicente.
\newblock {\em Introduction to derivative-free optimization}.
\newblock SIAM, 2009.

\bibitem{diekmannDefinitionComputationBasic1990}
O.~Diekmann, J.~Heesterbeek, and J.~Metz.
\newblock {On the Definition and the Computation of the Basic Reproduction
  Ratio R0 in Models for Infectious Diseases in Heterogeneous Populations}.
\newblock {\em Journal of Mathematical Biology}, 28(4), 1990.

\bibitem{MR1727362}
I.~Ekeland and R.~T\'{e}mam.
\newblock {\em Convex analysis and variational problems}, volume~28 of {\em
  Classics in Applied Mathematics}.
\newblock Society for Industrial and Applied Mathematics (SIAM), Philadelphia,
  PA, english edition, 1999.
\newblock Translated from the French.

\bibitem{franccoischolletandothersKeras2015}
{Fran\c cois Chollet and others}.
\newblock Keras, 2015.

\bibitem{fujieEffectsSuperspreadersSpread2007}
R.~Fujie and T.~Odagaki.
\newblock Effects of superspreaders in spread of epidemic.
\newblock {\em Physica A: Statistical Mechanics and its Applications},
  374(2):843--852, 2007.

\bibitem{garskeEffectSuperspreadingEpidemic2008}
T.~Garske and C.~Rhodes.
\newblock The effect of superspreading on epidemic outbreak size distributions.
\newblock {\em Journal of Theoretical Biology}, 253(2):228--237, 2008.

\bibitem{GILLESPIE1976403}
D.~T. Gillespie.
\newblock A general method for numerically simulating the stochastic time
  evolution of coupled chemical reactions.
\newblock {\em Journal of Computational Physics}, 22(4):403--434, 1976.

\bibitem{giustiMinimalSurfacesFunctions1984}
E.~Giusti.
\newblock {\em Minimal Surfaces and Functions of Bounded Variation}.
\newblock Number vol. 80 in Monographs in Mathematics. {Birkh{\"a}user}, 1984.

\bibitem{goodfellow2016deep}
I.~Goodfellow, Y.~Bengio, and A.~Courville.
\newblock {\em Deep learning}.
\newblock 2016.

\bibitem{hagbergaricandswartpieterandschultdanielExploringNetworkStructure2008}
{Hagberg, Aric and Swart, Pieter and S Chult, Daniel}.
\newblock Exploring network structure, dynamics, and function using
  {{NetworkX}}.
\newblock 2008.

\bibitem{MR3801235}
M.~Hinterm\"{u}ller, M.~Holler, and K.~Papafitsoros.
\newblock A function space framework for structural total variation
  regularization with applications in inverse problems.
\newblock {\em Inverse Problems}, 34(6):064002, 39, 2018.

\bibitem{kimAgentBasedModelingSuperSpreading2018}
Y.~Kim, H.~Ryu, and S.~Lee.
\newblock Agent-{{Based Modeling}} for {{Super}}-{{Spreading Events}}: {{A Case
  Study}} of {{MERS}}-{{CoV Transmission Dynamics}} in the {{Republic}} of
  {{Korea}}.
\newblock {\em International Journal of Environmental Research and Public
  Health}, 15(11):2369, 2018.

\bibitem{kissEffectContactHeterogeneity2006}
I.~Z. Kiss, D.~M. Green, and R.~R. Kao.
\newblock The effect of contact heterogeneity and multiple routes of
  transmission on final epidemic size.
\newblock {\em Mathematical Biosciences}, 203(1):124--136, 2006.

\bibitem{Kiss_2017}
I.~Z. Kiss, J.~C. Miller, and P.~L. Simon.
\newblock {\em Mathematics of Epidemics on Networks}.
\newblock Springer International Publishing, 2017.

\bibitem{koziel2013surrogate}
S.~Koziel and L.~Leifsson.
\newblock {\em Surrogate-based modeling and optimization}.
\newblock Springer, 2013.

\bibitem{LeeMarkus}
E.~B. Lee and L.~Markus.
\newblock {\em Foundations of optimal control theory}.
\newblock Wiley New York, 1967.

\bibitem{DBLP:journals/corr/abs-2004-12959}
C.~Liu.
\newblock A microscopic epidemic model and pandemic prediction using
  multi-agent reinforcement learning.
\newblock {\em CoRR}, abs/2004.12959, 2020.

\bibitem{lloyd-smithSuperspreadingEffectIndividual2005}
J.~O. Lloyd-Smith, S.~J. Schreiber, P.~E. Kopp, and W.~M. Getz.
\newblock Superspreading and the effect of individual variation on disease
  emergence.
\newblock {\em Nature}, 438(7066):355--359, 2005.

\bibitem{pmid:16292310}
J.~O. Lloyd-Smith, S.~J. Schreiber, P.~E. Kopp, and W.~M. Getz.
\newblock Superspreading and the effect of individual variation on disease
  emergence.
\newblock {\em Nature}, 438(7066):355--9, Nov 2005.

\bibitem{martchevaIntroductionMathematicalEpidemiology2015}
M.~Martcheva.
\newblock {\em An introduction to mathematical epidemiology}.
\newblock Number~61 in Texts in applied mathematics. {Springer}, 2015.

\bibitem{MR4363007}
S.~T. McQuade, R.~Weightman, N.~J. Merrill, A.~Yadav, E.~Tr\'{e}lat, S.~R.
  Allred, and B.~Piccoli.
\newblock Control of {COVID}-19 outbreak using an extended {SEIR} model.
\newblock {\em Math. Models Methods Appl. Sci.}, 31(12):2399--2424, 2021.

\bibitem{millerEoNEpidemicsNetworks2019}
J.~Miller and T.~Ting.
\newblock {{EoN}} ({{Epidemics}} on {{Networks}}): A fast, flexible {{Python}}
  package for simulation, analytic approximation, and analysis of epidemics on
  networks.
\newblock {\em Journal of Open Source Software}, 4(44):1731, 2019.

\bibitem{mkhatshwaModelingSuperspreadingEvents2010}
T.~Mkhatshwa and A.~Mummert.
\newblock Modeling {{Super}}-spreading {{Events}} for {{Infectious Diseases}}:
  {{Case Study SARS}}, 2010.

\bibitem{MolloyReed1998}
M.~Molloy and B.~Reed.
\newblock The size of the giant component of a random graph with a given degree
  sequence.
\newblock {\em Combinatorics, probability and computing}, 7(3):295--305, 1998.

\bibitem{pressNumericalRecipesArt1992}
W.~H. Press, editor.
\newblock {\em Numerical Recipes in {{C}}: The Art of Scientific Computing}.
\newblock {Cambridge University Press}, 2nd ed edition, 1992.

\bibitem{rafoSimpleEpidemicNetwork2020}
M.~d.~V. Rafo and J.~P. Aparicio.
\newblock Simple epidemic network model for highly heterogeneous populations.
\newblock {\em Journal of Theoretical Biology}, 486:110056, 2020.

\bibitem{ramirez2013x2+}
C.~Ramirez, R.~Sanchez, V.~Kreinovich, and M.~Argaez.
\newblock $\sqrt{x^2+ \mu}$ is the most computationally efficient smooth
  approximation to $|x|$: a proof.
\newblock 2013.

\bibitem{sutton2018reinforcement}
R.~S. Sutton and A.~G. Barto.
\newblock {\em Reinforcement learning: An introduction}.
\newblock MIT press, 2018.

\bibitem{vandendriesscheFurtherNotesBasic2008}
P.~van~den Driessche and J.~Watmough.
\newblock Further {{Notes}} on the {{Basic Reproduction Number}}.
\newblock In F.~Brauer, P.~van~den Driessche, and J.~Wu, editors, {\em
  Mathematical {{Epidemiology}}}, volume 1945 of {\em Lecture {{Notes}} in
  {{Mathematics}}}, pages 159--178. {Springer Berlin Heidelberg}, 2008.

\bibitem{zhangDiveDeepLearning2021}
A.~Zhang, Z.~C. Lipton, M.~Li, and A.~J. Smola.
\newblock Dive into {{Deep Learning}}, 2021.

\end{thebibliography}

\appendix
\renewcommand{\thesection}{\Alph{section}.\arabic{section}}
\setcounter{section}{0}

\begin{appendices}
\section{Averaging the IBM output}
\label{apdx:mean_traj_calculation}

Suppose we are interested in averaging $P\in \NN^*$ runs of the IBM over the time interval $[0, T]$ and let $M \geq 2$ be the number of points of a regular subdivision of this time interval with time-step $\Delta t.$ 

For each trajectory $p\in\{1,2,\ldots,P\}$, run the IBM and refer to the resulting discrete output as $S_p$ and $I_p$, both of which are elements of $\RR^M$. Their values at time $m\Delta t$ are respectively denoted $S_p^m$ and $I_p^m$ for $m\in\{1,2,\ldots,M\}$. Note that we have $R_p = 1 - S_p - I_p.$

 The $p$-th trajectory is considered an outlier whenever the size of recovered population ends up being underestimated in the following sense
$$ R_p^M - R_p^0 \leq  0.8 \times R_{\max}, $$
where $R_{\max}=\max \{R_p^M, \, p=1,\ldots,P\}$. In other words, all trajectories leading to immediate extinction are excluded since they would otherwise pull down the pointwise values of the mean trajectory. 

Next step is to find the average time of the first epidemic onset. For each $p$, let
\begin{equation}\label{eq:onset_time}\tau_p = \min \{ m\Delta t : I_p^m -I_p^0 > 10^{-3}, \, m=1,\ldots,M\}, \end{equation}
if the involved set is non-empty and zero otherwise. Based on these values, compute the mean time 
$$\Bar{\tau} = \frac{1}{P}\sum_{p=1}^P \tau_p,$$
and find for each trajectory $p$ the number of time-steps $d_p \in \ZZ$ by which the outbreak time is delayed (or in advance) with respect to the mean value $\Bar{\tau}$, that is
$$\forall p, \qquad  d_p = \left \lfloor \frac{\tau_p - \Bar{\tau}}{\Delta t}\right \rfloor.$$

Before averaging the trajectories, time-translate each trajectory $p$ by $d_p$ time-steps $\Delta t$ and denote $\Tilde{S}_p, \Tilde{I}_p$ the resulting vectors. To keep vectors of the same size, we extend the vector by constant values on the left or right depending on the sign of the translation:
\begin{align*}
    &\text{for }d_p > 0,&&\text{for }d_p < 0,&\\
    &\tilde{S}^m = \left\{\begin{array}{ccl}
				                                            S^{d_p+m}&& \mathrm{if}\; 1\leq m \leq M-d_p,\\
				                                            S^M      &&  \mathrm{if}\; M-d_p+1 \leq m \leq M,
		                                                    \end{array}\right.
	&&\Tilde{S}^m = \left\{\begin{array}{ccl}
				                                            S^{0}&& \mathrm{if}\; 1\leq m \leq |d_p|,\\
				                                            S^{|d_p|+m}      &&  \mathrm{if}\; |d_p|+1 \leq m \leq M.
		                                                    \end{array}\right.
\end{align*}
We therefore implicitly assume that the trajectories do not vary too much at the beginning and end of the simulations on a time scale $|d_p|\Delta t$. Lastly, we compute the point-wise average according to $$\forall m \in \{1,2,\ldots,M\}, \qquad \Bar{S}^m = \frac{1}{P}\sum_{p=1}^P \Tilde{S}^m_p,\ \Bar{I}^m=\frac{1}{P}\sum_{p=1}^P \Tilde{I}^m_p.$$

Note that in Definition~\eqref{eq:onset_time}, the threshold $10^{-3}$ offers a decent compromise. Indeed, the higher this value is, the more accurate the estimation of the family $(\tau_p)_p$ is, but at the same time, the higher the risk that some very stochastic trajectories reach the threshold much later (or earlier) than the others, resulting in a biased mean value $\Bar{\tau}.$ 

We wish to draw the reader's attention to the following observation: when we run the IBM with piece-wise constant parameters $\beta$ and $\kappa$, epidemic rebounds may occur several times in a given simulation, e.g. in Figures \ref{fig:pwcSubFig3} and \ref{fig:examp_RL_loc_7}. Nevertheless, since it is in the early stage of the epidemic that immediate extinctions are the most likely (due to stochasticity and very low proportions of infected people), translating the individual trajectories solely based on the time of the \emph{first} epidemic onset remains \textit{a priori} a reasonable assumption.

\section{More careful derivation of $\mathcal{R}_0$}
\label{apdx:rigorous_R0_derivation}
\paragraph{} The derivation proposed in Section \ref{sec:key-epidemio_quantities} is incomplete because one of the hypotheses required for applying the next-generation matrix theory is \emph{not} satisfied by the reduced model \eqref{eq:learned_model}. The fifth assumption stated in the paper of P. van den Driessche and J. Watmough  \cite{vandendriesscheFurtherNotesBasic2008} is lacking: in our case, it states that the ODE should have, provided no infected individuals ($I\equiv 0$), a \emph{unique asymptotically stable} equilibrium point, the so-called disease-free equilibrium (DFE). However, in the model \eqref{eq:learned_model}, there exists an infinite number of equilibrium points of the form $(S^*,0)$ for any $S^* \in \RR$ and none of them is asymptotically stable. Nevertheless, in order to fit to the theoretical framework, we can add \emph{demographic dynamics}, through birth and death rates, leading to a stabilizing population size ratio and such that the only disease-free equilibrium point is $(S^*,I^*)=(1,0)$. Given the uncertainty about the long term accuracy of the reduced model, the order of magnitude of the time horizons up to which the model is to be run is about 100 days. Moreover, since the demographic dynamics occur on a much large \emph{time scale} (years if not decades), this modelling assumption seems reasonable and will not strongly affect the dynamics arising from the reduced model \eqref{eq:learned_model}.

Therefore, we insert birth and death dynamics into the model:
\begin{align}
\begin{array}{c c c}
				S^\prime &=&-f_\theta(S,I;n,\beta,\kappa)SI + \mu - \mu S,\\
                I^\prime  &= & f_\theta(S,I;n,\beta,\kappa)SI-\gamma I - \mu I,
		\end{array}
		\label{eq:birth_death_learned_model}
\end{align}
where $\mu > 0$ stands for the population constant birth and death rates.  First, observe that $(S^*,I^*)=(1,0)$ is indeed the only disease-free state value making the dynamics of \eqref{eq:birth_death_learned_model} stationary.

Moreover, any solution with initial condition $(S_\mathrm{in}, 0)$, with $S_\mathrm{in} \in \RR$ converges to the unique equilibrium point $(1,0)$.

Then for any $\mu > 0$, the next-generation matrix theory, can be applied to model \eqref{eq:birth_death_learned_model}, leading to an expression of the threshold number, say $\mathcal{R}_0^\dagger=\mathcal{R}_0^\dagger(\mu)$. Indeed, let $\mathcal{F}$ denote the rate at which secondary infections increase in the infected compartment and $\mathcal{V}$ the sum of the rates at which the disease progresses and infected individuals die. We have that
\begin{align*}
I^\prime = \mathcal{F}(S,I)-\mathcal{V}(I), \qquad \qquad   \text{ with }\quad             \left\{\begin{array}{c c c}
				\mathcal{F}(S,I) &=& f_\theta(S,I;n,\beta,\kappa)SI,\\
                \mathcal{V}(I) &=& (\gamma  + \mu )I.
		\end{array}\right.
\end{align*}
Among the four other assumptions stated in \cite{vandendriesscheFurtherNotesBasic2008}, three of them are straightforward to verify. The last one states that the rate of secondary infections be positive or zero whenever susceptible or infected individuals remain ($\mathcal{F}(S,I)\geq 0$ for any $S,I \geq 0$). The latter was checked numerically for more than 10,000 randomly chosen different combinations of positive values. The requirement did not fail to hold and we can thus define the threshold number.

In the particular case of model \eqref{eq:birth_death_learned_model}, the next-generation matrix is actually a scalar and coincides with the threshold number:
\begin{equation*}
     \mathcal{R}_0^\dagger(\mu) = \left. \frac{\partial_I F_\theta }{\gamma+\mu}\right|_{(S=1,I=0,n,\beta,\kappa)}.
\end{equation*}
As $\mu \to 0$, we recover the expression of $\mathcal{R}_0$ given by Eq. \eqref{eq:R0_next_gen_matrix}.

\section{Properties of the controlled model}\label{apdx:well_posed_controlled_system}

\subsection*{Well-posedness and qualitative properties} 

It is notable that if $F_\theta$ is assumed to be locally Lipschitz with respect to $(S,I)$, continuous with respect to its other variables, then System~\eqref{eq:controlled_system} is well-posed according to the Carathéodory's existence theorem \cite[Theorem 1.1 of Chapter 2]{MR0069338}: it has a unique solution that belongs to $W^{1,\infty}(T_c,T;\RR^3)$.

Since we are interested in implementing an algorithm based on gradient like iterations, and in particular at deriving first order optimality conditions, we will further assume in what follows that $F_\theta$ satisfies \eqref{hypFtheta}.

The qualitative properties on $(S,I)$ follow from the uniqueness property above and the fact that $F_\theta$ is of the particular form \eqref{Ftheta_part}. Indeed, since $F_\theta(0,\cdot,\cdot,\cdot,\cdot)=F_\theta(\cdot,0,\cdot,\cdot,\cdot)=0$, the semi-axis $\{S=0, \ I\geq0\}$ and $\{S\geq 0,\ I=0\}$ correspond to particular orbits of System~\eqref{eq:controlled_system}. Therefore, a component of the solution $(S,I)$ associated with positive initial data cannot vanish. Furthermore, since initial data have been chosen in such a way that $S+I+R= 1$ at every time, and since $R$ is obviously non-negative, we infer that $\max \{S,I\}\leq 1-R\leq 1$ at every time. 

\subsection*{Analysis of the optimal control problem~\eqref{eq:OCPreg}}\label{sec:analysisOCP}
Before stating the first order optimality conditions for Problem~\eqref{eq:OCPreg}, let us first investigate existence properties for Problem~\eqref{eq:OCPreg}. 
\begin{lemma}
Let $\delta>0$. Problem~\eqref{eq:OCPreg} has a solution $(b_\delta,k_\delta)$.
\end{lemma}
\begin{proof}
Let $(b_p,k_p)_{p\in\NN}$ denote a minimizing sequence for Problem~\eqref{eq:OCPreg}. Since all terms of the cost functional are non-negative, the sequence $(\operatorname{TV}(b_p)+\operatorname{TV}(k_p)_{p\in\NN}$ is bounded. Since  $(b_p)_{p\in\NN}$ and  $(k_p)_{p\in\NN}$ are uniformly bounded in $L^\infty(T_c,T)$, it follows that $(\Vert (b_p,k_p)\Vert_{\operatorname{BV}(T_c,T)})_{p\in\NN}$ is bounded, and we infer that $(b_p,k_p)_{p\in\NN}$ converges up to a subsequence to some element $(b_\delta,k_\delta)\in \operatorname{BV}(T_c,T)$ in $L^1(T_c,T)$ and in particular pointwisely. In what follows, when there is no ambiguity, we will denote similarly a sequence and a subsequence with a slight abuse of notation. The pointwise convergence implies that $b_{\mathrm{min}}\leq b_\delta(\cdot) \leq 1$ and $1\leq k_\delta(\cdot) \leq k_{\mathrm{max}}$ a.e in $(T_c,T)$ which yields that $(b_\delta,k_\delta)$ belongs to $\mathcal{U}$. 

Let us denote by $(S_p,I_p,R_p)$ the solution to \eqref{eq:controlled_system} for the control choice $(b,k)=(b_p,k_p)$. Since $S_p+I_p+R_p$ is constant in $[T_c,T]$ and since $S_p$ and $I_p$ are non-negative because of the particular form of $F_\theta$ given by \eqref{Ftheta_part} and according to \eqref{hypFtheta}, it follows that $(S_p)_{p\in\NN}$ and $(I_p)_{p\in\NN}$ are uniformly bounded in $L^\infty(T_c,T)$. Since $f_\theta$ is assumed to be continuous with respect to each variable, it follows from \eqref{eq:controlled_system} that  $(S_p)_{p\in\NN}$ and $(I_p)_{p\in\NN}$ are uniformly bounded in $W^{1,\infty}(T_c,T)$. According to the Ascoli theorem, up to a subsequence, $(S_p)_{p\in\NN}$ and $(I_p)_{p\in\NN}$ converge in $C^0([T_c,T])$ to some element $(S_\delta,I_\delta)\in W^{1,\infty}(T_c,T)$.

Now, let us recast \eqref{eq:controlled_system} as
\begin{align}
\left\{\begin{array}{c c c}
				S_p(t) &=& S_c-\int_{T_c}^tF_\theta(S_p,I_p;n,\beta_0 b_pv(k_p),\kappa_0 k_p),\\
				I_p(t)  &= & I_c+\int_{T_c}^t F_\theta(S_p,I_p;n,\beta_0 b_pv(k_p),\kappa_0 k_p)-\gamma I_p,\\
		\end{array}\right.
		\label{eq:learned_model_seq}
\end{align}
for all $t\in [T_c,T]$. Since $F_\theta$ is assumed to be (at least) continuous with respect to any of its variable, passing to the limit in these equation follows straightforwardly from the Lebesgue dominated convergence theorem. We get
\begin{align}
\left\{\begin{array}{c c c}
				S_\delta(t) &=& S_c-\int_{T_c}^tF_\theta(S_\delta,I_\delta;n,\beta_0 b_\delta v(k_\delta),\kappa_0 k_\delta),\\
				I_\delta(t)  &= & I_c+\int_{T_c}^t F_\theta(S_\delta,I_\delta;n,\beta_0 b_\delta v(k_\delta),\kappa_0 k_\delta)-\gamma I_\delta,\\
		\end{array}\right.
\end{align}
for all $t\in [T_c,T]$, yielding that $(S_\delta,I_\delta)$ satisfies \eqref{eq:controlled_system}. We conclude by noting that, according to the Lebesgue dominated convergence theorem, one has
$$
\lim_{p\to +\infty}J(b_p,k_p)=J(b_\delta,k_\delta).
$$
By semicontinuity of the $\operatorname{TV}$ seminorm for the $L^1$ convergence, one has
$$
\operatorname{TV}[b_\delta]+\operatorname{TV}[k_\delta]\leq \liminf_{p\to +\infty}(\operatorname{TV}[b_p]+\operatorname{TV}[k_p]),
$$
leading to 
$$
J_\delta [b_\delta,k_\delta]\leq \liminf_{p\to +\infty}J_\delta [b_p,k_p] =\inf_{(b,k)\in \: \mathcal{U}}J_\delta[b,k].
$$
This concludes the proof.
\end{proof}

\subsection*{Computation of the differential of $J$ and first order optimality conditions}
\begin{proof}[Proof of Theorem~\ref{prop:diffJ}]
The first claim of the statement, related to the differentiability of $S$ and $I$ with respect to $[b,k]$ follows directly from the so-called Pontryagin maximum principle (see e.g. \cite{LeeMarkus}). The differentiability of $J$ is hence straightforward. Let $[b,k] \in \mathcal U$ and $[h_1,h_2]$ be an admissible perturbation. It remains to compute $  \langle dJ[b,k],[h_1,h_2]\rangle$. 

Let us introduce $\widehat{S}_1$, $\widehat{I}_1$ (resp. $\widehat{S}_2$, $\widehat{I}_2$) as the differentials of the mappings $b\mapsto S$, $b\mapsto I$ at $b$ in the direction $h_1$ (resp. the differentials of $k\mapsto S$, $k\mapsto I$ at $k$ in the direction $h_2$). In what follows, we will temporarily drop the variables $(S,I;n,\beta_0 bv(k),\kappa_0 k)$ in the quantities involving $F_\theta(S,I;n,\beta_0 bv(k),\kappa_0 k)$ in order to alleviate notations.

These functions solve the following ODE system:
\begin{equation}\label{metz:1815}
\frac{d}{dt}\begin{pmatrix}
\widehat{S}_1\\ \widehat{I}_1
\end{pmatrix}=M_\theta
\begin{pmatrix}
\widehat{S}_1\\ \widehat{I}_1
\end{pmatrix}+h_1\begin{pmatrix}
- \beta_0 v(k) \partial_4 F_\theta \\
 \beta_0 v(k) \partial_4 F_\theta 
\end{pmatrix}
\end{equation}
and
\begin{eqnarray}\label{metz:1816}
\frac{d}{dt}\begin{pmatrix}
\widehat{S}_2\\ \widehat{I}_2
\end{pmatrix}&=&M_\theta
\begin{pmatrix}
\widehat{S}_2\\ \widehat{I}_2
\end{pmatrix} +h_2 \begin{pmatrix}
- \beta_0bv'(k) \partial_4 F_\theta -\kappa_0 \partial_5 F_\theta \\
 \beta_0bv'(k) \partial_4 F_\theta +\kappa_0 \partial_5 F_\theta
\end{pmatrix}
\end{eqnarray}
completed with the initial conditions 
$$
\widehat{S}_i(T_c)=\widehat{I}_i(T_c)=0, \quad i=1,2,
$$
where $M_\theta$ is given by \eqref{def:mtheta}.

By using standard differentiation rules, one gets
 \begin{eqnarray*}
 \langle dJ[b,k],[h_1,h_2]\rangle &=&  \int_{T_c}^T \omega_\beta h_1  \left(b-1\right) 
                                      + \omega_\kappa h_2 \left(k-1\right) \\
                                      && + \int_{T_c}^T\frac{\omega_{\mathrm{hosp}}}{{I_{\mathrm{hosp}}}} (\widehat{I}_1+\widehat{I}_2)\left(\frac{I}{I_{\mathrm{hosp}}}-1\right)_+ 
                                      + \frac{1}{ I_{\mathrm{max}} \varepsilon} (\widehat{I}_1+\widehat{I}_2) \left(\frac{I}{I_{\mathrm{max}}}-1\right)_+.
 \end{eqnarray*}
Now, let us multiply System~\eqref{metz:1815} by $(p_1,q_1)$ and  System~\eqref{metz:1816} by $(p_2,q_2)$ in the sense of the inner product. By integrating by parts, one gets successively
$$
\left.\begin{pmatrix}
p_1\\ q_1
\end{pmatrix}\cdot \begin{pmatrix}
\widehat{S}_1\\ \widehat{I}_1
\end{pmatrix}\right|_{t=T}+
\int_{T_c}^T\left( -\frac{d}{dt}\begin{pmatrix}
p_1\\ q_1
\end{pmatrix}-M_\theta^\top \begin{pmatrix}
p_1\\ q_1
\end{pmatrix}\right)\cdot \begin{pmatrix}
\widehat{S}_1\\ \widehat{I}_1
\end{pmatrix}=\int_{T_c}^T h_1\begin{pmatrix}
p_1\\ q_1
\end{pmatrix}\cdot \begin{pmatrix}
- \beta_0 v(k) \partial_4 F_\theta \\
 \beta_0 v(k) \partial_4 F_\theta 
\end{pmatrix}
$$
and
$$
\left.\begin{pmatrix}
p_2\\ q_2
\end{pmatrix}\cdot \begin{pmatrix}
\widehat{S}_2\\ \widehat{I}_2
\end{pmatrix}\right|_{t=T}+
\int_{T_c}^T\left( -\frac{d}{dt}\begin{pmatrix}
p_2\\ q_2
\end{pmatrix}-M_\theta^\top \begin{pmatrix}
p_2\\ q_2
\end{pmatrix}\right)\cdot \begin{pmatrix}
\widehat{S}_2\\ \widehat{I}_2
\end{pmatrix}=\int_{T_c}^T h_2\begin{pmatrix}
p_2\\ q_2
\end{pmatrix}\cdot \begin{pmatrix}
- \beta_0bv'(k) v(k) \partial_4 F_\theta -\kappa_0 \partial_5 F_\theta \\
 \beta_0bv'(k) v(k) \partial_4 F_\theta +\kappa_0 \partial_5 F_\theta
\end{pmatrix}.
$$
Using that $(p_1,q_1,p_2,q_2)$ solves the linear system \eqref{eq:adjoint} yields 
\begin{eqnarray*} 
\lefteqn{
\int_{T_c}^T\frac{\omega_{\mathrm{hosp}}}{{I_{\mathrm{hosp}}}} (\widehat{I}_1+\widehat{I}_2)\left(\frac{I}{I_{\mathrm{hosp}}}-1\right)_+ 
                                      + \frac{1}{ I_{\mathrm{max}} \varepsilon} (\widehat{I}_1+\widehat{I}_2) \left(\frac{I}{I_{\mathrm{max}}}-1\right)_+=}\\
&& \int_{T_c}^T h_1\begin{pmatrix}
p_1\\ q_1
\end{pmatrix}\cdot \begin{pmatrix}
- \beta_0 v(k) \partial_4 F_\theta \\
 \beta_0 v(k) \partial_4 F_\theta 
\end{pmatrix}+\int_{T_c}^T h_2\begin{pmatrix}
p_2\\ q_2
\end{pmatrix}\cdot \begin{pmatrix}
- \beta_0bv'(k)  \partial_4 F_\theta -\kappa_0 \partial_5 F_\theta \\
 \beta_0bv'(k)   \partial_4 F_\theta +\kappa_0 \partial_5 F_\theta
\end{pmatrix},
\end{eqnarray*}
whence the desired expression of the differential. 

Let us now prove the last statement of this theorem. Let $[b,k]$ denote a solution to Problem~\eqref{eq:OCPreg}. Let us introduce the so-called indicator function $\iota_\mathcal{U}$ to the set $\mathcal{U}$, given by
$$
\iota_\mathcal{U}(x)=\left\{\begin{array}{ll}
0 & \text{if }x\in \mathcal{U}\\
+\infty & \text{else}.\end{array}\right.
$$
The functional $J_\delta$ is not differentiable in a standard sense because of the $\operatorname{TV}$ terms. For this reason, we will use subdifferentials to derive first order optimality conditions. We first claim that Problem~\eqref{eq:OCPreg} is equivalent to the optimization problem
$$
\inf_{(b,k)\in L^\infty(T_c,T)}J[b,k]+\delta (\operatorname{TV}[b]+\operatorname{TV}[k])+\iota_\mathcal{U}((b,k)).
$$
The standard first order optimality condition reads
$$
0\in \partial \left(J[b,k]+\delta (\operatorname{TV}[b]+\operatorname{TV}[k])+\iota_\mathcal{U}((b,k))\right).
$$
which rewrites
$$
\left\{\begin{array}{l}
-\partial_b J\in \partial\operatorname{TV}(b)+\partial \iota_{[ b_{\mathrm{min}},1]}\\
-\partial_k J\in \partial\operatorname{TV}(k)+\partial \iota_{[ 1,k_{\mathrm{max}}]}\\
\end{array}
\right.
$$
and therefore, there exist $T_b\in \partial \operatorname{TV}(b)$ and $T_k\in \partial \operatorname{TV}(k)$ such that the Euler inequation
$$
\forall (B,K)\in L^\infty(T_c,T;[ b_{\mathrm{min}},1])\times L^\infty(T_c,T;[ 1,k_{\mathrm{max}}]), \quad 
\left\{\begin{array}{l}
\langle \partial_b J-T_b,B-b\rangle_{L^2(T_c,T)} \geq 0\\
\langle \partial_k J-T_k,K-k\rangle_{L^2(T_c,T)} \geq 0\\
\end{array}
\right.
$$
holds true and the expected conclusion follows.
Note that, since it is not our main purpose, we do not provide details in this article but refer for instance to \cite{MR3801235} for explicit characterizations of such sets.
\end{proof}

\section{Numerical implementation}\label{apdx:numerical_implementation}
Since the adjoint system is state-dependent, given an initial estimate of the control, we first determine the set of $\{(S,I)(t),\ t\in \mathcal{S}\}$ (e.g., using an explicit fourth-order Runge-Kutta numerical scheme), and then $\{(p_1,q_1,p_2,q_2)(t),\ t\in \mathcal{S}\}$.

The regularization term appearing in the problem \eqref{eq:OCPreg} is computed using the following approximation of the total variation (see \cite[§1.30]{giustiMinimalSurfacesFunctions1984}:
$$ 
\operatorname{TV}[q] \simeq \sum _m\left|q(t^m)-q(t^{m-1})\right|, 
$$
where the family of points $(t^m)$ belongs to $\mathcal{S}$. Note that, in the expression of the total variation, the partition used for the calculation is not arbitrary.

In addition, notice that the resulting regularisation term involves the non-differentiable absolute value function. Since this term contributes to the gradient of the cost functional $J_\delta$, we use the following computationally efficient smooth approximation \cite{ramirez2013x2+}:
$|x|\simeq \sqrt{\eta + x^2}$, for all $x\in \RR$,
where $\eta$ is taken equal to $10^{-6}$. 

A concise description of the optimal projected gradient approach is given in Algorithm~\ref{alg:gradient} where $u\in \RR^{2\times |\mathcal{S}|}$ denotes\footnote{$|\mathcal{S}|$ denotes the cardinality of the set $\mathcal{S}$.} the current (discrete) control approximation to the vector-valued solution of the problem \eqref{eq:OCPreg}. In addition, the input $u_0$ is an initial guess for the control values whereas the parameters $N_\mathrm{g}$ and $\tau_\mathrm{g}$ are respectively the maximum number of iterations of the algorithm and a tolerance. Lastly, $\mathcal{P}_\mathcal{U}$ refers to the map which projects the discrete components $ b\in \RR^{1\times |\mathcal{S}|}$ and $ k\in \RR^{1\times |\mathcal{S}|}$ of $u$ according to
$$ \mathcal{P}_\mathcal{U} : u=\begin{pmatrix}
b_1, b_2, \ldots, b_{|\mathcal{S}|} \\
k_1, k_2, \ldots, k_{|\mathcal{S}|}
\end{pmatrix} \mapsto \begin{pmatrix}
P_b(b_1), \ldots, P_b(b_{|\mathcal{S}|}) \\ P_k(k_1), \ldots, P_k(k_{|\mathcal{S}|})
\end{pmatrix},$$
where $P_b : x \mapsto \min\left(1,\max \left(b_{\mathrm{min}}, x\right) \right)$ and $P_k:  x \mapsto \min\left(k_\mathrm{max},\max \left(1, x\right) \right)$ are defined on $\RR$.

\begin{algorithm}[h!]
\caption{Optimal step size projected gradient} \label{alg:gradient}
\KwRequire{Partition $\mathcal{S}$ of $[T_c, T].$}
\KwInput{$u_0\in \RR^{2\times |\mathcal{S}|}$, $N_\mathrm{g}\in \NN^*$, $\tau_\mathrm{g}> 0$.}
    \vspace{5pt}
      \begin{flushleft}
   $\begin{matrix}
       p &\xleftarrow{}& 0,\\
       u &\xleftarrow{}& u_0, \\
       j_\mathrm{old} &\xleftarrow{}& +\infty,\\
       j_\mathrm{new}, j_0 &\xleftarrow{}& J_\delta[u_0], \\
       \nabla j &\xleftarrow{}& \nabla J_\delta[u_0].
  \end{matrix}$
    \end{flushleft}
   \vspace{5pt}

 \While{$p < N_\mathrm{g}\; \mathrm{and} \;(j_\mathrm{old} - j_\mathrm{new}) > \tau_\mathrm{g} j_0 $}{
  \vspace{5pt}
  Golden-section search to find $\rho^*$ a local minimizer of $\rho \mapsto J_\delta \left[\mathcal{P}_\mathcal{U}\left(u - \rho \nabla j\right)\right],$\\
  $\mathrm{Update}~:$\\
  \begin{flushleft}
  $\begin{matrix}
       u &\xleftarrow{}& \mathcal{P}_\mathcal{U}\left(u - \rho^* \nabla j\right), \\
       j_\mathrm{old} &\xleftarrow{}& j_\mathrm{new},\\
       j_\mathrm{new} &\xleftarrow{}& J_\delta[u], \\
       \nabla j &\xleftarrow{}& \nabla J_\delta[u].
  \end{matrix}$
  \end{flushleft}
  \vspace{5pt}
  
  \If{$\nabla j \neq 0$}{
  \vspace{5pt}
   $\mathrm{Normalise}\;  \nabla j \xleftarrow{} \nabla j / \left||\nabla j|\right|,$
   }
   $p \xleftarrow{} p+1$.
 }
 \vspace{5pt}
 \KwReturn{$u$.}
\end{algorithm}
\end{appendices}

\end{document}